\documentclass[11pt]{article}

\usepackage{amsfonts}
\usepackage{amscd}
\usepackage{amssymb}
\usepackage{amsthm}
\usepackage{amsmath}
\usepackage{stmaryrd}
\usepackage{graphicx}
\usepackage{color}
\usepackage{verbatim}
\usepackage{mathrsfs}
\usepackage{dsfont}
\usepackage{tikz}
\usepackage{subfigure}

 \theoremstyle{plain}
\newtheorem{thm}{Theorem}[section]
\newtheorem{lemma}[thm]{Lemma}
\newtheorem{prop}[thm]{Proposition}
\newtheorem{cor}[thm]{Corollary}

\theoremstyle{definition}
\newtheorem{defn}[thm]{Definition}
\newtheorem{remark}[thm]{Remark}

\numberwithin{equation}{section}

\def\bI{\,\mathbf{I}\,}

\def\cA{\mathcal{A}}

\def\cL{\mathcal{L}}

\def \cP{\mathcal{P}}

\def\cR{\mathcal{R}}
\def\cS{\mathcal{S}}

\def\tr{\mathrm{tr}}

\makeatletter
\renewcommand{\@makefnmark}{\mbox{\textsuperscript{}}}
\makeatother

\title{The combinatorics of automorphisms and opposition in generalised polygons}
\author{James Parkinson\footnote{Research supported under the Australian Research Council (ARC) discovery grant DP110103205.} \and Beukje Temmermans\and Hendrik Van Maldeghem}
\date{September 24, 2013}

\begin{document}

\maketitle

\vspace{-0.2cm}

\begin{abstract} We investigate the combinatorial interplay between automorphisms and opposition in (primarily finite) generalised polygons. We provide restrictions on the fixed element structures of automorphisms of a generalised polygon mapping no chamber to an opposite chamber. Furthermore, we give a complete classification of automorphisms of finite generalised polygons which map at least one point and at least one line to an opposite, but map no chamber to an opposite chamber.  Finally, we show that no automorphism of a finite thick generalised polygon maps all chambers to opposite chambers, except possibly in the case of generalised quadrangles with coprime parameters.
\end{abstract}

\section{Introduction}
The ``opposition relation'' is fundamental in Tits' theory of spherical and twin buildings. Roughly speaking, two elements are opposite if they are at maximal distance from one another. An incredible amount of information is encoded in this relation. For instance, given two opposite elements of a spherical building, the geometry of elements incident with the first element is isomorphic to the geometry of elements incident with the second element, with the isomorphism given by the relation of ``not being opposite''. Using this observation Tits \cite{tit:74} was able to classify all spherical buildings as soon as these geometries are rich enough, and the latter is the case whenever each direct factor of the building has rank at least~$3$. Consequently the irreducible spherical buildings of rank at least~3 are essentially equivalent to simple linear algebraic groups of relative rank at least~3, or simple classical linear groups, or certain related groups called ``groups of mixed type''. 

Tits' classification cannot be boldly extended to the rank 2 case (where the building is a \textit{generalised polygon}), as there are many generalised polygons with very different automorphism groups. Tits \& Weiss \cite{TW:02} have classified all generalised polygons satisfying the additional so-called \emph{Moufang condition}. But many non-Moufang examples exist. This is particularly interesting in the finite case, since the main conjectures and problems about finite generalised polygons were stated almost 30 years ago and still remain unresolved, see \cite{kan:84}. The most exciting among these conjectures and problems relate to automorphism groups. For example, the classification of the flag-transitive finite polygons, or the classification of finite polygons with a collineation group acting sharply transitive on the point set, are important open problems.

One of the aims of the present paper is to provide restrictions on how an arbitrary automorphism of a generalised polygon can act, particularly with respect to the opposition relation. The case of generalised quadrangles is analysed in~\cite{TTM:13}, and by the Feit-Higman Theorem~\cite{FH:64} the remaining finite thick generalised polygons are the digons, projective planes, generalised hexagons and generalised octagons. Digons are trivial objects, and projective planes turn out to be well behaved for our purposes. Thus the bulk of this paper deals with the more complicated cases of generalised hexagons and generalised octagons.

Our starting point is a result of Leeb \cite{lee:00} stating that every nontrivial automorphism of a thick spherical building must map at least one residue to an opposite residue (see also Abramenko \& Brown~\cite{AB:09}, and note that this result has been extended to twin buildings; see~\cite{GHM:11} for the case of involutions, and~\cite{DPM:13} for the general case). By far the most ``normal'' behaviour is that there is a chamber (that is, a maximal residue) mapped to an opposite chamber. Our aim is to describe, as precisely as possible, the conditions under which the ``abnormal'' situation where no chamber is mapped to an opposite chamber can occur. In this case the automorphism is called~\emph{domestic}, and recent work suggests that domesticity is intimately related to interesting subconfigurations  of spherical buildings. For example the domestic dualities in odd-dimensional projective spaces are the symplectic polarities, and these fix a large polar space (see~\cite{TTM:12}). The domestic dualities in buildings of type $\mathsf{E_6}$ are the polarities that fix a split building of type $\mathsf{F_4}$ (see~\cite{mal:12}), and the domestic trialities in buildings of type $\mathsf{D_4}$ are the trialities of order 3 of type $\mathsf{I}_{\mathrm{id}}$ fixing a split Cayley generalised hexagon (see \cite{mal:13}). These  examples make one believe that domestic automorphisms are rather well-behaved. Making this explicit, all of the above examples have the following property: The residues which are maximal subject to the condition of being mapped onto an opposite residue all have the same type. If we call a domestic automorphism which does \emph{not} satisfy this property \emph{exceptional domestic}, then no exceptional domestic automorphism for a spherical building of rank at least 3 is known. 

For an automorphism of a generalised polygon, being exceptional domestic is plainly equivalent to mapping at least one point to an opposite, and at least one line to an opposite, yet mapping no chamber to an opposite. In contrast to the higher rank situation, for thick generalised quadrangles it is shown in \cite{TTM:13} that there are precisely three exceptional domestic automorphisms (up to duality), and they only occur in the small quadrangles of order $(2,2)$, $(2,4)$ and $(3,5)$.  They all turn out to have order~4, which is in itself a noteworthy and rather mysterious fact. In the present paper we classify the exceptional domestic automorphisms of \textit{all} finite thick generalised polygons. We show that there are precisely three more examples of such automorphisms (up to duality). The first is a duality of the projective plane of order~$2$ (the Fano plane). The other two are collineations in the small hexagons of order $(2,2)$ and $(2,8)$. In particular, rather surprisingly there are no exceptional domestic automorphisms of finite thick octagons. Thus, in total, precisely $6$ exceptional domestic automorphisms of finite thick generalised polygons exist (up to duality). It turns out that all of the exceptional domestic collineations miraculously have order~4. We do not have an explanation (other than the proof) for this curious phenomenon. The exceptional domestic duality of the Fano plane has order~$8$. 

We also study automorphisms of generalised polygons which are domestic but not exceptional domestic. If $\theta$ is such an automorphism then either $\theta$ maps no point to an opposite point/line (in which case $\theta$ is called \textit{point-domestic}), or $\theta$ maps no line to an opposite line/point (in which case $\theta$ is called \textit{line-domestic}). We show that point-domestic automorphisms and line-domestic automorphisms do not exist for thick $(2n+1)$-gons, while for $2n$-gons the fixed element structures of point-domestic automorphisms and line-domestic automorphisms are shown to be intimately related to \textit{ovoidal subspaces} in the polygon (extending the above observation that domesticity is associated with interesting subconfigurations in spherical buildings). We make this description more concrete by classifying ovoidal subspaces into 3 distinct classes. Moreover, in the case of hexagons and octagons we provide further restrictions on the fixed element structures of point-domestic and line-domestic collineations in terms of the parameters of the polygon. 

The problem of classifying exceptional domestic automorphisms of infinite generalised polygons with diameter~$5$ or more seems to be very difficult (certainly the techniques used in \cite{TTM:13} for infinite quadrangles do not readily generalise to larger diameter). We are tempted to conjecture that there exists many exceptional domestic automorphisms of these polygons, but we have failed to construct a single one of them. So this problem remains open. However, a partial motivation for our investigations is the fact that in \textit{Phan-theory} (see \cite{gra:09}) one is interested in the geometry of chambers mapped to an opposite by an automorphism (usually an involution). Thus knowing when this geometry is empty is good start. Phan-theory is particularly interesting when applied to ``algebraic'' buildings, and in this case the rank~$2$ residues are Moufang polygons. In this respect, we show that infinite Moufang hexagons do not admit exceptional domestic collineations. A similar result for the infinite Moufang octagons is not available, but it seems rather safe to conjecture that no exceptional domestic collineations exist for these polygons either.

We note that the two exceptional domestic collineations of finite hexagons are in some sense the analogues of the exceptional domestic collineations of the quadrangles with parameters $(2,2)$ and $(2,4)$. The exceptional domestic collineation of the generalised quadrangle with parameters $(3,5)$ has no analogue in hexagons (or octagons). The square of this collineation is an \emph{anisotropic} involution, i.e., an involution sending \textit{every} chamber to an opposite chamber (it is, rather amazingly, the exact ``opposite" of a domestic collineation in that it maps as much as possible to opposites). Anisotropic automorphisms seem very rare indeed for thick finite irreducible buildings of rank at least~$2$, as is evidenced in \cite{DPM:13} where it is shown that they can only live in generalised polygons with coprime parameters. In the present paper we strengthen this by showing that no finite thick generalised hexagon or octagon admits an anisotropic automorphism. Hence only quadrangles with coprime parameters remain (and for some quadrangles with parameters $(2^{n}-1,2^n+1)$ examples of anisotropic automorphisms are known). This gives a ``moral'' reason why there is no analogue of the exceptional domestic collineation of the quadrangle with parameters $(3,5)$ for hexagons or octagons. Moreover, the non-existence of anisotropic automorphisms for hexagons and octagons is important in its own right, as it will certainly be useful in the classification of finite flag-transitive polygons. 

We use a variety of methods in this paper. The main drive of our classification theorems are combinatorial counting arguments, combined with geometric and group-theoretic arguments. Moreover, in \cite{tem:10} a generalisation of Benson's Theorem \cite{ben:70} on collineations of finite quadrangles was developed, and in the present paper we place these eigenvalue techniques into a systematic framework and make extensive use of them. Indeed at quite a few critical points in our arguments it is a wonderful miracle that when all of the combinatorial, geometric and group-theoretic arguments seem hopeless, the eigenvalue techniques come through to deliver the solution.

This paper is organised as follows. In Section~\ref{sect:2} we give definitions and provide precise statements of our main results. In Section~\ref{sect:3} we consider domesticity in $(2n+1)$-gons and give the classification of domestic automorphisms in projective planes. In Section~\ref{sect:4} we classify ovoidal subspaces in generalised polygons, and we show that the fixed element structures of point-domestic and line-domestic automorphisms of $2n$-gons are closely related to these subspaces. In Section~\ref{sect:5} we begin our study of exceptional domestic automorphisms by developing general combinatorial techniques for studying automorphisms of generalised polygons. In Section~\ref{sect:6a} we classify exceptional domestic automorphisms of finite thick generalised quadrangles, and in Section~\ref{sect:6} we give the classification for hexagons. In Section~\ref{sect:7} we carry out the much more intricate analysis for octagons. In Section~\ref{sect:8} we show that infinite Moufang hexagons do not admit exceptional domestic collineations. In Section~\ref{sect:9} we study anisotropic automorphisms, and prove that finite thick hexagons and octagons do not admit such automorphisms. Finally, in Appendix~\ref{app:A} we develop the eigenvalue techniques that are used at various points in this paper.

 We thank Alice Devillers for noting a mistake in an earlier version of Theorem~\ref{thm:1a}, and for supplying the example of the exceptional domestic duality of the Fano plane. We had initially overlooked this example, and it was also overlooked in~\cite{TTM:11}.

\section{Definitions and statement of results}\label{sect:2}

\subsection{Definitions}

\begin{defn}
Let $n\geq 2$ be a natural number. A \textit{generalised $n$-gon} is a nonempty point-line  geometry $\Gamma=(\mathcal{P},\mathcal{L},\mathbf{I})$ such that the following two axioms are satisfied:
\begin{enumerate}
\item[1.] $\Gamma$ contains no ordinary $k$-gon (as a subgeometry) for $2\leq k<n$.
\item[2.] Any two elements $x,y\in\mathcal{P}\cup\mathcal{L}$ are contained in some ordinary $n$-gon in~$\Gamma$.
\end{enumerate}
A generalised $n$-gon $\Gamma$ is \textit{thick} if it satisfies the additional axiom:
\begin{enumerate}
\item[3.] Each point of $\Gamma$ is on at least~$3$ lines, and each line contains at least~$3$ points.
\end{enumerate}
\end{defn}

As a general rule, we use lower case letters $p,q,\ldots$ to denote points, and capital letters $L,M,\ldots$ to denote lines. However in some instances, especially in definitions, we use the lower case letters $x,y,z$ to denote either points or lines. The polygon obtained by interchanging the roles of points and lines is called the \textit{dual} polygon, and we have the principle of duality: Every statement about generalised polygons has a dual statement by interchanging the roles of points and lines.

By the Feit-Higman Theorem~\cite{FH:64}, thick finite generalised $n$-gons only exist for $n=2,3,4,6,8$. The \textit{digons} ($n=2$) are trivial objects with every point incident with every line, and are of very little interest here. The generalised $3$-gons are \textit{projective planes}, and the generalised $n$-gons with $n=4,6,8$ are called \textit{generalised quadrangles}, \textit{generalised hexagons}, and \textit{generalised octagons} respectively. We sometimes omit the adjective `generalised' from these expressions.

If $\Gamma$ is a thick finite generalised $n$-gon with $n\geq 3$ then it follows that every point is incident with the same number of lines, say $t+1$ lines, and every line contains the same number of points, say $s+1$ points. The integers $s$ and $t$ are called the \textit{parameters} of~$\Gamma$. The precise determination of the parameter values $(s,t)$ that occur for a finite thick generalised $n$-gon is a famous open problem, however there are some well known constraints on the parameters, including (see \cite{FH:64,hae:81,hig:75}):
\begin{enumerate}
\item[] Projective planes: $s=t$.
\item[] Quadrangles: $s\leq t^2$, $t\leq s^2$, and $s^2(st+1)/(s+t)\in\mathbb{Z}$.
\item[] Hexagons: $s\leq t^3$, $t\leq s^3$, $s^3(s^2t^2+st+1)/(s^2+st+t^2)\in\mathbb{Z}$, and $\sqrt{st}\in\mathbb{Z}$.
\item[] Octagons: $s\leq t^2$, $t\leq s^2$, $s^4(s^3t^3+s^2t^2+st+1)/(s^3+s^2t+st^2+t^3)\in\mathbb{Z}$, and $\sqrt{2st}\in\mathbb{Z}$.
\end{enumerate}
In particular, for generalised octagons we have $s\neq t$. Up to duality, the known examples of finite thick quadrangles have $(s,t)\in\{(q,q),(q,q^2),(q^2,q^3),(q-1,q+1)\}$, and the known examples of finite thick hexagons have parameters $(s,t)\in\{(q,q),(q,q^3)\}$, with $q$ a prime power. Up to duality, the only known examples of finite thick octagons have parameters $(s,t)=(r,r^2)$ with $r$ an odd power of~$2$. The smallest thick projective plane is the \textit{Fano plane}, with parameters $(s,t)=(2,2)$.

The \textit{incidence graph} of a generalised $n$-gon $\Gamma=(\cP,\cL,\mathbf{I})$ is the graph with vertex set $\cP\cup\cL$, and vertices $x$ and $y$ from an edge if and only if $x\bI y$. The distance between $x,y\in\cP\cup\cL$ in the incidence graph is denoted~$d(x,y)$. The incidence graph of $\Gamma$ has diameter~$n$ and girth~$2n$, and these properties actually characterise generalised $n$-gons.

An \textit{automorphism} of a generalised $n$-gon~$\Gamma$ is a bijection $\theta:\cP\cup\cL\to\cP\cup\cL$ such that $p\in\cP$ and $L\in\cL$ are incident if and only if $p^{\theta}$ and $L^{\theta}$ are incident. An automorphism which maps points to points and lines to lines is called a \textit{collineation}, and an automorphism which interchanges the point and line sets is called a \textit{duality}. The \textit{fixed element structure} of an automorphism $\theta:\Gamma\to \Gamma$ is $\Gamma_{\theta}=\{x\in\cP\cup\cL\mid x^{\theta}=x\}$. It is well known (and easy to prove) that the fixed element structure of an automorphism of a generalised $n$-gon is either empty, or consists of a set of elements all opposite one another, or is a tree of diameter at most $n$ in the incidence graph, or is a sub-generalised $n$-gon.

As already discussed in the introduction, the notion of \textit{opposition} is fundamental in the theory of generalised $n$-gons (and more generally twin buildings). The element $x\in\cP\cup\cL$ is \textit{opposite} the element $y\in \cP\cup\cL$ if $d(x,y)=n$. In other words, $x$ is opposite~$y$ if $x$ and $y$ are at maximum distance in the incidence graph of~$\Gamma$. It is clear that if $n$ is even then opposite elements are of the same type (that is, either both points, or both lines), and if $n$ is odd then points are opposite lines, and lines are opposite points. A \textit{maximal flag} in a generalised $n$-gon is a pair $\{p,L\}$ with $p\bI L$. Borrowing from the general language of buildings, we refer to a maximal flag as a \textit{chamber}. The chamber $\{p,L\}$ is \textit{opposite} the chamber $\{p',L'\}$ if either $p'$ is opposite $p$ and $L'$ is opposite $L$, or $p'$ is opposite $L$ and $L'$ is opposite~$p$ (the former case occurs if $n$ is even, and the latter if $n$ is odd).

As a special case of a result of Leeb \cite[Sublemma~5.22]{lee:00} (see also \cite[Proposition~4.2]{AB:09}) every nontrivial automorphism of a thick generalised polygon maps ``something'' to an opposite, in the sense that it must map at least~$1$ point to an opposite point/line, or at least $1$ line to an opposite line/point, or at least $1$ chamber to an opposite chamber. This motivates the following definition.

\begin{defn} An automorphism of a generalised polygon (or more generally of a spherical building) is called \textit{domestic} if it maps no chamber to an opposite chamber. Furthermore, a domestic automorphism of a generalised polygon is called
\begin{itemize}
\item[1.] \textit{line-domestic} if it maps no line to an opposite line/point,
\item[2.] \textit{point-domestic} if it maps no point to an opposite point/line,
\item[3.] \textit{exceptional domestic} if it maps at least 1 point to an opposite point/line, and at least 1 line to an opposite line/point.
\end{itemize}
\end{defn}
Thus every domestic automorphism of a generalised polygon is either line-domestic, or point-domestic, or exceptional domestic, and these classes are disjoint. 

\begin{defn} An automorphism of a generalised polygon (or more generally of a spherical building) is called \textit{anisotropic} if it maps every chamber to an opposite chamber.
\end{defn}

Thus anisotropic automorphisms are the ``opposite'' of a domestic automorphism because they map ``as much as possible'' to opposites. However there are some remarkable connections: For example, it turns out that the square of the unique exceptional domestic collineation of the generalised quadrangle of order $(3,5)$ is an anisotropic collineation. 

Our description of the fixed element structures of point-domestic and line-domestic collineations of generalised $2n$-gons in Section~\ref{sect:4} is in terms of \textit{ovoidal subspaces}. We now give the definition of these objects. Note that they first appeared in \cite{BM:98} in a rather different context. 

\begin{defn}
Let $\Gamma=(\cP,\cL,\mathbf{I})$ be a generalised $2n$-gon. A \textit{subspace} of $\Gamma$ is a subset $\mathcal{S}\subseteq \cP\cup\cL$ such that
\begin{enumerate}
\item[1.] if $x,y\in\mathcal{S}$ are distinct collinear points  then the line determined by $x$ and $y$ is in $\mathcal{S}$, and 
\item[2.] if a line $L$ is in $\mathcal{S}$ then all points on $L$ are in $\mathcal{S}$. 
\end{enumerate}
\end{defn}

\begin{defn}
Let $\Gamma$ be a generalised $2n$-gon. A subspace $\mathcal{S}$ of $\Gamma$ is \textit{ovoidal} if:
\begin{enumerate}
\item[1.] if $x\in\cP\cup\cL$ then there is $y\in\cS$ with $d(x,y)\leq n$, and
\item[2.] if $x\in\cP\cup\cL$ and $y\in\cS$ with $d(x,y)<n$ then~$x$ is at minimal distance from a unique element~$z$ of $\mathcal{S}$. (By the definition of subspaces, $z$ is necessarily a point). 
\end{enumerate}
\end{defn}

Let $\Gamma$ be a generalised $2n$-gon. Recall that a \textit{distance $n$-ovoid} in $\Gamma$ is a set $\mathcal{S}$ of mutually opposite points such that every element of $\Gamma$ is at distance at most $n$ from some element of~$\mathcal{S}$. Recall that a subpolygon $\Gamma'$ of $\Gamma$ is \textit{full} if every point of $\Gamma$ incident with a line of $\Gamma'$ belongs to $\Gamma'$, and dually a subpolygon $\Gamma'$ is \textit{ideal} if every line of $\Gamma$ incident with a point of $\Gamma'$ belongs to~$\Gamma'$. A subpolygon $\Gamma'$ is \textit{large} if every element of $\Gamma$ is at distance at most $n$ from some element of~$\Gamma'$.

A \textit{path} in the incidence graph of a generalised $n$-gon is a sequence $x_0\bI x_1\bI\cdots \bI x_k$ with $x_i\in \cP\cup\cL$ for $0\leq i\leq k$. This path is a \textit{geodesic} if $d(x_0,x_k)=k$. It is useful to note that if $k\leq n$ then the path $x_0\bI x_1\bI\cdots \bI x_k$ is a geodesic if and only if $x_{i-1}\neq x_{i+1}$ for each $1\leq i\leq k-1$. Furthermore, if $d(x,y)<n$ then there is a unique geodesic $x=x_0\bI x_1\bI\cdots\bI x_k=y$ joining $x$ to~$y$. In this case the \textit{projection of $x$ onto $y$} is defined to be $\mathrm{proj}_xy=x_{k-1}$. That is, $\mathrm{proj}_xy$ is the penultimate element on the unique geodesic from $x$ to~$y$. 

For each $x\in \cP\cup\cL$ and each $0\leq k\leq n$ let
$$
\Gamma_k(x)=\{y\in\cP\cup\cL\mid d(x,y)=k\}
$$
be the sphere of radius $k$ centred at~$x$.

\subsection{Statement of results}

Our first results (Section~\ref{sect:3}) deal rather generally with domestic automorphisms in the case of $(2n+1)$-gons. In particular we show that point-domestic and line-domestic automorphisms do not exist for $(2n+1)$-gons, and we give the complete classification of domestic automorphisms in the case of projective planes.

\begin{thm}\label{thm:1a}\mbox{}\vspace{-0.1cm}
\begin{itemize}
\item[\emph{1.}] If an automorphism of a thick generalised $(2n+1)$-gon is domestic then it is an exceptional domestic duality. 
\item[\emph{2.}] If $\Gamma$ is a thick projective plane (potentially infinite) not isomorphic to the Fano plane then~$\Gamma$ admits no domestic automorphisms.
\item[\emph{3.}] The Fano plane admits a unique domestic automorphism up to conjugation. This automorphism is an exceptional domestic duality of order~$8$. In $\mathbb{ATLAS}$ notation \emph{\cite[p.3]{atlas}} the conjugacy class of this domestic duality is class~$8A$.
\end{itemize}
\end{thm}

Theorem~\ref{thm:1a} corrects an error from~\cite{TTM:11}, where it is claimed that no duality of any even-dimensonal projective space is domestic. This conflicts with the existence of the domestic duality of the Fano plane described in Theorem~\ref{thm:1a}. The reason is that in the proof provided in~\cite{TTM:11}, the case where the point set of the projective space is the union of two points and two hyperplanes is overlooked, and this situation occurs precisely for the Fano plane. Thus the correct statement to replace the second sentence in the statement of \cite[Theorem~3.1]{TTM:11} is that the only domestic duality of an even-dimensional projective space is the duality of the Fano plane described above. 

Our next results (Section~\ref{sect:4}) characterise point-domestic and line-domestic automorphisms of $2n$-gons by classifying their fixed element structures. 

\begin{thm}\label{thm:1b} No duality of a thick generalised $2n$-gon is domestic. The fixed element structure of a line-domestic collineation of a thick generalised $4n$-gon and the fixed element structure of a point-domestic collineation of a thick generalised $(4n+2)$-gon is an ovoidal subspace. 
\end{thm}

The need to distinguish between $4n$ and $(4n+2)$-gons in the statement of Theorem~\ref{thm:1b} is superficial, and only arises because the notion of an ovoidal subspace is not a `self-dual' notion. Of course a point-domestic collineation of a generalised $4n$-gon $\Gamma$ is a line-domestic collineation of the dual $4n$-gon $\Gamma^D$, and thus the fixed element structure is an ovoidal subspace in~$\Gamma^D$. Similarly a line-domestic collineation of a generalised $(4n+2)$-gon~$\Gamma$ is a point-domestic collineation of the dual $(4n+2)$-gon~$\Gamma^D$, and thus the fixed element structure is an ovoidal subspace in~$\Gamma^D$.

Next we classify ovoidal subspaces, thus making Theorem~\ref{thm:1b} more explicit:

\begin{thm}\label{thm:ovoidal} An ovoidal subspace $\mathcal{S}$ of a generalised $2n$-gon $\Gamma$ is either:
\begin{itemize}
\item[\emph{1.}] A distance $n$-ovoid.
\item[\emph{2.}] A large full subpolygon. 
\item[\emph{3.}] The ball $B_n(x)=\{y\in \cP\cup \cL\mid d(x,y)\leq n\}$ of radius~$n$ centred at an element $x\in \cP\cup\cL$. Moreover, if $n$ is even then $x\in\cP$ and if $n$ is odd then $x\in\cL$.
\end{itemize}
\end{thm}

In the case of finite generalised $2n$-gons we also provide further restrictions on the fixed element structures of point-domestic and line-domestic automorphisms in terms of the parameters of the polygon. For example, we show that generalised hexagons with parameters $(t^3,t)$ do not admit point-domestic automorphisms (see the results in Section~\ref{sect:4} for more details).

Next we turn our attention to exceptional domestic collineations of finite generalised $2n$-gons. The case $n=2$ (quadrangles) is covered in~\cite{TTM:13}, although we provide a brief exposition of this case using our methods in~Section~\ref{sect:6a}. This serves as a useful illustration of the techniques that we use for the more involved cases of generalised hexagons and octagons. In Section~\ref{sect:6} we give the classification for finite hexagons:

\begin{thm}\label{thm:2}
Exceptional domestic collineations of finite thick generalised hexagons exist precisely for the classical hexagons with parameters $(s,t)\in\{(2,2),(2,8),(8,2)\}$. In each case there is a unique such collineation up to conjugation, and it has order~$4$. The fixed element structure in the $(s,t)=(2,2)$ consists of a point-line pair $\{p,L\}$, together with all lines through~$p$ and all points on~$L$. The fixed element structure in the $(s,t)=(8,2)$ case consists of a line $L$, all $9$ points on~$L$, and all lines through~$7$ of the points on~$L$ (and dually for the $(2,8)$ case).
\end{thm}

The situation for finite octagons is rather different, and perhaps surprisingly no exceptional domestic collineations exist. Thus, in Section~\ref{sect:7} we prove: 

\begin{thm}\label{thm:3}
No finite thick generalised octagon admits an exceptional domestic collineation.
\end{thm}

It is elementary that there are no exceptional domestic automorphisms of generalised $2$-gons (digons). Thus, combining Theorems~\ref{thm:1a},~\ref{thm:2} and~\ref{thm:3} with \cite[Theorem~2.1]{TTM:13} we obtain the following complete and very satisfying classification of exceptional domestic automorphisms of finite thick generalised polygons:

\begin{cor}\label{cor:4} Let $\Gamma$ be a thick finite generalised $n$-gon with parameters $(s,t)$, and suppose that $\Gamma$ admits an exceptional domestic automorphism~$\theta$. Then $n=3$, $n=4$, or $n=6$. Furthermore:
\begin{itemize}
\item[\emph{1.}] If $n=3$ then $(s,t)=(2,2)$, and $\theta$ is a duality of order~$8$, unique up to conjugation. 

\item[\emph{2.}] If $n=4$ then $(s,t)\in\{(2,2),(2,4),(4,2),(3,5),(5,3)\}$ and $\theta$ is a collineation of order~$4$. In each case there is a unique exceptional domestic collineation up to conjugation. In the cases $(s,t)\in\{(3,5),(5,3)\}$ the fixed element structure of this collineation is empty. The fixed element structures for the remaining cases are as follows:

\begin{figure}[!h]
\centering
\subfigure[$(s,t)=(2,2)$]{
\begin{tikzpicture}[scale=0.7]
\draw (-1.5,0) -- (1.5,0);
\node at (-1,0) {$\bullet$};
\end{tikzpicture}
}\hspace{2cm}
\subfigure[$(s,t)=(4,2)$, and dually for $(s,t)=(2,4)$]{
\begin{tikzpicture}[scale=0.7]
\draw (-1.5,0) -- (1.5,0);
\phantom{\draw (-4.5,0) -- (4.5,0);}
\node at (-1,0) {$\bullet$};
\node at (0,0) {$\bullet$};
\node at (1,0) {$\bullet$};
\end{tikzpicture}
}
\caption{Exceptional domestic collineations of quadrangles}\label{fig:fixedquad}
\end{figure}
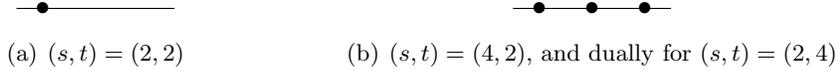

\item[\emph{3.}] If $n=6$ then $(s,t)\in\{(2,2),(2,8),(8,2)\}$ and $\theta$ is a collineation of order~$4$. In each case there is a unique exceptional domestic collineation up to conjugation, and the fixed element structures are as follows: 

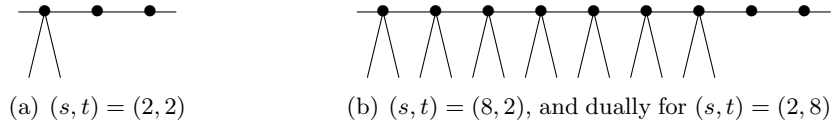
\begin{figure}[!h]
\centering
\subfigure[$(s,t)=(2,2)$]{
\begin{tikzpicture}[scale=0.7]
\draw (-1.5,0) -- (1.5,0);
\draw (-1.3,-1.25) -- (-1,0) -- (-0.7,-1.25);
\node at (-1,0) {$\bullet$};
\node at (0,0) {$\bullet$};
\node at (1,0) {$\bullet$};
\end{tikzpicture}
}\hspace{2cm}
\subfigure[$(s,t)=(8,2)$, and dually for $(s,t)=(2,8)$]{
\begin{tikzpicture}[scale=0.7]
\draw (-4.5,0) -- (4.5,0);
\draw (-4.3,-1.25) -- (-4,0) -- (-3.7,-1.25);
\draw (-3.3,-1.25) -- (-3,0) -- (-2.7,-1.25);
\draw (-2.3,-1.25) -- (-2,0) -- (-1.7,-1.25);
\draw (-1.3,-1.25) -- (-1,0) -- (-0.7,-1.25);
\draw (-0.3,-1.25) -- (0,0) -- (0.3,-1.25);
\draw (0.7,-1.25) -- (1,0) -- (1.3,-1.25);
\draw (1.7,-1.25) -- (2,0) -- (2.3,-1.25);
\node at (-4,0) {$\bullet$};
\node at (-3,0) {$\bullet$};
\node at (-2,0) {$\bullet$};
\node at (-1,0) {$\bullet$};
\node at (0,0) {$\bullet$};
\node at (1,0) {$\bullet$};
\node at (2,0) {$\bullet$};
\node at (3,0) {$\bullet$};
\node at (4,0) {$\bullet$};
\end{tikzpicture}}
\caption{Exceptional domestic collineations of hexagons}\label{fig:fixed}
\end{figure}
\end{itemize}
\end{cor}

\begin{remark}
For each of the cases occurring in Corollary~\ref{cor:4} it is known that there is a unique generalised $n$-gon with the given parameters (see \cite[Chapter~6]{PT:09} and \cite{CT:85}). In most cases detailed information on the automorphism groups can be found in the~$\mathbb{ATLAS}$ of Finite Groups, with the exceptions being the generalised quadrangles with $(s,t)= (3,5),(5,3)$ since in these cases the group is not simple. In all other cases we identify the conjugacy classes of the exceptional domestic automorphisms. Specifically, the conjugacy classes of the domestic automorphisms in the cases $(n,s,t)=(3,2,2),(4,2,2),(4,2,4),(6,2,2),(6,8,2)$ are~$8A$ \cite[p.3]{atlas}, $4A$ \cite[p.5]{atlas}, $4C$ \cite[p.27]{atlas}, 
$4C$ \cite[p.14]{atlas}, and~$4A$ \cite[p.89]{atlas}, respectively.
\end{remark}

\begin{remark} The classification of exceptional domestic collineations of non-thick polygons can be derived from the thick case. For example, suppose that $\Gamma$ is the double of a thick generalised polygon $\Gamma'$ such that $\Gamma$ has $2$ points on each line. Then: (1) $\Gamma$ does not admit an exceptional domestic automorphism. (2) If $\theta$ is a point domestic automorphism of $\Gamma$ then $\theta$ is the identity. (3) If $\theta$ is a line-domestic automorphism of $\Gamma$ then $\theta$ is domestic in the underlying polygon~$\Gamma'$.
\end{remark}

As discussed in the introduction, once the diameter of the polygon is $5$ or more the study of exceptional domestic automorphisms of infinite generalised polygons seems to be difficult. Thus a natural first step is to consider the Moufang case. We have been able to treat the case of Moufang hexagons. The case of Moufang octagons appears to be considerably harder, and remains open. In Section~\ref{sect:8} we prove:

\begin{thm}\label{thm:5} No infinite Moufang hexagon admits an exceptional domestic automorphism.
\end{thm}

The study of domestic automorphisms is, morally, the study of ``how little'' can be mapped to opposites. In Section~\ref{sect:9} we turn our attention to the companion question, and investigate ``how much'' can be sent to opposites. The most drastic situation is that of an anisotropic automorphism, where \textbf{every} chamber is mapped to an opposite chamber. We prove the following theorem (we actually prove a more precise result, see Section~\ref{sect:9}).

\begin{thm}\label{thm:6} No finite thick generalised hexagon or octagon admits an anisotropic automorphism.
\end{thm}

It follows from \cite[Theorem~1.3]{DPM:13} that no duality of a generalised $2n$-gon can be anisotropic, and that no collineation of a generalised $(2n+1)$-gon can be anisotropic. In \cite[Proposition~5.4]{DPM:13} it is shown that no duality of a finite projective plane is anisotropic, and in \cite[Corollary~4.3]{TTM:09} (see also \cite{ben:70}) it is shown that if a finite thick generalised quadrangle with parameters $(s,t)$ admits an anisotropic collineation then $s$ and $t$ are coprime. Thus, in the language of buildings (see~\cite{AB:08}), Theorem~\ref{thm:6} implies that:

\begin{cor}\label{cor:7}
If a finite thick irreducible spherical building of rank at least~$2$ admits an anisotropic automorphism, then the building is a quadrangle with coprime parameters.
\end{cor}

Note that there are examples of anisotropic automorphisms in certain finite generalised quadrangles with coprime parameters.  See \cite[Remark~4.5]{TTM:13}, where it is noted that there are quadrangles of order $(2^n-1,2^n+1)$ which admit anisotropic automorphisms. For example, it turns out that the square of the exceptional domestic collineation in the quadrangle with parameters~$(3,5)$ is anisotropic. There are also plenty of examples of anisotropic automorphisms of infinite generalised polygons.

\section{Domesticity in $(2n+1)$-gons}\label{sect:3}

The results of this section prove Theorem~\ref{thm:1a}.

\begin{lemma}\label{lem:oddcollineation}
No collineation of a thick generalised $(2n+1)$-gon is domestic.
\end{lemma}

\begin{proof}
In a $(2n+1)$-gon, if $x\in\cP\cup\cL$ is opposite $y\in\cP\cup\cL$ then either $x\in\cP$ and $y\in\cL$, or $x\in\cL$ and $y\in\cP$. Since collineations preserve $\cP$ and $\cL$ it is impossible for any point or line to be mapped to an opposite point or line. Thus every collineation of a thick generalised $(2n+1)$-gon must map at least one chamber to an opposite chamber (by \cite[Sublemma~5.22]{lee:00}).
\end{proof}

\begin{lemma}\label{lem:oddpldom}
No duality of a thick generalised $(2n+1)$-gon is point-domestic or line-domestic.
\end{lemma}

\begin{proof}
If $\theta$ is a domestic duality of a thick $(2n+1)$-gon then we may assume up to duality that a point $p$ is mapped to an opposite line $L=p^{\theta}$ (by \cite[Sublemma~5.22]{lee:00}). Let $p\bI L_1\bI p_1\bI\cdots \bI L_n\bI p_n\bI L$ be a path of length $2n+1$ from $p$ to $L$. Applying $\theta$ we have a path $L\bI L_1^{\theta}\bI p_1^{\theta}\bI\cdots\bI L_n^{\theta}\bI p_n^{\theta}\bI L^{\theta}$, and so the line~$L$ is mapped to an opposite point $p'=L^{\theta}$. Thus $\theta$ is neither point-domestic nor line-domestic.
\end{proof}

Thus the first part of Theorem~\ref{thm:1a} is proven. We now prove the remaining parts of Theorem~\ref{thm:1a}, which we state again for convenience.

\begin{thm}\label{thm:pp} Let $\Gamma$ be a thick projective plane, not necessarily finite.
\begin{itemize}
\item[\emph{1.}] If $\Gamma$ is not isomorphic to the Fano plane then $\Gamma$ admits no domestic automorphisms.
\item[\emph{2.}] If $\Gamma$ is isomorphic to the Fano plane then there exists a unique domestic automorphism up to conjugation. This automorphism is an exceptional domestic duality of order~$8$. In $\mathbb{ATLAS}$ notation \emph{\cite[p.3]{atlas}} the conjugacy class of the domestic duality is class~$8A$.
\end{itemize}
\end{thm}

\begin{proof} By Lemma~\ref{lem:oddcollineation} and Lemma~\ref{lem:oddpldom} the only possible domestic automorphisms of thick projective planes are exceptional domestic dualities. 

Let $\Gamma$ be a thick projective plane with parameter~$s$ (with the possibility that $s=\infty$), and suppose that $\theta$ is an exceptional domestic duality of~$\Gamma$. Thus there exists a point $p$ mapped to an opposite line~$L=p^{\theta}$, and by the proof of Lemma~\ref{lem:oddpldom} the line $L$ is mapped to an opposite point $p'=L^{\theta}$.

Let $L_1,\ldots,L_{s+1}$ be the lines through~$p$, and for each $i$ let $p_i=L\cap L_i$. By domesticity, $d(p_i,p_i^{\theta})=1$ for each~$i$, and so $p_i^{\theta}$ is a line through~$p_i$ (necessarily distinct from~$L$, but perhaps equal to $L_i$). Each line through $p_i$ distinct from $L_i,p_i^{\theta}$ is mapped to an opposite point. Similarly, note that $L_i^{\theta^{-1}}$ is a point on $L_i$ (necessarily distinct from $p$, but perhaps equal to~$p_i$), and that each point on $L_i$ distinct from $p_i,L_i^{\theta^{-1}}$ is mapped to an opposite line. 

If $p'=p$ then $p_i^{\theta}=L_i$ for all~$i$. By thickness there is a line $M$ through $p_1$ with $M\neq L,L_1$. Since $L_1=p_1^{\theta}$ we have $M\neq L_1,p_1^{\theta}$, and so $M$ is mapped by $\theta$ to an opposite point. Let $q=M\cap L_2$. Then $q\neq p_2$, and since $p_2=L_2^{\theta^{-1}}$ we have $q\neq p_2,L_2^{\theta^{-1}}$, and so~$q$ is mapped to an opposite line. Thus the chamber $\{q,M\}$ is mapped to an opposite chamber, a contradiction.

Thus $p'\neq p$. After relabelling the lines $L_1,\ldots,L_{s+1}$ we may assume that $p'$ is a point on~$L_1$. Then $p_1^{\theta}=L_1$, and $p_i^{\theta}\neq L_i$ for all $2\leq i\leq s+1$. If $s\geq 3$ then there is a line $M$ through $p_2$ with $M\neq L,L_2,p_2^{\theta}$, and so $M$ is mapped to an opposite point. Let $q=M\cap L_1$. Since $q\neq p_1=L_1^{\theta^{-1}}$ we see that $q$ is mapped to an opposite line, and so the chamber $\{q,M\}$ is mapped to an opposite chamber, a contradiction. 
Thus $s=2$, and so $\Gamma$ is the Fano plane. It is straight forward to see that the conditions $p_1^{\theta}=L_1$, $p_2^{\theta}\neq L_2$, and $p_3^{\theta}\neq L_3$ plus domesticity completely determine~$\theta$. Explicitly, let $\cP=\{p,p_1,p_2,p_3,q_1,q_2,q_3\}$ where for each $i$ we write $q_i$ for the point on $L_i$ with $q_i\neq p,p_i$, and let $\cL=\{L,L_1,L_2,L_3,M_1,M_2,M_3\}$ where for each $i$ we write $M_i$ for the line through $p_i$ with $M_i\neq L,L_i$ (see the figure). Then~$\theta$ is as follows.

\vspace{-0.3cm}
\noindent\begin{minipage}{0.6\textwidth}
\begin{align*}
p^{\theta}&= L&L^{\theta}&= q_1\\
p_1^{\theta}&= L_1& L_1^{\theta}&= p_1\\
p_2^{\theta}&= M_2&L_2^{\theta}&= p_2\\
p_3^{\theta}&= M_3&L_3^{\theta}&= p_3\\
q_1^{\theta}&= M_1&M_1^{\theta}&= p\\
q_2^{\theta}&= L_2&M_2^{\theta}&= q_3\\
q_3^{\theta}&= L_3&M_3^{\theta}&= q_2
\end{align*}
\end{minipage}
\begin{minipage}{0.4\textwidth}
\vspace{0.5cm}\begin{center}
\begin{tikzpicture} [scale=1.1]
\path 
(0,0) node (0) [shape=circle,draw,fill=black,scale=0.5,label=right:\footnotesize{$q_2$}] {}
({cos(30)},{sin(30)}) node (1) [shape=circle,draw,fill=black,scale=0.5,label=right:\footnotesize{$p_2$}] {}
({cos(150)},{sin(150)}) node (2) [shape=circle,draw,fill=black,scale=0.5,label=left:\footnotesize{$q_3$}] {}
({cos(270)},{sin(270)}) node (3) [shape=circle,draw,fill=black,scale=0.5,label=below:\footnotesize{$q_1$}] {}
(0,{1/(sin(30))}) node (4) [shape=circle,draw,fill=black,scale=0.5,label=right:\footnotesize{$p_3$}] {}
({-1/(tan(30))},{-1}) node (5) [shape=circle,draw,fill=black,scale=0.5,label=below left:\footnotesize{$p$}] {}
({1/(tan(30))},{-1}) node (6) [shape=circle,draw,fill=black,scale=0.5,label=below right:\footnotesize{$p_1$}] {}
;
\draw (1)--(4)--(2)--(5)--(3)--(6)--(1);
\draw (4)--(0)--(3);
\draw (5)--(0)--(1);
\draw (6)--(0)--(2);
\draw (0,0) circle (1cm);
\draw [<-] (1.4,-0.3)--(1.75,-0.2);
\path (1.9,-0.15) node {\footnotesize{$L$}};
\draw [<-] (-1.4,-0.3)--(-1.75,-0.2);
\path (-2,-0.15) node {\footnotesize{$L_3$}};
\draw [->] (-0.9,1.25)--(-0.1,1.25);
\path (-1.2,1.25) node {\footnotesize{$M_3$}};
\draw [bend left,<-] (0.3,1) to (1,1.5);
\path (1.3,1.5) node {\footnotesize{$M_2$}};
\draw [<-] (-1,-0.65)--(-1,-1.4);
\path (-1,-1.6) node {\footnotesize{$L_2$}};
\draw [<-] (1.3,-0.65)--(1.85,-0.65);
\path (2.1,-0.65) node {\footnotesize{$M_1$}};
\draw [<-] (1,-1.1)--(1,-1.4);
\path (1,-1.6) node {\footnotesize{$L_1$}};
\end{tikzpicture}
\end{center}
\end{minipage}
\vspace{0.2cm}

Thus $\theta$ has order~$8$. Note that $\theta$ is determined by the choice of a triple $(p,L_1,L)$ with $p$ a point, $L_1$ a line through $p$, and $L$ a line opposite $p$. Since the automorphism group acts transitively on such triples it follows that $\theta$ is unique up to conjugation. Thus in $\mathbb{ATLAS}$ notation \cite[p.3]{atlas} the conjugacy class of $\theta$ is either~$8A$ or $B*$. 

To show that $\theta$ is in class~$8A$ we make an explicit calculation to verify that $\chi_4(\theta)=\sqrt{2}$ (in the notation of \cite[p.3]{atlas}). To do this we build the~$6$-dimensional irreducible representation~$\chi_4$ from the point set of~$\Gamma$. Identify the points with the column vectors $p=(6,-1,-1,-1,-1,-1,-1)/\sqrt{42}$, $p_1=(-1,6,-1,-1,-1,-1,-1)/\sqrt{42}$, and so on. Identify each line $L'$ of~$\Gamma$ with a unit vector in the intersection of the space spanned by the points on~$L'$ with the space spanned by the points not on~$L'$.  Explicitly we may take
$$
L'=\frac{1}{7\sqrt{2}}\left(4\sum p'-3\sum p''\right)
$$
where the first sum is over the points~$p'$ with $p'\bI L'$, and the second is over the points $p''$ not incident with~$L'$. The automorphism $\theta$ acts on the space $V$ spanned by the points with matrix
$$
\pi(\theta)=\begin{pmatrix}
a&b&a&a&a&b&b\\
b&b&a&a&b&a&a\\
b&a&b&a&a&b&a\\
b&a&a&b&a&a&b\\
a&b&b&b&a&a&a\\
a&a&a&b&b&b&a\\
a&a&b&a&b&a&b
\end{pmatrix}\qquad\textrm{where $a=(2-3\sqrt{2})/14$ and $b=(1+2\sqrt{2})/7$}.
$$
In particular, $\mathrm{Tr}(\pi(\theta))=1+\sqrt{2}$. Since $\pi$ is a direct sum of the trivial representation and the $6$-dimensional irreducible representation~$\chi_4$ we deduce that $\chi_4(\theta)=\sqrt{2}$.
\end{proof}

\begin{cor}\label{cor:oddgons} The only domestic automorphism of a finite thick generalised $(2n+1)$-gon is the exceptional domestic duality of the Fano plane described in Theorem~\emph{\ref{thm:pp}}.
\end{cor}

\begin{proof}
By the Feit-Higman Theorem the only finite thick generalied $(2n+1)$-gons are projective planes, and so the corollary is immediate from Theorem~\ref{thm:pp}.
\end{proof}

We note that there is an error in the proof of the theorems in Section~7.3 of the thesis of the second author that seems to be unfixable, and so deciding if Corollary~\ref{cor:oddgons} holds without the finiteness assumption remains an open problem.

\section{Domesticity in $2n$-gons}\label{sect:4}

The results of this section prove Theorem~\ref{thm:1b}. Firstly, an argument similar to Lemma~\ref{lem:oddcollineation} shows that no duality of a thick generalised $2n$-gon is domestic. Thus the proof is complete when we prove:

\begin{prop} Let $n\geq 1$. The fixed element structure of a line-domestic collineation of a thick generalised $4n$-gon and the fixed element structure of a point-domestic collineation of a thick generalised $(4n+2)$-gon is an ovoidal subspace. 
\end{prop}

\begin{proof}
We give the proof for $(4n+2)$-gons. The proof for $4n$-gons is analogous, and only requires renaming some points as lines, and vice versa. So let $\theta$ be a point-domestic collineation of a generalised $(4n+2)$-gon $\Gamma$. 

We first show that the fixed element structure $\Gamma_{\theta}$ is a subspace. It is clear that if $x$ and $y$ are collinear fixed points then the line determined by $x$ and $y$ is also fixed. Now suppose that there is a fixed line~$L$, and suppose for a contradiction that there is a point $x$ on $L$ with $x^{\theta}\neq x$. Complete the path $L\bI x$ to a path $\gamma$ of length $2n+1$, and let $z$ be the last element on this path (and so $z$ is a point). The concatenation of the paths $\gamma^{-1}$ and $\gamma^{\theta}$ gives a non-stuttering path of length $4n+2$ connecting $z$ with $z^{\theta}$, and so $z^{\theta}$ is opposite~$z$, contradicting point-domesticity. Thus $\Gamma_{\theta}$ is a subspace.

We claim that:
\begin{align}
\label{eq:claim1}
\textrm{There are no points $x$ with $d(x,x^{\theta})=4\ell+2$ for some $0\leq \ell\leq n$.}
\end{align} 
By point-domesticity the claim is true for $\ell=n$. Suppose that $\ell<n$ and that $d(x,x^{\theta})=4\ell+2$. Since $\Gamma$ is thick we may choose a line $M$ through $x$ such that $M\neq \mathrm{proj}_{x^{\theta}}x$ and $M^{\theta}\neq \mathrm{proj}_{x}x^{\theta}$. Thus $d(M,M^{\theta})=4\ell+4$. Since $4\ell+4\neq 4n+2$ we can repeat the argument to find a point $y$ on $M$ with $d(y,y^{\theta})=4(\ell+1)+2$, and continuing we construct a point mapped to an opposite point, contradicting point-domesticity. Thus~(\ref{eq:claim1}) holds. 

Let $x$ be any point of $\Gamma$. By (\ref{eq:claim1}) $d(x,x^{\theta})$ is a multiple of~$4$, and so there is a unique point $z$ with $d(x,z)=d(z,x^{\theta})=d(x,x^{\theta})/2$. We claim that $z^{\theta}=z$, which shows that every point of $\Gamma$ is at distance at most $2n$ from a fixed point, verifying the first property of ovoidal subspaces. Suppose instead that $z^{\theta}\neq z$. Let $d(x,z)=2\ell$, and let $(p_0,p_1,\ldots,p_{\ell})$ be the points on the geodesic from $p_0=z$ to $p_{\ell}=x^{\theta}$, and let $(q_0,q_1,\ldots,q_{\ell})$ be the points on the geodesic from $q_0=z$ to $q_{\ell}=x$. Let $k\geq 0$ be minimal with the property that $p_k=q_k^{\theta}$. Since $p_{\ell}=q_{\ell}^{\theta}$ and $z^{\theta}\neq z$ we have $0<k\leq \ell$. There are two possibilities:
\begin{enumerate}
\item[(1)] If the point $q_{k-1}^{\theta}$ is on the line determined by $p_k$ and $p_{k-1}$ then by minimality of $k$ we have $q_{k-1}^{\theta}\neq p_{k-1}$, and so $d(q_{k-1},q_{k-1}^{\theta})=4(k-1)+2$, contradicting~(\ref{eq:claim1}).
\item[(2)] If $q_{k-1}^{\theta}$ is not on the line determined by $p_k$ and $p_{k-1}$ then for each point $u$ on the line determined by $q_{k-1}$ and $q_k$ distinct from points $q_{k-1}$ and $q_k$ we have $d(u,u^{\theta})=4k+2$, again contradicting~(\ref{eq:claim1}). 
\end{enumerate}
Thus $z=z^{\theta}$.

It remains to show that if some element $x_0$ is at distance less than $2n$ from some fixed point, then it is at minimal distance from a unique fixed point. Let $x_0$ be such a point, and suppose that $x_0$ is at minimal distance $\ell$ from two fixed points $y$ and~$z$. Thus $d(y,x_0)=d(z,x_0)=\ell$. Let $x_0\bI x_1\bI\cdots\bI x_{\ell}$ be a minimal length path from $x_0$ to $y=x_{\ell}$, and let $x_0'\bI x_1'\bI\cdots\bI x_{\ell}'$ be a minimal length path from $x_0'=x_0$ to $z=x_{\ell}'$. Let $j\leq \ell$ be maximal with respect to the property that $x_j=x_j'$. Since $2\ell<4n+2$ the element $x_j$ belongs to the unique shortest path connecting $z$ with $y$, and hence is fixed under~$\theta$. By minimality of $\ell$ we have $j=\ell$, and thus $y=x_{\ell}=x_{\ell}'=z$, completing the proof.
\end{proof}

We now classify ovoidal subspaces, making Theorem~\ref{thm:1b} considerably more explicit.

\begin{lemma}\label{lem:close} Let $\mathcal{S}$ be an ovoidal subspace of a generalised $2n$-gon. Suppose that $x$ and $y$ are points of $\cS$ with $d(x,y)<2n$. Then every element on the unique geodesic joining $x$ and $y$ is contained in~$\cS$.
\end{lemma}

\begin{proof}
Suppose that $x$ and $y$ are points of $\cS$ with $d(x,y)=2k<2n$, and let $\gamma$ be the unique geodesic joining $x$ and $y$. Suppose that the only elements of $\gamma$ in $\cS$ are $x$ and~$y$ (for otherwise we can apply the following argument to a pair of shorter geodesics). Let $m$ be the middle element of the geodesic~$\gamma$. By definition of ovoidal subspaces there is a unique point $z\in\cS$ at minimal distance to~$m$, and since $d(m,x)=d(m,y)=k$ we have $d(m,z)=k_1<k$. Now we can choose $x_1\in\{x,y\}$ such that $m$ is on the geodesic $\gamma_1$ joining $x_1$ to $y_1=z$. Note that the middle element $m_1$ of $\gamma_1$ is contained in $\gamma$, and since $z$ has minimal distance to~$m$ no element on the geodesic $\gamma_1$ other than $x_1$ and $y_1$ is in~$\cS$. Repeating the above argument to the new pair $x_1,y_1$ we eventually find $x',y'\in\cS$ with $d(x',y')=4$ such that no element on the geodesic $\gamma'$ joining $x'$ and $y'$ is contained in~$\cS$. This is plainly in contradiction with the definition of ovoidal subspaces (because there is no unique element of $\cS$ nearest the middle element $m'$ of $\gamma'$), completing the proof.
\end{proof}

\begin{proof}[Proof of Theorem~\emph{\ref{thm:ovoidal}}]
Let $\mathcal{S}$ be an ovoidal subspace of a generalised $2n$-gon $\Gamma$. If $\mathcal{S}$ contains only mutually opposite points then, by definition, $\mathcal{S}$ is a distance $n$-ovoid. If $\mathcal{S}$ contains an ordinary $2n$-gon then by the definition of subspaces and Lemma~\ref{lem:close} we see that $\mathcal{S}$ is a full subpolygon, and from the definition of ovoidal subspace this subpolygon is large. 

It remains to consider the case that $\mathcal{S}$ contains at least one line, but no ordinary $2n$-gon, and we aim to show that $\mathcal{S}=B_n(m)$ for some element $m\in\cP\cup\cL$. Let $x,y$ be points of $\mathcal{S}$ at maximal distance~$d(x,y)=2k$. We claim that $k=n$. Suppose that $k<n$, and let $\Sigma$ be an ordinary $2n$-gon containing $x$ and $y$. Then $\Sigma$ is the union of paths $\gamma_1$ and $\gamma_2$, where $\gamma_1$ is the geodesic joining $x$ and $y$. Let $z$ be the middle element of the longer path~$\gamma_2$. Note that neither $\mathrm{proj}_zx$ nor $\mathrm{proj}_zy$ is in $\mathcal{S}$, for this would contradict maximality of~$d(x,y)$. We also have $z\notin\mathcal{S}$ (for otherwise $\Sigma\subseteq \mathcal{S}$ by Lemma~\ref{lem:close}), and $d(x,z)=d(y,z)=2n-k>n$. Hence there is $u\in\mathcal{S}$ with $d(u,z)\leq n$. Since $\mathrm{proj}_xz\neq\mathrm{proj}_yz$ we may suppose (up to renaming $x$ and $y$) that $\mathrm{proj}_uz\neq \mathrm{proj}_xz$. The non-back-tracking path $\gamma$ from $u$ to $x$ consisting of the geodesic from $u$ to $z$ followed by the geodesic from $z$ to $x$ has length $\ell=d(u,z)+2n-k\leq 3n-k$. Since $z\neq \mathcal{S}$ it follows that $2n\leq \ell$. Let $v$ be the point on $\gamma$ at distance $2n$ from $x$ (with distance measured in $\gamma$). Thus $x$ and $v$ are opposite one another, and $d(u,v)\leq n-k$. Thus, by the triangle inequality, $d(u,x)\geq n+k$, contradicting the maximality of $d(x,y)=2k$. 

Hence $\mathcal{S}$ contains points $x$ and $y$ with $d(x,y)=2n$. By assumption there is a line $L\in\mathcal{S}$. By the properties of subspaces the points $\mathrm{proj}_xL$ and $\mathrm{proj}_yL$ are in $\mathcal{S}$, and so by Lemma~\ref{lem:close} at least $1$ minimal path (of length~$2n$) between $x$ and $y$ is contained in~$\mathcal{S}$. Since we assume that $\mathcal{S}$ contains no ordinary $2n$-gon we conclude that exactly one minimal length path, say $\gamma'$, from $x$ to $y$ is contained in~$\mathcal{S}$. Let $m$ be the middle element of $\gamma'$ (and so $m$ is either a point or a line). We claim that no element at distance bigger than $n$ from $m$ is in~$\mathcal{S}$. Suppose that $z\in\mathcal{S}$ with $d(z,m)=j>n$. We may assume that $j=n+1$ (for if $j<2n$ then, using Lemma~\ref{lem:close}, we can replace $z$ by the element contained in the geodesic joining $z$ and $m$ at distance $n+1$ from $m$, and if $j=2n$ then we find a geodesic from $m$ to $z$ contained in~$\mathcal{S}$, and replace $z$ by the element at distance $n+1$ from $m$ on this geodesic). By renaming $x$ and $y$ if necessary we may suppose that $\mathrm{proj}_zm\neq \mathrm{proj}_xm$. Hence there is a unique ordinary $2n$-gon $\Sigma$ containing $x,m,z$, and by Lemma~\ref{lem:close} $\Sigma$ is contained in $\mathcal{S}$, a contradiction. 

Finally we claim that every element $v$ with $d(m,v)\leq n$ is in~$\mathcal{S}$. Let $d(m,v)=j\leq n$, and suppose that $v\notin\mathcal{S}$. Suppose first that $j<n$. Choose an element $u$ with $d(m,u)=n+j$ and $d(v,u)=n$. By the previous paragraph $u\notin \mathcal{S}$, and so from the definition of  ovoidal subspaces there is an element $x'\in\mathcal{S}$ with $d(u,x')\leq n$. Since $d(x',m)+d(m,u)+d(u,x')\leq n+n+j+n<4n$, the union of the geodesics joining $x'$ and $m$, and $m$ and $u$, and $u$ and $x'$, is not a cycle, and thus is a tree with one branch point. Since $d(u,x')<d(u,v)$ and $d(m,x')\geq d(m,v)$ we see that $v$ lies on the shortest path from $m$ to $x'$, and hence $v\in\mathcal{S}$, a contradiction. If $j=n$ then $v$ is a point, and by the previous argument the line $\mathrm{proj}_mv$ is in~$\mathcal{S}$. Thus, by the definition of a subspace, $v\in\mathcal{S}$ as well.
\end{proof}

Using the classification of ovoidal subspaces and the eigenvalue techniques developed in Appendix~\ref{app:A} we can give some more precise information for point-domestic and line-domestic collineations of generalised hexagons and octagons. Recall that a \textit{central collineation} of a generalised $2n$-gon is a collineation whose fixed element structure equals the ball of radius $n$ centred at a point.

Consider generalised hexagons. By \cite{off:02} distance~$3$-ovoids in generalised hexagons only exist for parameters $s=t$. It is shown in \cite[Corollary~5.8]{TTM:09} that no nontrivial automorphism of a finite generalised hexagon with parameters $(t,t)$ with $t\equiv0\mod 3$ can fix a distance $3$-ovoid. By \cite{tha:76} the parameters of any large full subhexagon $\Gamma'$ are $(s,\sqrt{t/s})$, and so in particular $t\geq s$. In regards to central collineations, we have the following proposition. 

\begin{prop}
Let $\Gamma$ be a finite thick generalised hexagon with parameters $(s,t)$. If $\Gamma$ admits a central collineation then $t^3(st+s^2+1)/(s^2+st+t^2)$ is an integer. In particular, if $s=t^3$ then $\Gamma$ does not admit a central collineation.
\end{prop}

\begin{proof}
If $\Gamma$ admits a central collineation $\theta$ then it is elementary that $\theta$ has $(t+1)(st+1)$ fixed lines, and that $\theta$ maps $(t+1)s^2t^2$ lines to distance~$4$. Using these values in~(\ref{eq:nhex1}) (see Appendix~\ref{app:A}) we compute $n_1=t^3(st+s^2+1)/(s^2+st+t^2)$. In particular, if $s=t^3$ then $n_1=t^3-t(t^2-1)/(t^4+t^2+1)$, which is never an integer for $t\geq 2$.  
\end{proof}

In particular, the above observations, together with Theorem~\ref{thm:1b}, establish that generalised hexagons with parameters $(t^3,t)$ do not admit point-domestic automorphisms. 

Now consider generalised octagons. 

\begin{prop} Large full proper suboctagons do not exist in finite thick generalised octagons.
\end{prop}

\begin{proof}
Let $\Gamma$ be a finite thick generalised octagon with parameters $(s,t)$. Suppose that $\Gamma'$ is a large full suboctagon with parameters $(s,t')$. From \cite[Proposition~1.8.8]{mal:98} we have $t'=1$. The number of lines in $\Gamma'$ is $2(s^3+s^2+s+1)$. The number of lines at distance $2$ from some line in $\Gamma'$ is $(s+1)(s^3+s^2+s+1)(t-1)$, and the number of lines at distance $4$ from some line in $\Gamma'$ is $(s+1)(s^3+s^2+s+1)(t-1)st$. If $\Gamma'$ is large then these lines exhaust all lines of $\Gamma$, and so
$$
(s^3+s^2+s+1)(2+(s+1)(t-1)+(s+1)(t-1)st)=(t+1)(s^3t^3+s^2t^2+st+1).
$$
Thus $s^2(s-t)(t-1)(st^2+s^2t+2st+s+t)=0$, a contradiction as $s,t\geq 2$.
\end{proof}

In regards to distance $4$-ovoids, we note that no such ovoid has been found in the Ree-Tits octagons, and one may conjecture that none exist in any finite thick octagons. Indeed \cite[Theorem~1]{CM:00} implies that no distance $4$-ovoid exists in the Ree-Tits octagon with parameters $(2,4)$. Thus the most likely possibility for a line-domestic collineation is a central collineation. In this direction we give the following proposition.

\begin{prop}
Let $\Gamma$ be a finite thick generalised octagon with parameters $(s,t)$. If $\Gamma$ admits a central collineation then $t^4(s-1)(st+s^2+s+1)/(s^3+st^2+s^2t+t^3)$ is an integer. In particular, if $s=t^2$ then $\Gamma$ does not admit a central collineation. If $(s,t)=(2t,t)$ then $t\equiv 0,3,4,5,8,9,10,13,14\mod 15$, and if $(s,t)=(s,2s)$ then $s\equiv 0,1,2,5,6,7,10,11,12\mod 15$.  \end{prop}

\begin{proof}
If $\Gamma$ admits a central collineation $\theta$ then it is elementary that there are $(t+1)(st+1)$ fixed lines, $(t+1)s^2t^2$ lines mapped to distance~$2$, and $(t+1)s^3t^3$ lines mapped to distance~$4$. Using these values in~(\ref{eq:noct1}) gives
$n_1=-t^4(s-1)(st+s^2+s+1)/(s^3+st^2+s^2t+t^3)$. The remaining claims easily follow.
\end{proof}

\section{Exceptional domestic collineations}\label{sect:5}

We now begin the classification of exceptional domestic collineations of generalised $2n$-gons. In this section we prove some general lemmas, before specialising to the case of quadrangles, hexagons and octagons in the following sections. 

Let $\theta$ be a collineation of a generalised $2m$-gon. For $k=0,1,2,\ldots,m$ let
\begin{align}\label{eq:pointlinesets}
\mathcal{P}_{2k}(\theta)=\{p\in\cP\mid d(p,p^{\theta})=2k\}\qquad\textrm{and}\qquad \mathcal{L}_{2k}(\theta)=\{L\in\cL\mid d(L,L^{\theta})=2k\}.
\end{align}
We will usually suppress the dependence on $\theta$, and simply write $\mathcal{P}_{2k}=\mathcal{P}_{2k}(\theta)$ and $\mathcal{L}_{2k}=\mathcal{L}_{2k}(\theta)$.

The following simple argument shows that exceptional domestic collineations are necessarily rather rare for finite generalised $2m$-gons.

\goodbreak

\begin{lemma}\label{lem:Cgen}
Let $\Gamma$ be a finite thick generalised $2m$-gon with parameters $(s,t)$, and suppose that~$\Gamma$ admits an exceptional domestic collineation~$\theta$. Then
\begin{enumerate}
\item If $m=2$ then $(s,t)$ or $(t,s)$ is in $\{(2,2),(4,2),(3,3),(5,3),(6,3),(9,3),(4,4)\}$.
\item If $m=3$ then $(s,t)$ or $(t,s)$ is in $\{(2,2),(8,2),(3,3),(27,3),(4,4),(5,5),(6,6),(7,7)\}$.
\item If $m=4$ then $(s,t)$ or $(t,s)$ is in $\{(4,2),(6,3),(8,4),(10,5),(12,6),(14,7)\}$.
\end{enumerate}
\end{lemma}

\begin{proof}
We will prove the $m=3$ case. The $m=2$ and $m=4$ cases are analogous. Suppose that $\theta$ is a domestic collineation, and let $\mathcal{P}_0,\mathcal{P}_2,\mathcal{P}_4,\mathcal{P}_6$ and $\mathcal{L}_0,\mathcal{L}_2,\mathcal{L}_4,\mathcal{L}_6$ be the sets defined in~(\ref{eq:pointlinesets}). Since $\theta$ is exceptional domestic there exists a line $L\in \mathcal{L}_{6}$. Since $\theta$ maps no chamber to an opposite, all $s+1$ points in $\Gamma_1(L)$ are in $\mathcal{P}_{4}$. Then at least $(s+1)(t-2)$ of the lines in $\Gamma_2(L)$ are in $\mathcal{L}_{6}$, and thus at least $(s+1)(t-2)s$ of the points in $\Gamma_3(L)$ are in $\mathcal{P}_{4}$. Then at least $(s+1)(t-2)^2s$ lines in $\Gamma_4(L)$ are in $\mathcal{L}_6$, and at least $(s+1)(t-2)^2s^2$ points in $\Gamma_5(L)$ are in $\mathcal{P}_4$. Since $\mathcal{P}=\Gamma_1(L)\cup\Gamma_3(L)\cup\Gamma_5(L)$ we see that
$$
|\mathcal{P}_{4}|\geq(s+1)(1+s(t-2)+s^2(t-2)^2).
$$

On the other hand, since $\theta$ is exceptional domestic there exists a line $L'\in \mathcal{L}_{4}$ (for otherwise $\mathcal{P}_{6}=\emptyset$). At least $s-1$ of the points in $\Gamma_1(L')$ are in $\mathcal{P}_{6}$, and at least $(s-1)t$ of the lines in $\Gamma_2(L')$ are in $\mathcal{L}_{4}$. Then at least $(s-1)t(s-2)$ points in $\Gamma_3(L')$ are in $\mathcal{P}_{6}$, and at least $(s-1)t^2(s-2)$ of the lines in $\Gamma_4(L')$ are in $\mathcal{L}_{4}$. Then at least $(s-1)t^2(s-2)^2$ of the points in $\Gamma_5(L')$ are in $\mathcal{P}_6$, and so
$$
|\mathcal{P}_{6}|\geq (s-1)(1+(s-2)t+(s-2)^2t^2).
$$

But $|\mathcal{P}_{4}|+|\mathcal{P}_{6}|\leq |\cP|=(s+1)(s^2t^2+st+1)$, and thus
$$
(s+1)(1+s(t-2)+s^2(t-2)^2)+(s-1)(1+(s-2)t+(s-2)^2t^2)-(s+1)(s^2t^2+st+1)\leq 0.
$$
Repeating the argument counting from a point $p\in \mathcal{P}_6$ and then a point $p'\in \mathcal{P}_4$ we see that the dual inequality (with $s$ and $t$ interchanged) also holds. Thus we may assume that $s\geq t$. It is clear that the inequality fails for large values of $s$ and $t$, because the leading term on the left hand side is $+s^3t^2$. Indeed the inequality fails if $s,t\geq 8$ (write $s=x+8$ and $t=y+8$ and observe that the left hand side of the inequality is a polynomial in $x$ and $y$ with nonnegative coefficients). It remains to consider the cases $2\leq t\leq 7$. If $t=7,6,5,4$ then the inequality forces $s< 10,11,14,22$ (respectively). Combining this with the known constraints on the parameters of a hexagon from Section~\ref{sect:2} proves the lemma in this case. 
\end{proof}

\begin{remark}
The possibility $(s,t)=(6,3)$ in the  $m=2$ case of Lemma~\ref{lem:Cgen} can be removed, since it is known that no quadrangle with parameters $(6,3)$ exists, see \cite{DZ:76} or \cite[6.2.2]{PT:09}. However this case is easily eliminated in our later arguments anyway.
\end{remark}

\begin{lemma}\label{lem:Xgen}
Let $\Gamma$ be a generalised $2m$-gon, and let $\theta$ be a domestic collineation of~$\Gamma$. If there exist $p_0\in \mathcal{P}_0$ and $L_0\in \mathcal{L}_0$ with $d(p_0,L_0)=2m-1$ then either all lines through $p_0$ are fixed by~$\theta$, or all points on $L_0$ are fixed by~$\theta$.
\end{lemma}

\begin{proof}
Suppose that there is a line $L\bI p_0$ with $L^{\theta}\neq L$ and a point $p\bI  L_0$ with $p^{\theta}\neq p$. Since $d(p_0,p)=2m$ we can choose a geodesic $p_0\bI L\bI x_2\bI x_3\bI\cdots\bI x_{2m-1}\bI p$ from $p_0$ to $p$ passing through~$ L$ (the $x_j$ with $j$ even are points, and the $x_j$ with $j$ odd are lines). But then the chamber $\{x_m,x_{m+1}\}$ is mapped to an opposite chamber, because we have the geodesic paths 
\begin{align*}
&x_m\bI x_{m-1}\bI\cdots \bI x_2\bI L \bI p_0\bI L^{\theta}\bI x_2^{\theta}\bI\cdots\bI x_m^{\theta},\\
&x_{m+1}\bI x_{m+2}\bI\cdots\bI x_{2m-1}\bI p\bI  L_0\bI p^{\theta}\bI x_{2m-1}^{\theta}\bI\cdots\bI x_{m+1}^{\theta},
\end{align*}
a contradiction.
\end{proof}

Let $\theta$ be a collineation of a generalised $2m$-gon, and let $\mathcal{P}_{2k}=\cP_{2k}(\theta)$ and $\mathcal{L}_{2k}=\cL_{2k}(\theta)$ be the sets from~(\ref{eq:pointlinesets}). We further decompose these sets as follows. If $p\in \mathcal{P}_{2k}$ with $1\leq k\leq m-1$ then there is a unique geodesic
$
p=x_0\bI x_1\bI x_2\bI\cdots\bI x_{2k-1}\bI x_{2k}=p^{\theta}
$
joining $p$ to $p^{\theta}$. For $r=1,2,\ldots,k+1$, let
$$
\mathcal{P}_{2k}^{r}=\{p\in \mathcal{P}_{2k}\mid \textrm{$x_{j}^{\theta}=x_{2k-j}$ for $0\leq j\leq k+1-r$ and $x_j^{\theta}\neq x_{2k-j}$ for $k+1-r<j\leq k$}\},
$$
with the obvious dual definition for $\mathcal{L}_{2k}^r$. For example, for generalised octagons ($m=4$) we have
\begin{align*}
\mathcal{P}_6^1&=\{p\in \mathcal{P}_6\mid p\bI L_1\bI p_2\bI L_3\bI p_4\bI L_5\bI p^{\theta}\textrm{ with $L_1^{\theta}= L_5$ and $p_2^{\theta}= p_4$ and $L_3^{\theta}= L_3$}\}\\
\mathcal{P}_6^2&=\{p\in \mathcal{P}_6\mid p\bI L_1\bI p_2\bI L_3\bI p_4\bI L_5\bI p^{\theta}\textrm{ with $L_1^{\theta}= L_5$ and $p_2^{\theta}= p_4$ and $L_3^{\theta}\neq L_3$}\}\\
\mathcal{P}_6^3&=\{p\in \mathcal{P}_6\mid p\bI L_1\bI p_2\bI L_3\bI p_4\bI L_5\bI p^{\theta}\textrm{ with $ L_1^{\theta}= L_5$ and $p_2^{\theta}\neq p_4$ and $L_3^{\theta}\neq L_3$}\}\\
\mathcal{P}_6^4&=\{p\in \mathcal{P}_6\mid p\bI L_1\bI p_2\bI L_3\bI p_4\bI L_5\bI p^{\theta}\textrm{ with $L_1^{\theta}\neq L_5$ and $p_2^{\theta}\neq p_4$ and $L_3^{\theta}\neq L_3$}\}.
\end{align*}
Then for $1\leq k\leq m-1$ we have
$$
\mathcal{P}_{2k}=\mathcal{P}_{2k}^1\cup \mathcal{P}_{2k}^2\cup\cdots\cup \mathcal{P}_{2k}^{k+1}\qquad\textrm{and}\qquad \mathcal{L}_{2k}=\mathcal{L}_{2k}^1\cup \mathcal{L}_{2k}^2\cup\cdots\cup \mathcal{L}_{2k}^{k+1}
$$
where the union is disjoint.

It is helpful to keep the following elementary observations in mind. Let $\theta$ be a collineation of a  generalised $2m$-gon with parameters $(s,t)$, and let $1\leq k\leq m-1$.
\begin{enumerate}
\item If $p\in \mathcal{P}_{2k}^r$ with $1\leq r\leq k$ then exactly $1$ of the lines through $p$ is in $\mathcal{L}_{2k-2}^r$, and the remaining $t$ lines through $p$ are in $\mathcal{L}_{2k+2}^r$.
\item If $p\in \mathcal{P}_{2k}^{k+1}$ then exactly $2$ of the lines through $p$ are in $\mathcal{L}_{2k}^{k+1}$, and the remaining $t-2$ lines through $p$ are in $\mathcal{L}_{2k+2}^{k+1}$.
\item If $\theta$ is domestic, and if $p\in\cP_{2m}$, then all $t+1$ lines through $p$ are in $\cL_{2m-2}$. 
\end{enumerate}
(Here we interpret $\mathcal{L}_{2m}^r=\mathcal{L}_{2m}$ and $\mathcal{L}_0^r=\mathcal{L}_0$ for all $r$, and the dual statements also hold.)

The following formulae are a main ingredient in our arguments.

\begin{prop}\label{prop:Agen}
Let $\Gamma$ be a thick finite generalised $2m$-gon with parameters $(s,t)$. The following formulae hold for any collineation $\theta$ of~$\Gamma$, where $1\leq k<m$.
\begin{align*} 
|\mathcal{P}_{2k}^r|&=\begin{cases}s|\mathcal{L}_{2k-2}^r|&\textrm{if $1\leq r< k$}\\
(s-1)|\mathcal{L}_{2k-2}^r|&\textrm{if $r=k$}\\
|\mathcal{L}_{2k}^r|&\textrm{if $r=k+1$.}
\end{cases}
\end{align*}
If $\theta$ is domestic then also
\begin{align*} 
(t+1)|\mathcal{P}_{2m}|&=s|\mathcal{L}_{2m-2}^1|+s|\mathcal{L}_{2m-2}^2|+\cdots+s|\mathcal{L}_{2m-2}^{m-1}|+(s-1)|\mathcal{L}_{2m-2}^m|.
\end{align*}
The dual formulae (with $P$ and $L$ interchanged and $s$ and $t$ interchanged) also hold. 
\end{prop}

\begin{proof}
Suppose that $1\leq r\leq k$. Let $X=\{(p, L)\in \mathcal{P}_{2k}^r\times \mathcal{L}_{2k-2}^r\mid p\bI L\}$. We count the cardinality of~$X$ in two ways. Firstly, suppose that $p\in \mathcal{P}_{2k}^r$. Then exactly~$1$ of the $t+1$ lines $ L\bI p$ is in $\mathcal{L}_{2k-2}^r$, and so $|X|=|\mathcal{P}_{2k}^r|$. On the other hand, suppose that $ L\in \mathcal{L}_{2k-2}^r$. If $1\leq r\leq k-1$ then exactly $s$ of the $s+1$ points $p\bI  L$ are in $\mathcal{P}_{2k}^r$, and if $r=k$ then exactly $s-1$ of these points are in $\mathcal{P}_{2k}^r$. Therefore $|\mathcal{P}_{2k}^r|=|X|=s|\mathcal{L}_{2k-2}^r|$ if $1\leq r<k$, and $|\mathcal{P}_{2k}^r|=|X|=(s-1)|\mathcal{L}_{2k-2}^r|$ if $r=k$. 

Now suppose that $r=k+1$. Let $Y=\{(p, L)\in \mathcal{P}_{2k}^{k+1}\times \mathcal{L}_{2k}^{k+1}\mid p\bI L\}$. If $p\in \mathcal{P}_{2k}^{k+1}$ then exactly~$2$ of the lines $ L\bI p$ are in $\mathcal{L}_{2k}^{k+1}$, and so $|Y|=2|\mathcal{P}_{2k}^{k+1}|$. The dual count gives $|Y|=2|\mathcal{L}_{2k}^{k+1}|$, and so $|\mathcal{P}_{2k}^{k+1}|=|\mathcal{L}_{2k}^{k+1}|$.

Finally, let $Z=\{(p, L)\in \mathcal{P}_{2m}\times \mathcal{L}_{2m-2}\mid p\bI L\}$. If $p\in \mathcal{P}_{2m}$ then all $t+1$ of the lines $ L\bI p$ are in $\mathcal{L}_{2m-2}$ (since $\theta$ maps no chamber to an opposite chamber). Therefore $|Z|=(t+1)|\mathcal{P}_{2m}|$. On the other hand, if $p\in \mathcal{L}_{2m-2}^r$ and $1\leq r<m$ then $s$ of the points $p\bI  L$ are in $\mathcal{P}_{2m}$, while if $r=m$ then $s-1$ of the points $p\bI L$ are in~$\mathcal{P}_{2m}$. Thus $|Z|=s|\mathcal{L}_{2m-2}^1|+\cdots+s|\mathcal{L}_{2m-2}^{m-1}|+(s-1)|\mathcal{L}_{2m-2}^m|$, completing the proof.
\end{proof}

\section{Exceptional domestic collineations of finite quadrangles}\label{sect:6a}

In this section we sketch the classification of exceptional domestic collineations of finite thick generalised quadrangles. These results are contained in \cite{TTM:13}, however we include an exposition here using the methods that we will use in this paper for the higher girth cases. This is partly for completeness, but more importantly because the simpler case of quadrangles serves as a helpful illustration of our techniques.

Let $\Gamma$ be a finite thick generalised quadrangle of order $(s,t)$, and let $\theta$ be a domestic collineation of~$\Gamma$. From Proposition~\ref{prop:Agen} we have the formulae
\begin{align}\label{eq:qA}
|\cP_2^2|&=|\cL_2^2|,&
(t+1)|\cP_4|&=s|\cL_2^1|+(s-1)|\cL_2^2|,&
(s+1)|\cL_4|&=t|\cP_2^1|+(t-1)|\cP_2^2|.
\end{align}

\begin{lemma}\label{lem:qB} Let $\theta$ be an exceptional domestic collineation of a finite thick quadrangle~$\Gamma$.
\begin{itemize}
\item[\emph{1.}] If $\theta$ has no fixed elements then $s+t\mid st+1$. In particular $s$ and $t$ are relatively prime.
\item[\emph{2.}] The collineation $\theta$ has a fixed point if and only if it has a fixed line.
\item[\emph{3.}] The fixed element structure of $\theta$ is a (possibly empty) tree of diameter at most $2$ in the incidence graph.
\end{itemize}
\end{lemma}

\begin{proof}
1. If $\theta$ has no fixed elements then $|\cP_0|=|\cL_0|=|\cP_2^1|=|\cL_2^1|=0$. Then~(\ref{eq:qA}) gives $(t+1)|\cP_4|=(s-1)|\cP_2^2|$, and since the total number of points is $(s+1)(st+1)$ we also have $(s+1)(st+1)=|\cP_2^2|+|\cP_4|$. Solving these equations for $|\cP_4|$ gives $|\cP_4|=(s^2-1)(st+1)/(s+t)$. By the divisibility conditions in Section~\ref{sect:2} we know that $s^2(st+1)/(s+t)$ is an integer, and thus $(st+1)/(s+t)$ is an integer as well.

2. Suppose that $\theta$ has a fixed point $p_0\in\cP_0$, but no fixed lines. Thus all lines in $\Gamma_1(p_0)$ are in $\cL_2^1$, and all points in $\Gamma_2(p_0)$ are in $\cP_4$. So, by domesticity, all lines in $\Gamma_3(p_0)$ are in $\cL_2$. In particular $\cL_4=\emptyset$, a contradiction. The dual argument also applies.

3. The fixed element structure is either a subquadrangle, or is a (possibly empty) tree of diameter at most~$4$ in the incidence graph (note that the previous part eliminates the possibility of a set of mutually opposite points or lines). Suppose that the fixed element structure is a subquadrangle~$\Gamma'$. Then either $\Gamma'$ has parameters $(s',t')$, or $\Gamma'$ does not have parameters (in this case $\Gamma'$ is a \textit{weak} generalised quadrangle, c.f. \cite{mal:98}). Suppose first that $\Gamma'$ has parameters. Then by Lemma~\ref{lem:Xgen} $\Gamma'$ is either ideal or full. Suppose that it is ideal (the dual argument applies if it is full). Let $L_0\in \cL_0$. Then all points of $\Gamma_1(L_0)$ are in $\cP_0\cup\cP_2^1$, and by ideal-ness all lines of $\Gamma_2(L_0)$ are in $\cL_0\cup \cL_4$. Thus, by domesticity, all points of $\Gamma_3(L_0)$ are in $\cP_0\cup\cP_2$, and so in particular $\cP_4=\emptyset$, a contradiction. Now suppose that $\Gamma'$ does not have parameters. Then by \cite[Theorem~1.6.2]{mal:98} $\Gamma'$ is the double of a digon. But in this case there are elements $p, L$ with $d(p, L)=3$ in the fixed element structure such that neither all lines through $p$ are fixed, nor all points on $ L$ are fixed, contradicting Lemma~\ref{lem:Xgen}.

Thus the fixed element structure is a (possibly empty) tree of diameter at most~$4$. By Lemma~\ref{lem:Xgen} the fixed element structure cannot have diameter~$3$, so suppose that the diameter equals~$4$. Then the tree has a unique centre $x\in\cP\cup\cL$, and up to duality we may suppose that $x\in\cP$. If $L$ is a fixed line through $x$ then by Lemma~\ref{lem:Xgen} all points on $L$ are fixed. Thus the lines of $\Gamma_1(x)$ are in $\cL_0\cup\cL_2$, the points of $\Gamma_2(x)$ are in $\cP_0\cup\cP_4$, and so by domesticity the lines of $\Gamma_3(x)$ are in $\cL_0\cup \cL_2$. In particular $\cL_4=\emptyset$, a contradiction. 
\end{proof}

\begin{thm}\label{thm:mainquad}
The classification of exceptional domestic collineations of thick finite generalised quadrangles is as claimed in Corollary~\emph{\ref{cor:4}}.
\end{thm}

\begin{proof} If there are no fixed elements then Lemmas~\ref{lem:Cgen} and~\ref{lem:qB} imply that $(s,t)\in\{(3,5),(5,3)\}$. So suppose that there are fixed elements. Then Lemma~\ref{lem:qB} implies that there are both fixed points and fixed lines, and that the fixed element structure is a tree of diameter at most~$2$ in the incidence graph. We may suppose, up to duality, that the fixed element structure consists of a single fixed line $L_0$ with $a+1$ fixed points on it (with $a\geq 0$). Let $p_0$ be one of these fixed points.

In $\Gamma_1(L_0)$ there are $a+1$ points in $\cP_0$, and the remaining $s-a$ points are in $\cP_2^1$. In $\Gamma_2(L_0)$ there are $t(a+1)$ lines in $\cL_2^1$, and the remaining $t(s-a)$ lines are in~$\cL_4$. Then in $\Gamma_3(L_0)$ there are $st(a+1)$ points in $\cP_4$, $\alpha$ points in $\cP_2^2$, and the remaining $st(s-a)-\alpha$ points are in $\cP_2^1$, where $\alpha\geq 0$ is an unknown integer (to be determined later). Therefore
\begin{align*}
|\cP_0|&=a+1,&|\cP_2^1|&=(st+1)(s-a)-\alpha,&|\cP_2^2|&=\alpha,&|\cP_4|&=st(a+1).
\end{align*}

On the other hand, in $\Gamma_1(p_0)$ we have 1 line in $\cL_0$, and the remaining $t$ lines are in $\cL_2^1$. Then in $\Gamma_2(p_0)$ we have $a$ points in $\cP_0$, $s-a$ points in $\cP_2^1$, and the remaining $st$ points are in $\cP_4$. Then in $\Gamma_3(p_0)$ we have $at$ lines in $\cL_2^1$, $(s-a)t$ lines in $\cL_4$, $\alpha'$ lines in $\cL_2^2$, and a further $st^2-\alpha'$ lines in $\cL_2^1$, where $\alpha'$ is an unknown integer. Therefore
\begin{align*}
|\cL_0|&=1&|\cL_2^1|,&=t(a+1)+st^2-\alpha',&|\cL_2^2|&=\alpha',&|\cL_4|&=(s-a)t.
\end{align*} 

Since $|\cP_2^2|=|\cL_2^2|$ we conclude that $\alpha'=\alpha$. Then the formulae~(\ref{eq:qA}) give 2 equations in the unknowns $a$ and $\alpha$, and solving gives $a=s-t$ and $\alpha=st^2(t-1)$. Therefore $s\geq t$, and since $|\cP_2^1|=2st^2+t-st^3\geq 0$ we conclude from Lemma~\ref{lem:Cgen} that $(s,t)$ equals $(2,2)$ or $(4,2)$.

It remains to prove existence and uniqueness of exceptional domestic collineations in quadrangles with parameters $(2,2),(2,4),(4,2),(3,5)$, or $(5,3)$. For each of these parameter values it is well known \cite[Chapter~6]{PT:09} that there exists a unique quadrangle, and explicit constructions are available. The result follows by direct examination of each case. Full details can be found in \cite[Section~4]{TTM:13}. 
\end{proof}

\section{Exceptional domestic collineations of finite hexagons}\label{sect:6}

In this section we prove Theorem~\ref{thm:2} (the classification of exceptional domestic collineations of finite thick generalised hexagons). Let $\Gamma=(\mathcal{P},\mathcal{L},\mathbf{I})$ be a finite thick generalised hexagon with parameters $(s,t)$, and let $\theta:\Gamma\to\Gamma$ be a domestic collineation of~$\Gamma$. By Proposition~\ref{prop:Agen} the following formulae hold (with only the last two formulae requiring domesticity):
\begin{align}\label{eq:hexformulae}
\begin{aligned}
|\mathcal{P}_2^2|&=|\mathcal{L}_2^2|& |\mathcal{P}_4^3|&=|\mathcal{L}_4^3|\\
|\mathcal{P}_4^1|&=s|\mathcal{L}_2^1|& |\mathcal{L}_4^1|&=t|\mathcal{P}_2^1|\\
|\mathcal{P}_4^2|&=(s-1)|\mathcal{L}_2^2|& |\mathcal{L}_4^2|&=(t-1)|\mathcal{P}_2^2|\\
(t+1)|\mathcal{P}_6|&=s|\mathcal{L}_4^1|+s|\mathcal{L}_4^2|+(s-1)|\mathcal{L}_4^3|& (s+1)|\mathcal{L}_6|&=t|\mathcal{P}_4^1|+t|\mathcal{P}_4^2|+(t-1)|\mathcal{P}_4^3|.
\end{aligned}
\end{align}

\begin{lemma}\label{lem:B}
If $\theta$ is an exceptional domestic collineation of a finite thick generalised hexagon then $\theta$ has at least one fixed element. 
\end{lemma}

\begin{proof}
If $|\mathcal{P}_0|=0$ and $|\mathcal{L}_0|=0$ then $|\mathcal{P}_2^1|=|\mathcal{P}_4^1|=|\mathcal{L}_2^1|=|\mathcal{L}_4^1|=0$. Thus by (\ref{eq:hexformulae}) we have the equation $(t+1)|\mathcal{P}_6|=s(t-1)|\mathcal{P}_2^2|+(s-1)|\mathcal{P}_4^3|$. From the total point sum formula we have the equation 
$$
(s+1)(s^2t^2+st+1)=|\mathcal{P}_0|+|\mathcal{P}_2^1|+|\mathcal{P}_2^2|+|\mathcal{P}_4^1|+|\mathcal{P}_4^2|+|\mathcal{P}_4^3|+|\mathcal{P}_6|=s|\mathcal{P}_2^2|+|\mathcal{P}_4^3|+|\mathcal{P}_6|.
$$
Eliminating $|\mathcal{P}_4^3|$ from these equations gives
$$
|\mathcal{P}_6|=\frac{s(t-s)|\mathcal{P}_2^2|+(s^2-1)(s^2t^2+st+1)}{s+t}.
$$
It follows that $s$ and $t$ are relatively prime, contradicting Lemma~\ref{lem:Cgen}.
\end{proof}

\begin{lemma}\label{lem:X} Let $\Gamma$ be a generalised hexagon, and let $\theta:\Gamma\to\Gamma$ be a domestic collineation. If $\theta$ has a fixed point and no fixed lines then
$\mathcal{L}_0$, $\mathcal{L}_2^2$, $\mathcal{L}_4^1$, $\mathcal{L}_4^2$, $\mathcal{P}_2^1$, $\mathcal{P}_2^2$, and $\mathcal{P}_4^2$ are empty. The dual statement also holds.
\end{lemma}

\begin{proof}
If there are no fixed lines then $\mathcal{L}_0$, $\mathcal{P}_2^1$, and $\mathcal{L}_4^1$ are empty by their definitions. We now show that $\mathcal{P}_2^2$ is empty, from which the remaining claims follow. Let $p_0\in \mathcal{P}_0$, and suppose that $p\in \mathcal{P}_2^2$. If $d(p_0,p)=2$ then the points $p_0,p,\theta(p)$ form the vertices of a non-degenerate triangle, and similarly if $d(p_0,p)=4$ we obtain a pentagon. Therefore $d(p_0,p)=6$. There are exactly $2$ lines through $p$ which are in $\mathcal{L}_2^2$, and the remaining $t-1$ lines are in $\mathcal{L}_4^2$. Choose $ L_3\bI p$ such that $ L_3\in \mathcal{L}_4^2$, and let $p_0\bI  L_1\bI p_1\bI L_2\bI p_2\bI L_3\bI p$ be the unique geodesic from $p_0$ to $p$ passing through~$ L_3$. Since $ L_3\in \mathcal{L}_4^2$ and $p_2\bI L_3$ we have $p_2\in \mathcal{P}_2^2\cup \mathcal{P}_6$, but since $d(p_0,p_2)=4$ we have $p_2\notin \mathcal{P}_2^2$ by the above argument. Thus $p_2\in \mathcal{P}_6$. Also, since there are no fixed lines we have $ L_1\in \mathcal{L}_2^1$, and so $p_1\in \mathcal{P}_4^1$ and $ L_2\in \mathcal{L}_6$. Therefore the chamber $\{p_2, L_2\}$ is mapped to an opposite chamber, a contradiction.
\end{proof}

\begin{lemma}\label{lem:hexP_2} Let $\Gamma$ be a generalised hexagon, and let $\theta:\Gamma\to\Gamma$ be any collineation (not necessarily domestic). If all points of a line are fixed then $\mathcal{P}_2^2$ is empty (and hence $\mathcal{L}_2^2$ is also empty). The dual statement also holds.
\end{lemma}

\begin{proof}
Suppose that all points on some line $ L_0$ are fixed, and suppose that $p\in \mathcal{P}_2^2$. Thus $d( L_0,p)\neq 1$, and if $d( L_0,p)=3$ we obtain a non-degenerate triangle, and if $d( L_0,p)=5$ we obtain a pentagon, a contradiction.
\end{proof}

\begin{lemma}\label{lem:F}
If the finite thick generalised hexagon $\Gamma$ admits an exceptional domestic collineation then $\theta$ has at least one fixed point and at least one fixed line.
\end{lemma}

\begin{proof} For the proof it is helpful to recall the observations made before Proposition~\ref{prop:Agen}. By Lemma~\ref{lem:B} there exists at least one fixed element. Suppose (up to duality) that there is a fixed point and no fixed lines.  Let $p_0\in \mathcal{P}_0$ be a fixed point. We decompose each of the sets $\Gamma_1(p_0)$, $\Gamma_3(p_0)$ and $\Gamma_5(p_0)$ into the sets $\mathcal{L}_i^j$. Since there are no fixed lines, all $t+1$ of the lines in $\Gamma_1(p_0)$ are in $\mathcal{L}_2^1$, and then all $(t+1)s$ of the points in $\Gamma_2(p_0)$ are in $\mathcal{P}_4^1$. Then all $(t+1)st$ of the lines in $\Gamma_3(p_0)$ are in $\mathcal{L}_6$, and since $\theta$ is domestic all $(t+1)s^2t$ of the points in $\Gamma_4(p_0)$ are in $\mathcal{P}_4=\mathcal{P}_4^1\cup \mathcal{P}_4^2\cup \mathcal{P}_4^3$. But $\mathcal{P}_4^2=\emptyset$ by Lemma~\ref{lem:X}. So suppose that $\alpha$ of the points in $\Gamma_4(p_0)$ are in $\mathcal{P}_4^1$ and $\beta$ of the points are in $\mathcal{P}_4^3$, with $\alpha+\beta=(t+1)s^2t$. Then there are $\alpha$ lines of $\Gamma_5(p_0)$ in $\mathcal{L}_2^1$, $(t-1)\alpha$ lines in $\mathcal{L}_6$, $2\beta$ lines in $\mathcal{L}_4^3$, and $(t-2)\beta$ lines in $\mathcal{L}_6$. 

Since $\mathcal{L}=\Gamma_1(p_0)\cup\Gamma_3(p_0)\cup \Gamma_5(p_0)$ it follows that $|\mathcal{L}_2^1|=t+1+\alpha$, $|\mathcal{L}_4^3|=2\beta$, and $|\mathcal{L}_6|=(t+1)st+(t-1)\alpha+(t-2)\beta$. Since there are no fixed lines we have $|\mathcal{L}_2^1|=(t+1)|\mathcal{P}_0|$, and therefore $\alpha=(t+1)(|\mathcal{P}_0|-1)$. Thus $\beta=(t+1)(s^2t+1-|\cP_0|)$, and using these values in the above formulae we deduce that
$|\mathcal{L}_2^2|=|\mathcal{L}_4^1|=|\mathcal{L}_4^2|=0$ and
\begin{align*}
|\mathcal{L}_2^1|&=(t+1)|\mathcal{P}_0|,&
|\mathcal{L}_4^3|&=2(t+1)(s^2t+1-|\mathcal{P}_0|),\\
|\mathcal{L}_6|&=(t+1)(s^2t^2-2s^2t+st-1+|\mathcal{P}_0|).
\end{align*}

Thus, using (\ref{eq:hexformulae}) we compute $|\mathcal{P}_2^1|=0$, $|\mathcal{P}_2^2|=0$, $|\mathcal{P}_4^1|=s|\mathcal{L}_2^1|=s(t+1)|\mathcal{P}_0|$, $|\mathcal{P}_4^2|=0$, $|\mathcal{P}_4^3|=|\mathcal{L}_4^3|=2(t+1)(s^2t+1)-2(t+1)|\mathcal{P}_0|$, and 
\begin{align*}
|\mathcal{P}_6|&=\frac{s|\mathcal{L}_4^1|+s|\mathcal{L}_4^2|+(s-1)|\mathcal{L}_4^3|}{t+1}=2(s-1)(s^2t+1)-2(s-1)|\mathcal{P}_0|.
\end{align*}
Using these values in the total point sum formula, and solving for $|\mathcal{P}_0|$, gives
$$
|\mathcal{P}_0|=1+s^2t\bigg(\frac{st-2s-t+1}{st-2t-s+1}\bigg).
$$
If $s=t$ we get $|\mathcal{P}_0|=1+s^3$, but then $|\mathcal{P}_6|=0$, a contradiction since $\theta$ is exceptional domestic. For the remaining values of $(s,t)$ from Lemma~\ref{lem:Cgen} with $s\neq t$ we see that the right hand side of the above formula for $|\mathcal{P}_0|$ is not an integer, a contradiction.
\end{proof}

\begin{lemma}\label{lem:E}
Let $\Gamma$ be a finite thick generalised hexagon, and suppose that $\theta$ is an exceptional domestic collineation of~$\Gamma$. The fixed element structure $\Gamma_{\theta}$ is a tree of diameter at most~$4$ in the incidence graph.
\end{lemma}

\begin{proof}
If the fixed element structure is a subhexagon $\Gamma'$ then either $\Gamma'$ has parameters $(s',t')$, or $\Gamma'$ does not have parameters. Suppose first that $\Gamma'$ has parameters. Then by Lemma~\ref{lem:Xgen} $\Gamma'$ is either ideal or full. Assume that $\Gamma'$ is ideal (a dual argument applies if $\Gamma'$ is full).  Suppose that $ L\in\cL$ with $ L\in \mathcal{L}_{6}$, and let $p_0\in\cP'$ be a fixed point. Since $\Gamma'$ is ideal we have $d(p_0, L)=5$ (for if $d(p_0, L)=1$ then $ L\in \mathcal{L}_0$, and if $d(p_0, L)=3$ then $ L\in \mathcal{L}_2\cup \mathcal{L}_4$). Let $p_0\bI  L_1\bI p_2\bI L_3\bI p_4\bI L$ be the unique geodesic from $p_0$ to $ L$. Since $ L_1\in \mathcal{L}_0$ we have $p_2\in \mathcal{P}_0\cup \mathcal{P}_2^1$. But if $p_2\in \mathcal{P}_0$ then $ L_3\in \mathcal{L}_0$ (since $\Gamma'$ is ideal), which contradicts $ L_5\in \mathcal{L}_6$. Therefore $p_2\in \mathcal{P}_2^1$, and so $ L_3\in \mathcal{L}_4^1$ and $p_4\in \mathcal{P}_6$. Thus the chamber $\{p_4, L\}$ is mapped to an opposite chamber, a contradiction. Now suppose that $\Gamma'$ does not have parameters. Then by \cite[Theorem~1.6.2]{mal:98} $\Gamma'$ is the triple of a digon. But in this case there are elements $p, L$ with $d(p, L)=5$ in the fixed element structure such that neither all lines through $p$ are fixed, nor all points on $ L$ are fixed, contradicting Lemma~\ref{lem:Xgen}.

Therefore $\Gamma_{\theta}$ is a tree in the incidence graph with diameter at most~$6$. Suppose that $\Gamma_{\theta}$ has diameter~$6$. Then, up to duality, we may assume that there are fixed points $p_0,p_0'\in \mathcal{P}_0$ with $d(p_0,p_0')=6$, and a line $ L_0\in \mathcal{L}_0$ with $d(p_0, L_0)=d( L_0,p_0')=3$.
We claim that every point on $ L_0$ is fixed. We may assume that $s>2$ (for otherwise $2$ of the $3$ points on $L_0$ are already fixed, and so all points on $L_0$ are fixed). Suppose that there is a point $p_1\bI  L_0$ which is not fixed. Then every line $ L_1\neq  L_0$ through $p_1$ is in $\mathcal{L}_4^1$, and every point $p_2\neq p_1$ on $ L_1$ is in $\mathcal{P}_6$. Each line $ L_2\neq  L_1$ through $p_2$ is in $\mathcal{L}_4$. There are at least $s-2$ points $p_3\neq p_2$ on $ L_2$ which are in $\mathcal{P}_6$, and so if $s>2$ we can choose $p_3\in \mathcal{P}_6$. Now, let $p_3\bI  L_3\bI p_4\bI  L_4\bI p_5\bI  L_0'$ be the geodesic from $p_3$ to $ L_0'$, where $L_0'$ is the unique line of the fixed element tree incident with~$p_0'$. Then $p_5\in \mathcal{P}_0$ (since all points of $ L_0'$ are fixed), and so $ L_4\in \mathcal{L}_2^1$ (since $ L_4$ is not fixed by the assumption on the diameter). Thus $p_4\in \mathcal{P}_4^1$, and so $ L_3\in \mathcal{L}_6$. Therefore the chamber $\{p_3, L_3\}$ is mapped to an opposite, a contradiction.

Therefore all points on $ L_0$ are fixed. Let $ L_0\bI p_1\bI L_1\bI p_2\bI  L_2\bI p_3$ be a minimal path starting at $L_0$. Then $p_1\in \mathcal{P}_0$ and $ L_1\in \mathcal{L}_0\cup \mathcal{L}_2^1$. If $ L_1\in \mathcal{L}_0$ then $p_2\in \mathcal{P}_0$ (by Lemma~\ref{lem:Xgen}), and $ L_2\in \mathcal{L}_2^1$ (by the diameter assumption), and then $p_3\in \mathcal{P}_4^1$. On the other hand, if $ L_1\in \mathcal{L}_2^1$ then $p_2\in \mathcal{P}_4^1$, and then $ L_2\in \mathcal{L}_6$, and so $p_3\in \mathcal{P}_4$ (for otherwise the chamber $\{L_2,p_3\}$ is mapped to an opposite chamber). Therefore $\mathcal{P}_6=\emptyset$, a contradiction. Hence the diameter of $\Gamma_{\theta}$ is strictly smaller than~$6$. Diameter~$5$ is impossible by Lemma~\ref{lem:Xgen}, hence the diameter of $\Gamma_{\theta}$ is at most~$4$. 
\end{proof}

\begin{lemma}\label{lem:counthex2}
Let $\theta$ be an exceptional domestic collineation of a finite thick generalised hexagon $\Gamma$. Let $p_0$ be a fixed point, and suppose that $|\Gamma_1(p_0)\cap \mathcal{L}_0|=a_1$, $|\Gamma_2(p_0)\cap \mathcal{P}_0|=a_2$, $|\Gamma_3(p_0)\cap \mathcal{L}_0|=a_3$, $|\Gamma_4(p_0)\cap \mathcal{P}_0|=a_4$, $|\Gamma_5(p_0)\cap \mathcal{L}_0|=0$, and $|\Gamma_6(p_0)\cap \mathcal{P}_0|=0$. Then
\begin{gather*}
\begin{aligned}
&\begin{aligned}
|\mathcal{L}_0|&=a_1+a_3&\qquad|\mathcal{L}_2^1|&=1+t(1+a_2+a_4)-(a_1+a_3)\\
|\mathcal{L}_2^2|&=\alpha&\qquad|\mathcal{L}_4^1|&=st(a_1+a_3)-t(a_2+a_4)\\
|\mathcal{L}_4^2|&=(t-1)\alpha
\end{aligned}\\
&\begin{aligned}
|\mathcal{L}_4^3|&=2(t+1)s^2t+a_1s^2t(t-2)-a_2st^2-(t+1)\alpha\\
|\mathcal{L}_6|&=(t+1)st(st+1-2s)-a_1st(st+1-2s)+a_2st^2-a_3st+\alpha,
\end{aligned}
\end{aligned}
\end{gather*}
where
$$
\alpha=\frac{st(s(st^2-t^2-st-2s+1)+a_1s(2s+2t-st-t^2-1)+a_2t(s+t-1)-a_3s-a_4t)}{st-t^2-s-t}.
$$ 
\end{lemma}

\begin{proof} We perform a count similar to the proof of Lemma~\ref{lem:F}. In $\Gamma_1(p_0)$ there are $a_1$ lines in $\mathcal{L}_0$ and $t+1-a_1$ lines in $\mathcal{L}_2^1$. In $\Gamma_2(p_0)$ there are $a_2$ points in $\mathcal{P}_0$, $sa_1-a_2$ points in $\mathcal{P}_2^1$, and $(t+1-a_1)s$ points in $\mathcal{P}_4^1$. Then in $\Gamma_3(p_0)$ we have $a_3$ lines in $\mathcal{L}_0$, $ta_2-a_3$ lines in $\mathcal{L}_2^1$, $(sa_1-a_2)t$ lines in $\mathcal{L}_4^1$, and $(t+1-a_1)st$ lines in $\mathcal{L}_6$. In $\Gamma_4(p_0)$ there are $a_4$ points in $\mathcal{P}_0$, $sa_3-a_4$ points in $\mathcal{P}_2^1$, $(ta_2-a_3)s$ points in $\mathcal{P}_4^1$, $(sa_1-a_2)st$ points in $\mathcal{P}_6$, $\alpha_1$ points in $\mathcal{P}_4^1$, $\alpha_2$ points in $\mathcal{P}_4^2$, and $\alpha_3$ points in $\mathcal{P}_4^3$ where $\alpha_1+\alpha_2+\alpha_3=(t+1-a_1)s^2t$. In fact $\alpha_1=0$, for otherwise there is a path $p_0\bI L_1\bI p_2\bI L_3\bI p_4$ with $ L_1\in \mathcal{L}_2^1$, $p_2\in \mathcal{P}_4^1$, $ L_3\in \mathcal{L}_6$, and $p_4\in \mathcal{P}_4^1$. But then we obtain a pentagon by taking the path from $p_0$ to $p_4$, followed by the path from $p_4$ to the fixed line on the geodesic joining $p_4$ and $\theta(p_4)$, followed by a path back to~$p_0$. 

Continuing, in $\Gamma_5(p_0)$ we have $a_4t$ lines in $\mathcal{L}_2^1$, $(sa_3-a_4)t$ lines in $\mathcal{L}_4^1$, $(ta_2-a_3)st$ lines in $\mathcal{L}_6$, $\alpha_2$ lines in $\mathcal{L}_4^2$, $(t-1)\alpha_2$ lines in $\mathcal{L}_6$, $2\alpha_3$ lines in $\mathcal{L}_4^3$, $(t-2)\alpha_3$ lines in $\mathcal{L}_6$, $\beta_1$ lines in $\mathcal{L}_4^1$, $\beta_2$ lines in $\mathcal{L}_4^2$, and $\beta_3$ lines in $\mathcal{L}_4^3$, where $\beta_1+\beta_2+\beta_3=(sa_1-a_2)st^2$. A similar argument to the one above shows that $\beta_1=0$. 

Adding together the contributions to the line sets we see that $|\mathcal{L}_0|=a_1+a_3$, $|\mathcal{L}_2^1|=1+t(1+a_2+a_4)-(a_1+a_3)$, $|\mathcal{L}_2^2|=\alpha_2$, $|\mathcal{L}_4^1|=st(a_1+a_3)-t(a_2+a_4)$, $|\mathcal{L}_4^2|=\beta_2$, $|\mathcal{L}_4^3|=2\alpha_3+\beta_3$, and $|\mathcal{L}_6|=(t+1-a_1+ta_2-a_3)st+(t-1)\alpha_2+(t-2)\alpha_3$. From~(\ref{eq:hexformulae}) we have $|\mathcal{L}_4^2|=(t-1)|\mathcal{L}_2^2|$, and so $\beta_2=(t-1)\alpha_2$. Using $\alpha_3=(t+1-a_1)s^2t-\alpha_2$ and $\beta_3=(sa_1-a_2)st^2-\beta_2$ we obtain the claimed formulae (with $\alpha=\alpha_2$). Finally we compute $\alpha$ from the equation $(s+1)|\mathcal{L}_6|=st|\mathcal{L}_2^1|+(s-1)t|\mathcal{L}_2^2|+(t-1)|\mathcal{L}_4^3|$  (see (\ref{eq:hexformulae})).
\end{proof}

\begin{lemma}\label{lem:chamber}
Let $\theta$ be any collineation of a finite thick generalised hexagon $\Gamma$ with fixed element structure being a tree in the incidence graph. Let $F_{\mathrm{opp}}$ be the set of chambers mapped to opposite chambers by~$\theta$. Then
\begin{align*}
|F_{\mathrm{opp}}|&=(s+1)(t+1)(s^2t^2+st+1)+1-|\mathcal{L}_0|-|\mathcal{P}_0|-(st+s+1)|\mathcal{L}_2^1|\\
&\quad-(st+s-t+1)|\mathcal{L}_2^2|-(s+t^{-1}+1)|\mathcal{L}_4^1|-(s+1)|\mathcal{L}_4^2|-(s+t)|\mathcal{L}_4^3|.
\end{align*}
\end{lemma}

\begin{proof}
Declare chambers $f_1=\{p_1, L_1\}$ and $f_2=\{p_2, L_2\}$ to be $p$-adjacent (written $f_1\sim_p f_2$) if $ L_1= L_2$ with $p_1\neq p_2$, and $ L$-adjacent (written $f_1\sim_{ L} f_2$) if $p_1=p_2$ with $ L_1\neq  L_2$. Let $F_0$ be the set of chambers fixed by~$\theta$, and let $F_{p L p\cdots}$ be the set of chambers $f$ for which the geodesic sequence of chambers from $f$ to $f^{\theta}$ is of the form $f\sim_p f_1\sim_{ L} f_2\sim_p\cdots \sim f^{\theta}$. It is elementary that $|F_0|=|\mathcal{L}_0|+|\mathcal{P}_0|-1$, $|F_p|=|\mathcal{P}_2^1|$, $|F_{ L}|=|\mathcal{L}_2^1|$, $|F_{p L}|=|\mathcal{L}_2^2|$, $|F_{ L p}|=|\mathcal{P}_2^2|$, $|F_{p L p}|=|\mathcal{P}_4^1|+|\mathcal{P}_4^2|$, $|F_{ L p  L}|=|\mathcal{L}_4^1|+|\mathcal{L}_4^2|$, $|F_{p L p L}|=|\mathcal{P}_4^3|$, $|F_{ L p L p}|=|\mathcal{L}_4^3|$, $|F_{p L p L p}|=s|\mathcal{L}_4^1|+s|\mathcal{L}_4^2|+(s-1)|\mathcal{L}_4^3|$ and $|F_{ L p  L p L}|=t|\mathcal{P}_4^1|+t|\mathcal{P}_4^2|+(t-1)|\mathcal{P}_4^3|$. Then $|F_{\mathrm{opp}}|$ equals the total number of chambers $(s+1)(t+1)(s^2t^2+st+1)$ minus the above values. The result follows from~(\ref{eq:hexformulae}) (recall that only the final two formulae in (\ref{eq:hexformulae}) require domesticity).
\end{proof}

\begin{proof}[Proof of Theorem~\ref{thm:2}] We first show that if $\Gamma$ admits an exceptional domestic collineation $\theta$ then $(s,t)=(2,2),(2,8)$ or $(8,2)$ and the fixed element structure of $\theta$ is as claimed in the statement of the theorem. After this we prove existence and uniqueness of exceptional domestic collineations for these hexagons.

 By Lemma~\ref{lem:F} the exceptional domestic collineation~$\theta$ has both fixed points and fixed lines. By Lemma~\ref{lem:E} the fixed element structure is a tree of diameter at most~$4$ in the incidence graph. Suppose first that $\Gamma_{\theta}$ has diameter~$1$. Then $\Gamma_{\theta}$ is a chamber $\{p_0, L_0\}$. From Lemma~\ref{lem:counthex2} (with $a_1=1$ and $a_2=a_3=a_4=0$) we get $|\mathcal{L}_0|=1$, $|\mathcal{L}_2^1|=t$, $|\mathcal{L}_2^2|=\alpha$, $|\mathcal{L}_4^1|=st$, $|\mathcal{L}_4^2|=(t-1)\alpha$, $|\mathcal{L}_4^3|=3s^2t^2-(t+1)\alpha$, and $|\mathcal{L}_6|=st^2(st+1-2s)+\alpha$, where $\alpha$ is an integer (with an explicit formula). Dually we have $|\mathcal{P}_0|=1$, $|\mathcal{P}_2^1|=s$, $|\mathcal{P}_2^2|=\alpha$, $|\mathcal{P}_4^1|=st$, $|\mathcal{P}_4^2|=(s-1)\alpha$, $|\mathcal{P}_4^3|=3s^2t^2-(s+1)\alpha$, and $|\mathcal{P}_6|=s^2t(st+1-2t)+\alpha$ (with the same $\alpha$ because $|\mathcal{P}_2^2|=\alpha=|\mathcal{L}_2^2|$). Since $|\mathcal{L}_4^3|=|\mathcal{P}_4^3|$ we have $(s-t)\alpha=0$. Therefore $\alpha=0$ or $s=t$. If $\alpha=0$ then the explicit formula for $\alpha$ gives $t=2+2/(s-2)$ and so $(s,t)=(3,4)$ or $(s,t)=(4,3)$, contradicting the divisibility conditions from Section~\ref{sect:2}. If $s=t$ then the same formula gives $\alpha=(s^3/2)(-s^2+4s-2)$, and since $\alpha\geq 0$ is an integer this forces $s=2$. But then $|\mathcal{P}_2^2|=2$, and so the two points of $\mathcal{P}_2^2$ are mapped onto one another by $\theta$, and hence the line determined by the two points is fixed, a contradiction.

Therefore $\Gamma_{\theta}$ has diameter at least~$2$. Since $\mathrm{diam}(\Gamma_{\theta})\leq 4$ we may assume up to duality that the fixed element structure consists of a fixed point $p_0\in \mathcal{P}_0$ with $|\Gamma_1(p_0)\cap \mathcal{L}_0|=A+1$ (with $A\geq 0$), $|\Gamma_2(p_0)\cap \mathcal{P}_0|=B+1$ (with $B\geq 0$) and $|\Gamma_3(p_0)\cap \mathcal{L}_0|=0$. Let $ L_0\in \Gamma_1(p_0)\cap \mathcal{L}_0$ be such that there is a fixed point $p_0'\neq p_0$ on~$ L_0$. Suppose that $ L_0$ has $C+1$ fixed points in total (and so $C\geq 1$). 

We now use Lemma~\ref{lem:counthex2} (and its dual) three times: Firstly centred at the fixed point $p_0$, then from the fixed line $ L_0$, and finally from the fixed point $p_0'$. Firstly, from Lemma~\ref{lem:counthex2} with $a_1=A+1$, $a_2=B+1$, and $a_3=a_4=0$ we obtain $|\mathcal{L}_4^1|=Ast-Bt+st$, $|\mathcal{L}_4^2|=(t-1)\alpha$, and
\begin{align}
\label{eq:l431}|\mathcal{L}_4^3|&=3s^2t^2+As^2t(t-2)-Bst^2-(t+1)\alpha,
\end{align}
and then from (\ref{eq:hexformulae}) we compute
\begin{align}
\label{eq:p61}|\mathcal{P}_6|&=\frac{s^2t(3st-3t+1)+As^2t(1+(s-1)(t-2))-Bst(st+1-t)+(t-2s+1)\alpha}{t+1}.
\end{align}
 
Applying the dual version of Lemma~\ref{lem:counthex2} based at the fixed line $L_0$, with $a_1=C$, $a_2=A$, $a_3=B-C+1$ and $a_4=0$ we obtain 
\begin{align}
\label{eq:p43}|\mathcal{P}_4^3|&=3s^2t^2-As^2t+Cst^2(s-2)-(s+1)\alpha\\
\label{eq:p62}|\mathcal{P}_6|&=s^2t(st-2t+1)+As^2t-Bst-Cst^2(s-2)+\alpha.
\end{align}
(note that we have the same $\alpha$ as in the previous count, because $|\mathcal{P}_2^2|=\alpha=|\mathcal{L}_2^2|$).

Next we perform a count centred at $p_0'$. Using Lemma~\ref{lem:counthex2} based at $p_0'$ with $a_1=1$, $a_2=C$, $a_3=A$, and $a_4=B-C+1$ we get
\begin{align}
\label{eq:l432}|\mathcal{L}_4^3|&=3s^2t^2-Cst^2-(t+1)\alpha.
\end{align}

We now make our deductions from the above formulae. From (\ref{eq:l431}) and (\ref{eq:l432}) we deduce that
\begin{align}
\label{eq:B}Bt=As(t-2)+Ct.
\end{align}
From (\ref{eq:p43}) and (\ref{eq:l432}) we see that
\begin{align}
\label{eq:A}As^2t+s\alpha=Cst^2(s-1)+t\alpha.
\end{align}
Equating the two formulae (\ref{eq:p61}) and (\ref{eq:p62}) for $|\mathcal{P}_6|$ gives a formula relating $A,B$ and $C$. Using (\ref{eq:B}) to eliminate~$B$, and equation~(\ref{eq:A}) to eliminate~$A$, we deduce that
\begin{align}
\label{eq:third}st(C-s)(st-2s-2t+2)=2\alpha.
\end{align}

Suppose that $s=t$. Then (\ref{eq:A}) gives $A=C(s-1)$. But $A\leq t=s$ (since there are $A+1$ fixed lines through~$p_0$) and so $C(s-1)\leq s$. Since $C\geq 1$ we have either $C=1$, or $C=2$ and $s=2$. But if $C=1$ then $A=s-1$, and so there are $A+1=s$ lines through $p_0$ fixed, and hence all lines through $p_0$ are fixed, and so $A+1=s+1$, a contradiction. Thus $C=2$ and $s=2$. Then $A=B=2$ and $\theta$ has the fixed element structure as in the statement of the theorem.
 
By Lemma~\ref{lem:Cgen} the only remaining cases are $(s,t)=(2,8),(8,2),(3,27),(27,3)$. Suppose that $(s,t)=(2,8)$. Then~(\ref{eq:A}) gives $64C+3\alpha=16A$. Therefore $\alpha$ is divisible by $16$, and since $A\leq t=8$ we have $64C+3\alpha\leq 128$, and so $C=1$, or $C=2$ and $\alpha=0$. Suppose that $C=1$. Then $\alpha=0$ or $\alpha=16$. If $\alpha=16$ then $A=7$, contradicting (\ref{eq:B}) (since $B$ is not integral). Hence $\alpha=0$. Then (\ref{eq:A}) gives $A=4$ and (\ref{eq:B}) gives $B=7$, but then the formula for $\alpha$ from Lemma~\ref{lem:counthex2} gives $\alpha=-448/29$, a contradiction. Therefore $C=2$ and $\alpha=0$. Then $A=8$ and so $B=13$, and so $\theta$ has the dual of the fixed element structure shown in the statement of the theorem.

Suppose that $(s,t)=(8,2)$. Then (\ref{eq:A}) gives $112C=3\alpha+64A$. Since $A\leq t=2$ we have $A=0,1,2$. But $A=1$ is impossible (because it is not possible to fix only $2$ of the $3$ lines through $p_0$) and so $A=0,2$. If $A=0$ we have $112C=3\alpha$. But (\ref{eq:third}) gives $16C+\alpha=128$, a contradiction. Similarly, if $A=2$ then $112C=3\alpha+128$, which contradicts~(\ref{eq:third}).
 
Completely analogous arguments rule out the $(3,27)$ and $(27,3)$ cases. Thus we have shown that if there exists an exceptional domestic collineation, then $(s,t)=(2,2),(2,8),(8,2)$ and the fixed element structure of $\theta$ is as claimed in the statement of the theorem. 

It is known \cite{CT:85} that for each of the parameters $(s,t)=(2,2),(2,8),(8,2)$ there is a unique generalised hexagon up to isomorphism. These hexagons are classical hexagons, and very detailed information on their automorphism groups can be found in the $\mathbb{ATLAS}$ \cite[p.14 and p.89]{atlas}. We now prove existence and uniqueness of exceptional domestic collineations with the claimed fixed element structure for these hexagons. 

Consider the $(s,t)=(2,2)$ hexagon~$\Gamma$. Analysis of the character table in the $\mathbb{ATLAS}$ \cite[p.14]{atlas} shows that there is a unique (up to conjugation) collineation $\theta$ of $\Gamma$ with fixed element structure as in the statement of the theorem (the class of type $4C$ in $\mathbb{ATLAS}$ notation). Directly from the fixed element structure (and without assuming domesticity) we compute $|\mathcal{P}_0|=3$, $|\mathcal{L}_0|=3$, $|\mathcal{L}_2^1|=4$, $|\mathcal{L}_2^2|=0$ (by Lemma~\ref{lem:hexP_2}), $|\mathcal{L}_4^1|=8$, and $|\mathcal{L}_4^2|=0$ (since $|\mathcal{L}_2^2|=0$). Let $|\mathcal{L}_4^3|=z$, and so $|\mathcal{L}_6|=63-3-4-0-8-0-z=48-z$. Using these values in~(\ref{eq:nhex1}) gives $n_1=(z-8)/12$, $n_2=(z-8)/24$, and $n_3=(24-z)/8$. Thus $z\equiv 8\mod 24$.  Since $|\mathcal{L}_6|=48-z\geq 0$ it follows that $z=8$ or $z=32$. We claim that $z=8$ is impossible. Let $L$ be the unique fixed line containing $3$ fixed points and let $p$ be the unique fixed point incident with $3$ fixed lines. Each line $M\in\cL_4^3$ has distance~$5$ from $p$ and is opposite~$L$. So if an element $g$ of the centraliser of~$\theta$ fixes~$M$ then it fixes an ordinary hexagon containing $p$ and $L$. Since $s=t=2$ the element $g$ also fixes all lines through $p$ and all points on $L$, and therefore $g=1$ by \cite[Theorem~4.4.2(v)]{mal:98}. Thus the centraliser acts semi-regularly on $\cL_4^3$. But the centraliser of $\theta$ has order~$16$ (see the $\mathbb{ATLAS}$ \cite[p.14]{atlas}), and so $|\cL_4^3|\geq 16$. It follows that $z=32$. Then by Lemma~\ref{lem:chamber} we compute $|F_{\mathrm{opp}}|=0$, and so $\theta$ is indeed domestic. It is exceptional domestic because $|\mathcal{L}_6|=|\mathcal{P}_6|=16$.

Now consider the $(s,t)=(8,2)$ case. Analysis of the character table in the $\mathbb{ATLAS}$ \cite[p.89]{atlas} shows that there is a unique (up to conjugation) collineation~$\theta$ with fixed element structure as in the statement of the theorem (the class of type $4A$ in $\mathbb{ATLAS}$ notation). 
Directly from the fixed element structure we compute $|\mathcal{P}_0|=9$, $|\mathcal{L}_0|=15$, $|\mathcal{L}_2^1|=4$, $|\mathcal{L}_2^2|=0$, $|\mathcal{L}_4^1|=224$ and $|\mathcal{L}_4^2|=0$. Write $|\mathcal{L}_4^3|=z$, and so $|\mathcal{L}_6|=576-z$. Using these values in~(\ref{eq:nhex1}) gives $n_1=(z-8)/84$, $n_2=(z+384)/112$, and $z_3=(512-z)/48$. Thus $z\equiv 176\mod 336$, and so either $z=176$ or $z=512$. A similar (although slightly more involved) argument to the $(s,t)=(2,2)$ case using a Sylow $2$-subgroup $P$ of the centraliser $C$, where $|C|=1536$ and $|P|=512$, shows that $z=176$ is impossible, and so $z=512$. Then by Lemma~\ref{lem:chamber} there are no chambers mapped to opposite chambers, and so $\theta$ is indeed domestic. It is exceptional domestic because $|\mathcal{L}_6|=64$ and $|\mathcal{P}_6|=1792$.
\end{proof}

\section{Exceptional domestic collineations of finite octagons}\label{sect:7}

In this section we prove Theorem~\ref{thm:3} (the non-existence of exceptional domestic collineations of finite thick generalised octagons). The argument is considerably more involved than the case of hexagons or quadrangles. Apart from the obvious inherent additional complexity of larger girth, there are a few other reasons for the difficulty that we face here. Firstly, in the quadrangle and hexagon cases it was possible to restrict the diameter of the fixed element tree of an exceptional domestic collineation by geometric and combinatorial arguments (restricting to diameter $2$ and $4$ respectively). However in the octagon case there appears to be no a priori reason why the diameter of the fixed element tree could not be the maximum possible diameter~$8$. This means that there are many more potential fixed element configurations to eliminate. Secondly, in the core counting arguments for quadrangles and hexagons (see Lemma~\ref{lem:counthex2} for the hexagon case) it was possible to precisely determine all cardinalities $|\cL_i^j|$ and $|\cP_i^j|$ for an exceptional domestic collineation. In contrast, in the octagon case (see Proposition~\ref{prop:maincountoct}) there are 2 degrees of freedom in the counts, since there appears to be no simple way to pin down the values of $x=|\cL_2^2|$ and $y=|\cL_4^3|$. Again, this significantly compounds the difficulty because in principle one would need to check all possible values of $x$ and $y$. Of course this is impractical to do. Fortunately the eigenvalue techniques of Appendix~\ref{app:A} can be incorporated into the arguments to drastically reduce the number of feasible values of $x$ and $y$ that need to be eliminated, however we still have much less control in the octagon case, and consequently we need to work considerably harder.

It is conjectured \cite[p.102]{kan:84} that the only finite thick generalised octagons are the Ree-Tits octagons. Up to duality, these octagons have parameters $(s,t)=(s,s^2)$ with $s=2^{2n+1}$ for some integer $n\geq 0$. If this conjecture is true, then by Lemma~\ref{lem:Cgen} exceptional domestic collineations can only exist in the Ree-Tits octagon with parameters $(s,t)=(2,4)$ (or dually, $(4,2)$). It is then possible to use the information in the $\mathbb{ATLAS}$ for the group $^2F_4(2)$ (with the help of the eigenvalue techniques in appendix~\ref{app:A}) to verify directly that no automorphism of this octagon is exceptional domestic. This would be a significant streamlining of the arguments involved in this section. However we note that the above conjecture appears to be light years away from becoming a theorem. Currently even the uniqueness of the smallest octagon (with parameters $(s,t)=(2,4)$) is unresolved (although we note the work of De Bruyn \cite{deb:13} showing that the Ree-Tits octagon associated to $^2F_4(2)$ is the unique octagon with parameters $(2,4)$ containing a thin suboctagon of order $(2,1)$).

Let $\Gamma$ be a thick generalised octagon with parameters~$(s,t)$. Suppose that $\theta$ is a domestic collineation of~$\Gamma$. By Proposition~\ref{prop:Agen} the following formulae hold, with only the last two formulae requiring domesticity:
\begin{gather}\label{eq:octformulae}
\begin{aligned}
&\begin{aligned}
|\mathcal{P}_2^2|&=|\mathcal{L}_2^2|&|\mathcal{P}_4^3|&=|\mathcal{L}_4^3|&|\mathcal{P}_6^4|&=|\mathcal{L}_6^4|&|\mathcal{P}_4^1|&=s|\mathcal{L}_2^1|&|\mathcal{P}_4^2|&=(s-1)|\mathcal{L}_2^2|\\
|\mathcal{P}_6^1|&=s|\mathcal{L}_4^1|&|\mathcal{P}_6^2|&=s|\mathcal{L}_4^2|&|\mathcal{P}_6^3|&=(s-1)|\mathcal{L}_4^3|&|\mathcal{L}_4^1|&=t|\mathcal{P}_2^1|&|\mathcal{L}_4^2|&=(t-1)|\mathcal{P}_2^2|\\
|\mathcal{L}_6^1|&=t|\mathcal{P}_4^1|&|\mathcal{L}_6^2|&=t|\mathcal{P}_4^2|&|\mathcal{L}_6^3|&=(t-1)|\mathcal{P}_4^3|
\end{aligned}\\
&\qquad\qquad\qquad\qquad\begin{aligned}
(t+1)|\mathcal{P}_8|&=s|\mathcal{L}_6^1|+s|\mathcal{L}_6^2|+s|\mathcal{L}_6^3|+(s-1)|\mathcal{L}_6^4|\\
(s+1)|\mathcal{L}_8|&=t|\mathcal{P}_6^1|+t|\mathcal{P}_6^2|+t|\mathcal{P}_6^3|+(t-1)|\mathcal{P}_6^4|
\end{aligned}
\end{aligned}
\end{gather}

\begin{lemma}\label{lem:Boct}
If $\theta$ is an exceptional domestic collineation of a finite thick generalised octagon then $\theta$ has at least one fixed element.
\end{lemma}

\begin{proof}
Suppose that $|\mathcal{P}_0|=|\mathcal{L}_0|=0$. Then $|\mathcal{P}_2^1|=|\mathcal{P}_4^1|=|\mathcal{P}_6^1|=|\mathcal{L}_2^1|=|\mathcal{L}_4^1|=|\mathcal{L}_6^1|=0$. The formulae~(\ref{eq:octformulae}) imply that $(t+1)|\mathcal{P}_8|=st(s-1)|\mathcal{P}_2^2|+s(t-1)|\mathcal{P}_4^3|+(s-1)|\mathcal{P}_6^4|$, and from the total point sum formula and (\ref{eq:octformulae}) we have have $st|\mathcal{P}_2^2|+s|\mathcal{P}_4^3|+|\mathcal{P}_6^4|+|\mathcal{P}_8|=(s+1)(s^3t^3+s^2t^2+st+1)$. Eliminating $|\mathcal{P}_6^4|$ from these equations and solving for $|\mathcal{P}_8|$ gives
$$
|\mathcal{P}_8|=\frac{s(t-s)|\mathcal{P}_4^3|+(s^2-1)(s^3t^3+s^2t^2+st+1)}{s+t}.
$$
It follows that $s$ and $t$ are relatively prime, contradicting Lemma~\ref{lem:Cgen}.
\end{proof}

\begin{lemma}\label{lem:Xoct}
Let $\Gamma$ be a generalised octagon, and let $\theta:\Gamma\to\Gamma$ be a domestic collineation. If $\theta$ has a fixed point and no fixed lines then
$\mathcal{L}_0$, $\mathcal{L}_2^2$, $\mathcal{L}_4^1$, $\mathcal{L}_4^2$, $\mathcal{L}_6^2$, $\mathcal{P}_2^1$, $\mathcal{P}_2^2$, $\mathcal{P}_4^2$, $\mathcal{P}_6^1$,  and $\mathcal{P}_6^2$
 are empty. The dual statement also holds.
\end{lemma}

\begin{proof} It is sufficient to show that $\mathcal{P}_2^1=\mathcal{P}_2^2=\emptyset$. Since there are no fixed lines we have $\mathcal{P}_2^1=\emptyset$, and so suppose that $p\in \mathcal{P}_2^2$. Let $p_0\in \mathcal{P}_0$ be a fixed point. Then $d(p_0,p)=8$ (for if $d(p_0,p)=2$ then $p_0,p,\theta(p)$ form the vertices of a nondegenerate triangle, if $d(p_0,p)=4$ we obtain a pentagon, and if $d(p_0,p)=6$ we obtain a 7-gon). Choose $ L\in \Gamma_1(p)$ with $ L\in \mathcal{L}_4^2$, and let $p_0\bI L_1\bI p_2\bI L_3\bI p_4\bI L_5\bI p_6\bI L$ be the unique geodesic from $p_0$ to $ L$. Then $p_6\in \mathcal{P}_6^2$ (for otherwise $p_6\in \mathcal{P}_2^2$, a contradiction because $d(p_0,p_6)=6$). Therefore $ L_5\in \mathcal{L}_8$ (because all lines through $p_6$ other than $ L$ are in $\mathcal{L}_8$). On the other hand, $ L_1\in \mathcal{L}_2^1$, $p_2\in \mathcal{P}_4^1$, $ L_3\in \mathcal{L}_6^1$, and so $p_4\in \mathcal{P}_8$. Therefore the chamber $\{p_4, L_5\}$ is mapped to an opposite chamber, a contradiction.
\end{proof}

\begin{lemma}\label{lem:octP_2} Let $\Gamma$ be a generalised octagon, and let $\theta:\Gamma\to\Gamma$ be any collineation. If all points of a line are fixed then $\mathcal{P}_2^2$ is empty. 
The dual statement also holds.
\end{lemma}

\begin{proof}
Analogous to the proof of Lemma~\ref{lem:hexP_2}.
\end{proof}

\begin{lemma}\label{lem:Foct}
If a finite thick generalised octagon $\Gamma$ admits an exceptional domestic collineation~$\theta$, then $\theta$ has at least one fixed point and at least one fixed line. 
\end{lemma}

\begin{proof} By Lemma~\ref{lem:Boct} $\theta$ has at least one fixed element. Suppose (up to duality) that there is a fixed element $p_0\in \mathcal{P}_0$, and suppose that there are no fixed lines. From Lemma~\ref{lem:Xoct} we have $|\mathcal{L}_2^2|=|\mathcal{L}_4^1|=|\mathcal{L}_4^2|=|\mathcal{L}_6^2|=0$. We compute the decomposition of each set $\Gamma_1(p_0)$, $\Gamma_3(p_0)$, $\Gamma_5(p_0)$ and $\Gamma_7(p_0)$ into the sets $\mathcal{L}_i^j$. Since $p\in p_0$, each of the $t+1$ lines $ L\in\Gamma_1(p_0)$ are in $\mathcal{L}_2^1$. Then each of the $(t+1)s$ points in $\Gamma_2(p_0)$ are in $\mathcal{P}_4^1$, and each of the $(t+1)st$ lines in $\Gamma_3(p_0)$ are in $\mathcal{L}_6^1$. Thus each of the $(t+1)s^2t$ points in $\Gamma_4(p_0)$ are in $\mathcal{P}_8$, and hence by domesticity each of the $(t+1)s^2t^2$ lines in $\Gamma_5(p_0)$ are in $\mathcal{L}_6^1$. Since $\mathcal{L}_6^2=\emptyset$ (Lemma~\ref{lem:Xoct}) we have that $\alpha$ of the lines in $\Gamma_5(p_0)$ are in $\mathcal{L}_6^1$, $\beta$ are in $\mathcal{L}_6^3$, and $\gamma$ are in $\mathcal{L}_6^4$, where $\alpha+\beta+\gamma=(t+1)s^2t^2$. Thus in $\Gamma_6(p_0)$ we have $\alpha$ points in $\mathcal{P}_4^1$, $\beta$ points in $\mathcal{P}_4^3$, $2\gamma$ points in $\mathcal{P}_6^4$, and $(s-1)\alpha+(s-1)\beta+(s-2)\gamma$ points in $\mathcal{P}_8$. Thus, finally, in $\Gamma_7(p_0)$ we have $\alpha$ lines in $\mathcal{L}_2^1$, $2\beta$ lines in $\mathcal{L}_4^3$, $2\gamma$ lines in $\mathcal{L}_6^4$, $(t-1)\alpha$ lines in $\mathcal{L}_6^1$, $(t-2)\beta$ lines in $\mathcal{L}_6^3$, $2(t-1)\gamma$ lines in $\mathcal{L}_8$, and $\alpha'$, $\beta'$, $\gamma'$ lines in $\mathcal{L}_6^1$, $\mathcal{L}_6^3$, $\mathcal{L}_6^4$ respectively, where $\alpha'+\beta'+\gamma'=(s-1)t\alpha+(s-1)t\beta+(s-2)t\gamma$. 

It follows that
\begin{align*}
|\mathcal{L}_2^1|&=t+1+\alpha&|\mathcal{L}_4^3|&=2\beta&|\mathcal{L}_6^1|&=(t+1)st+t\alpha+\alpha'\\
|\mathcal{L}_6^3|&=(t-1)\beta+\beta'&|\mathcal{L}_6^4|&=3\gamma+\gamma'&|\mathcal{L}_8|&=2(t-1)\gamma
\end{align*}
where $\alpha+\beta+\gamma=(t+1)s^2t^2$ and $\alpha'+\beta'+\gamma'=(s-1)t\alpha+(s-1)t\beta+(s-2)t\gamma$. Since there are no fixed lines we have $|\mathcal{L}_2^1|=(t+1)|\mathcal{P}_0|$, and thus $\alpha=(t+1)(|\mathcal{P}_0|-1)$. By (\ref{eq:octformulae}) we have $|\mathcal{L}_6^1|=st|\mathcal{L}_2^1|$, and hence $\alpha'=(s-1)(t+1)t(|\mathcal{P}_0|-1)$. Then $|\mathcal{L}_6^3|=(t-1)|\mathcal{L}_4^3|$, and so $\beta=\beta'/(t-1)$. After eliminating $\gamma$ and $\gamma'$ from the formulae, the equation $(s+1)|\mathcal{L}_8|=st|\mathcal{L}_4^1|+st|\mathcal{L}_4^2|+(s-1)t|\mathcal{L}_4^3|+(t-1)|\mathcal{L}_6^4|$ can be solved for $\beta'$. Thus we have shown that $|\mathcal{L}_2^2|=|\mathcal{L}_4^1|=|\mathcal{L}_4^2|=|\mathcal{L}_6^2|=0$, and $|\mathcal{L}_2^1|=(t+1)|\mathcal{P}_0|$, $|\mathcal{L}_6^1|=(t+1)st|\mathcal{P}_0|$,  and
\begin{align*}
|\mathcal{L}_4^3|&=\frac{(t^2-1)(2s+2t-st-1)(s^2t^2+1-|\mathcal{P}_0|)}{2st-s-t}\\
|\mathcal{L}_6^3|&=\frac{(2s+2t-st-1)(t-1)^2(t+1)(s^2t^2+1-|\mathcal{P}_0|)}{2st-s-t}\\
|\mathcal{L}_6^4|&=\frac{(t+1)(2s^2t^2-s^2t-4st^2+5st-s-1)(s^2t^2+1-|\mathcal{P}_0|)}{2st-s-t}\\
|\mathcal{L}_8|&=\frac{(t^2-1)(st^2+st-2t^2+t-1)(s^2t^2+1-|\mathcal{P}_0|)}{2st-s-t}.
\end{align*}

Observe that if $s,t\geq 2$ then $st^2+st-2t^2+t-1\geq 0$, and therefore from the above formula for $|\mathcal{L}_8|$ we have $|\mathcal{P}_0|<s^2t^2+1$ (strict inequality since we assume $|\mathcal{L}_8|>0$). Then the formula for $|\mathcal{L}_4^3|$ implies that $2s+2t-st-1\geq 0$, and thus $(s,t)=(2,4)$ or $(s,t)=(4,2)$. To exclude these final cases we apply the eigenvalue techniques from Appendix~\ref{app:A}. Let $d_k$ be the number of lines mapped to distance $2k$ in the incidence graph (hence distance $k$ in the line graph). Consider the case $(s,t)=(2,4)$. The above formulae give $d_0=0$, $d_1=|\mathcal{L}_2^1|=5|\mathcal{P}_0|$, $d_2=|\mathcal{L}_4^3|=9(65-|\mathcal{P}_0|)/2$, $d_3=|\mathcal{L}_6^1|+|\mathcal{L}_6^3|+|\mathcal{L}_6^4|=8(195+2|\mathcal{P}_0|)$, and $d_4=|\mathcal{L}_8|=33(65-|\mathcal{P}_0|)/2$. Using these values in~(\ref{eq:noct1}) gives $n_1=-(65+63|\mathcal{P}_0|)/80$ and $n_2=15(|\mathcal{P}_0|-65)/64$. Since $|\mathcal{P}_0|<s^2t^2+1=65$ the formula for $n_2$ gives $|\mathcal{P}_0|=1$. But then $n_1=-8/5$.
A similar argument applies for the $(s,t)=(4,2)$ case. 
\end{proof}

\begin{lemma}\label{lem:Goct} 
Let $\Gamma$ be a finite thick generalised octagon, and suppose that $\theta$ is an exceptional domestic collineation of~$\Gamma$. The fixed element structure of $\theta$ is a tree of diameter at least $1$ and at most~$8$ in the incidence graph.
\end{lemma}

\begin{proof} By Lemma~\ref{lem:Foct} there are both fixed points and fixed lines. Suppose that $\Gamma_{\theta}$ is a suboctagon $\Gamma'$. Then either $\Gamma'$ has parameters $(s',t')$, or $\Gamma'$ does not have parameters (in which case $\Gamma'$ is a \textit{weak} generalised octagon). In the latter case \cite[Theorem~1.6.2]{mal:98} implies that $\Gamma'$ is either the double of a generalised quadrangle, or the quadruple of a digon. In either case there are elements $p, L$ in the fixed element structure with $d(p, L)=7$ such that neither all lines through $p$ are fixed, nor all points on $ L$ are fixed, contradicting Lemma~\ref{lem:Xgen}. 

Thus $\Gamma'$ has parameters $(s',t')$. Then by Lemma~\ref{lem:Xgen} $\Gamma'$ is either full or ideal, and by \cite[Theorem~1.8.8]{mal:98} we have either $s'=1$ (if $\Gamma'$ is full) or $t'=1$ (if $\Gamma'$ is ideal). Consider the $s'=1$ case (the $t'=1$ case is analogous). Let $p_0$ be a fixed point. All $t+1$ lines in $\Gamma_1(p_0)$ are in $\mathcal{L}_0$. In $\Gamma_2(p_0)$ there are $t+1$ points in $\mathcal{P}_0$, and $(t+1)(s-1)$ points in $\mathcal{P}_2^1$. Then in $\Gamma_3(p_0)$ there are $(t+1)t$ lines in $\mathcal{L}_0$, and $(t+1)(s-1)t$ lines in $\mathcal{L}_4^1$. Thus in $\Gamma_4(p_0)$ we have $(t+1)t$ points in $\mathcal{P}_0$, $(t+1)(s-1)t$ points in $\mathcal{P}_2^1$, and $(t+1)(s-1)st$ points in $\mathcal{P}_6^1$. Then in $\Gamma_5(p_0)$ there are $(t+1)t^2$ lines in $\mathcal{L}_0$, $(t+1)(s-1)t^2$ lines in $\mathcal{L}_4^1$, and $(t+1)(s-1)st^2$ lines in $\mathcal{L}_8$. Thus in $\Gamma_6(p_0)$ there are $(t+1)t^2$ points in $\mathcal{P}_0$, $(t+1)(s-1)t^2$ points in $\mathcal{P}_2^1$, $(t+1)(s-1)st^2$ points in $\mathcal{P}_6^1$, and $(t+1)(s-1)s^2t^2$ points divided amongst $\mathcal{P}_6^1$, $\mathcal{P}_6^2$, $\mathcal{P}_6^3$, and $\mathcal{P}_6^4$. Since $|\mathcal{P}_2^2|=0$ (as $\Gamma'$ is ideal), we have $|\mathcal{P}_6^2|=0$. Therefore, suppose that there are $\alpha_1$ points in $\mathcal{P}_6^1$, $\alpha_3$ points in $\mathcal{P}_6^3$, and $\alpha_4$ points in $\mathcal{P}_6^4$, where $\alpha_1+\alpha_3+\alpha_4=(t+1)(s-1)s^2t^2$. Then, finally, in $\Gamma_7(p_0)$ there are $(t+1)t^3$ lines in $\mathcal{L}_0$, $(t+1)(s-1)t^3$ lines in $\mathcal{L}_4^1$, $\alpha_1$ lines in $\mathcal{L}_4^1$, $\alpha_3$ lines in $\mathcal{L}_4^3$, $2\alpha_4$ lines in $\mathcal{L}_6^4$, and 
$
(t+1)(s-1)st^3+(t-1)(\alpha_1+\alpha_3)+(t-2)\alpha_4
$
lines in $\mathcal{L}_8$. Since $\mathcal{L}=\Gamma_1(p_0)\cup\Gamma_3(p_0)\cup\Gamma_5(p_0)\cup\Gamma_7(p_0)$ it follows that
$|\mathcal{L}_0|=(t+1)(1+t+t^2+t^3)$, $|\mathcal{L}_4^1|=(t+1)(s-1)t(1+t+t^2)$, 
$|\mathcal{L}_4^3|=\alpha_3$, 
$|\mathcal{L}_6^4|=2\alpha_4$, 
$|\mathcal{L}_8|=(t+1)^2(s-1)st^2+(t-1)(\alpha_1+\alpha_3)+(t-2)\alpha_4$, 
and $|\mathcal{L}_2^1|=|\mathcal{L}_2^2|=|\mathcal{L}_4^2|=|\mathcal{L}_6^1|=|\mathcal{L}_6^2|=|\mathcal{L}_6^3|=0$. Since $|\mathcal{L}_6^3|=(t-1)|\mathcal{L}_4^3|$ we have $\alpha_3=0$, and we can eliminate $\alpha_4$ from the formulae using $\alpha_4=(t+1)(s-1)s^2t^2-\alpha_1$. Finally we compute $\alpha_1$ using the formula $(s+1)|\mathcal{L}_8|=st|\mathcal{L}_4^1|+st|\mathcal{L}_4^2|+(s-1)t|\mathcal{L}_4^3|+(t-1)|\mathcal{L}_6^4|$, giving
$$
\alpha_1=\frac{st^2(s-1)(t+1)(2s^2+t^2-s^2t-s)}{s+2t-1}.
$$

The above formula for $\alpha_1$ fails to give a nonnegative integer for all parameter values in Lemma~\ref{lem:Cgen} except for the $(s,t)=(4,2)$ case, where $\alpha_1=0$. In this case we have $|\cL_0|=45$, $|\cL_4^1|=126$, $|\cL_6^4|=1152$, $|\cL_8|=432$, and $|\cL_2^1|=|\cL_2^2|=|\cL_4^2|=|\cL_4^3|=|\cL_6^1|=|\cL_6^2|=|\cL_6^3|=0$. Using these values in (\ref{eq:noct1}) gives $n_1=-9/5$, a contradiction. 

Thus the fixed element structure is a tree of diameter at most~$8$ in the incidence graph, and by Lemma~\ref{lem:Foct} this tree has diameter at least~$1$. 
\end{proof}

\begin{lemma}\label{lem:l21}
Let $\Gamma$ be a finite thick generalised octagon, and suppose that $\theta$ is an exceptional domestic collineation of~$\Gamma$. Then $|\mathcal{L}_2^1|=1+|\mathcal{P}_0|t-|\mathcal{L}_0|$ and $|\mathcal{L}_4^1|=|\mathcal{L}_0|st-(|\mathcal{P}_0|-1)t$.
\end{lemma}

\begin{proof} Let $p_0$ be a fixed point such that $d(p_0,x)\leq 5$ for all $x$ in the fixed element tree~$T$. Let $a=|\Gamma_1(p_0)\cap \mathcal{L}_0|$, $b=|\Gamma_2(p_0)\cap \mathcal{P}_0|=b$, $c=|\Gamma_3(p_0)\cap \mathcal{L}_0|$, $d=|\Gamma_4(p_0)\cap \mathcal{P}_0|$ and $e=|\Gamma_5(p_0)\cap \mathcal{L}_0|$. Then $|\mathcal{L}_2^1|=(t+1-a)+(tb-c)+(td-e)=1+t(1+b+d)-(a+c+e)=1+t|\mathcal{P}_0|-|\mathcal{L}_0|$. Similarly $|\mathcal{P}_2^1|=1+s|\mathcal{L}_0|-|\mathcal{P}_0|$, and hence $|\mathcal{L}_4^1|=t|\mathcal{P}_2^1|=t+st|\mathcal{L}_0|-t|\mathcal{P}_0|$. 
\end{proof}

The following count is our main tool in proving Theorem~\ref{thm:3}.

\begin{prop}\label{prop:maincountoct}
Let $\theta$ be an exceptional domestic collineation of a finite thick generalised octagon $\Gamma$. Then for some integers $x,y\geq 0$ we have
\begin{gather*}
\begin{aligned}
&\begin{aligned}
|\mathcal{L}_2^1|&=1+|\mathcal{P}_0|t-|\mathcal{L}_0|&|\mathcal{L}_2^2|&=x&|\mathcal{L}_4^1|&=|\mathcal{L}_0|st-(|\mathcal{P}_0|-1)t\\
|\mathcal{L}_4^2|&=(t-1)x&|\mathcal{L}_4^3|&=y&|\mathcal{L}_6^1|&=st+|\mathcal{P}_0|st^2-|\mathcal{L}_0|st\\
|\mathcal{L}_6^2|&=(s-1)tx&|\mathcal{L}_6^3|&=(t-1)y
\end{aligned}\\
&\begin{aligned}
|\mathcal{L}_6^4|&=\frac{s^2t^2(st+1)(st+s+t+1)-(|\mathcal{P}_0|+|\mathcal{L}_0|-1)s^2t^2-st(s+t)x-2sty}{s+t}\\
|\mathcal{L}_8|&=\frac{(st+1)(t^2-1)s^2t^2+|\mathcal{L}_0|s^2t^2-(|\mathcal{P}_0|-1)st^3+(s-t)ty}{s+t}.
\end{aligned}
\end{aligned}
\end{gather*}
Moreover, if $p_0\in \mathcal{P}_0$, and $|\Gamma_k(p_0)\cap(\mathcal{P}_0\cup \mathcal{L}_0)|=a_k$ for $0\leq k\leq 8$ then the following are nonnegative integers:
\begin{align*}
X_1(p_0)&=\frac{1}{2t(s+t)}\left(\sum_{k=0}^8p_k(s,t)a_k-(t-1)(s+t)x-(st-s-t-t^2)y\right)\\
X_2(p_0)&=(t+1-a_1)s^2t^2-a_8t-x-X_1(p_0)\\
X_3(p_0)&=y-2X_1(p_0)\\
X_4(p_0)&=(sa_1-a_2)s^2t^2-(sa_7-a_8)t-(t-1)x-y+2X_1(p_0)\\
X_5(p_0)&=(t-1)(y-X_1(p_0))\\
X_6(p_0)&=(s-2)(t+1)s^2t^3-(s-2)s^2t^3a_1+s^2t^3a_2-s^2t^2a_3\\
&\quad-st^2a_6+sta_7-(s-2)t^2a_8-(s-3)tx-(t-1)y+(2t-1)X_1(p_0),
\end{align*}
where $p_0(s,t)=-s^2t^2(t^2-1)(st-2s-2t+1)$, $p_3(s,t)=s^2t^2(s+t-1)$, $p_4(s,t)=p_5(s,t)=-s^2t^2$, $p_6(s,t)=st^3$, $p_7(s,t)=-st(st-s-t)$, $p_8(s,t)=t(st^2-2st+s+t-2t^2)$, and 
\begin{align*}
p_1(s,t)&=s^2t^2(s^2t+st^2-2s^2-2t^2+3s+3t-4st-1).
\end{align*}
\end{prop}

\begin{proof}
Let $p_0$ be a fixed point and let $|\Gamma_k(p_0)\cap (\mathcal{P}_0\cup \mathcal{L}_0)|=a_k$ for $0\leq k\leq 8$.  In $\Gamma_1(p_0)$ we have $a_1$ lines in $\mathcal{L}_0$ and $(t+1-a_1)$ lines in $\mathcal{L}_2^1$. In $\Gamma_2(p_0)$ we have $a_2$ points in $\mathcal{P}_0$, $(sa_1-a_2)$ points in $\mathcal{P}_2^1$, and $(t+1-a_1)$ points in $\mathcal{P}_4^1$. In $\Gamma_3(p_0)$ there are $a_3$ lines in $\mathcal{L}_0$, $(ta_2-a_3)$ lines in $\mathcal{L}_2^1$, $(sa_1-a_2)t$ lines in $\mathcal{L}_4^1$, and $(t+1-a_1)st$ lines in $\mathcal{L}_6^1$. Then in $\Gamma_4(p_0)$ we have $a_4$ points in $\mathcal{P}_0$, $(sa_3-a_4)$ points in $\mathcal{P}_2^1$, $(ta_2-a_3)s$ points in $\mathcal{P}_4^1$, $(sa_1-a_2)st$ points in $\mathcal{P}_6^1$, and $(t+1-a_1)s^2t$ points in $\mathcal{P}_8$.

Then, using domesticity, in $\Gamma_5(p_0)$ we have $a_5$ lines in $\mathcal{L}_0$, $(ta_4-a_5)$ lines in $\mathcal{L}_2^1$, $(sa_3-a_4)t$ lines in $\mathcal{L}_4^1$, $(ta_2-a_3)st$ lines in $\mathcal{L}_6^1$, $(sa_1-a_2)st^2$ lines in $\mathcal{L}_8$, plus $\alpha_1$, $\alpha_2$, $\alpha_3$, and $\alpha_4$ lines in $\mathcal{L}_6^1$, $\mathcal{L}_6^2$, $\mathcal{L}_6^3$, and $\mathcal{L}_6^4$ respectively, where $\alpha_1+\alpha_2+\alpha_3+\alpha_4=(t+1-a_1)s^2t^2$. In $\Gamma_6(p_0)$ we have $a_6$, $(sa_5-a_6)$, $(ta_4-a_5)s$, $(sa_3-a_4)st$, $(ta_4-a_3)s^2t$, $\alpha_1$, $(s-1)\alpha_1$, $\alpha_2$, $(s-1)\alpha_2$, $\alpha_3$, $(s-1)\alpha_3$, $2\alpha_4$, $(s-2)\alpha_4$, $\beta_1$, $\beta_2$, $\beta_3$, $\beta_4$ points in $\mathcal{P}_0$, $\mathcal{P}_2^1$, $\mathcal{P}_4^1$, $\mathcal{P}_6^1$, $\mathcal{P}_8$, $\mathcal{P}_4^1$, $\mathcal{P}_8$, $\mathcal{P}_4^2$, $\mathcal{P}_8$, $\mathcal{P}_4^3$, $\mathcal{P}_8$, $\mathcal{P}_6^4$, $\mathcal{P}_8$, $\mathcal{P}_6^1$, $\mathcal{P}_6^2$, $\mathcal{P}_6^3$, $\mathcal{P}_6^4$, respectively, where $\beta_1+\beta_2+\beta_3+\beta_4=(sa_1-a_2)s^2t^2$.

Then, in $\Gamma_7(p_0)$ there are $a_7$, $(ta_6-a_7)$, $(sa_5-a_6)t$, $(ta_4-a_5)st$, $(sa_3-a_4)st^2$, $\alpha_1$, $(t-1)\alpha_1$, $\beta_1$, $(t-1)\beta_1$, $\alpha_2$, $(t-1)\alpha_2$, $\beta_2$, $(t-1)\beta_2$, $2\alpha_3$, $(t-2)\alpha_3$, $\beta_3$, $(t-1)\beta_3$, $2\alpha_4$, $2(t-1)\alpha_4$, $2\beta_4$, $(t-2)\beta_4$, $\gamma_1$, $\gamma_2$, $\gamma_3$, $\gamma_4$ lines in $\mathcal{L}_0$, $\mathcal{L}_2^1$, $\mathcal{L}_4^1$, $\mathcal{L}_6^1$, $\mathcal{L}_8$, $\mathcal{L}_2^1$, $\mathcal{L}_6^1$, $\mathcal{L}_4^1$, $\mathcal{L}_8$, $\mathcal{L}_2^2$, $\mathcal{L}_6^2$, $\mathcal{L}_4^2$, $\mathcal{L}_8$, $\mathcal{L}_4^3$, $\mathcal{L}_6^3$, $\mathcal{L}_4^3$, $\mathcal{L}_8$, $\mathcal{L}_6^4$, $\mathcal{L}_8$, $\mathcal{L}_6^4$, $\mathcal{L}_8$, $\mathcal{L}_6^1$, $\mathcal{L}_6^2$, $\mathcal{L}_6^3$, $\mathcal{L}_6^4$, respectively, where 
\begin{align}\label{eq:gamma4}
\gamma_1+\gamma_2+\gamma_3+\gamma_4=(ta_4-a_3)s^2t^2+(s-1)t\alpha_1+(s-1)t\alpha_2+(s-1)t\alpha_3+(s-2)t\alpha_4.
\end{align}

Collecting together the lines in $\Gamma_1(p_0),\Gamma_3(p_0),\Gamma_5(p_0)$ and $\Gamma_7(p_0)$ we get
\begin{align*}
|\mathcal{L}_0|&=a_1+a_3+a_5+a_7\\
|\mathcal{L}_2^1|&=1+t(1+a_2+a_4+a_6)-(a_1+a_3+a_5+a_7)+\alpha_1\\
|\mathcal{L}_2^2|&=\alpha_2\\
|\mathcal{L}_4^1|&=st(a_1+a_3+a_5)-t(a_2+a_4+a_6)+\beta_1\\
|\mathcal{L}_4^2|&=\beta_2\\
|\mathcal{L}_4^3|&=2\alpha_3+\beta_3\\
|\mathcal{L}_6^1|&=st^2(1+a_2+a_4)-st(a_1+a_3+a_5)+st+t\alpha_1+\gamma_1\\
|\mathcal{L}_6^2|&=t\alpha_2+\gamma_2\\
|\mathcal{L}_6^3|&=(t-1)\alpha_3+\gamma_3\\
|\mathcal{L}_6^4|&=3\alpha_4+2\beta_4+\gamma_4\\
|\mathcal{L}_8|&=s^2t^2(a_1+a_3)-st^2(a_2+a_4)+(t-2)\beta_4+(t-1)(2\alpha_4+\beta_1+\beta_2+\beta_3).
\end{align*}
Use the formulae $\alpha_4=(t+1-a_1)s^2t^2-\alpha_1-\alpha_2-\alpha_3$, $\beta_4=(sa_1-a_2)s^2t^2-\beta_1-\beta_2-\beta_3$ and (\ref{eq:gamma4}) to eliminate $\alpha_4,\beta_4$ and $\gamma_4$ from these equations. 

Note that $|\mathcal{P}_0|=1+a_2+a_4+a_6+a_8$ and $|\mathcal{L}_0|=a_1+a_3+a_5+a_7$. We compute $\alpha_1$ and $\beta_1$ from the formulae in Lemma~\ref{lem:l21}. Then $\gamma_1$ is determined using the formula $|\mathcal{L}_6^1|=st|\mathcal{L}_2^1|$ (see (\ref{eq:octformulae})). Next, let $\alpha_2=x$ and $\alpha_3=(y-\beta_3)/2$. Then $\beta_2$ is determined from the equation $|\mathcal{L}_4^2|=(t-1)|\mathcal{L}_2^2|$, and $\gamma_2$ is determined from $|\mathcal{L}_6^2|=(s-1)t|\mathcal{L}_2^2|$. Then $\gamma_3$ is determined using $|\mathcal{L}_6^3|=(t-1)|\mathcal{L}_4^3|$, and finally $\beta_3$ is determined using the formula $(s+1)|\mathcal{L}_8|=st|\mathcal{L}_4^1|+st|\mathcal{L}_6^4|+(s-1)t|\mathcal{L}_4^3|+(t-1)|\mathcal{L}_6^4|$. The stated formulae for the cardinalities $|\mathcal{L}_i^j|$ follow after some algebra (which is best done on a computer). The formulae $X_k(p_0)$ with $1\leq k\leq 6$ in the statement of the proposition are the expressions for $\alpha_3,\alpha_4,\beta_3,\beta_4,\gamma_3,\gamma_4$ respectively, and so these must be nonnegative integers.
\end{proof}

\begin{remark}
Note that if one counts from a different root of the fixed element structure in Proposition~\ref{prop:maincountoct} then the values of $x$ and $y$ remain the same since $x=|\mathcal{L}_2^2|=|\mathcal{P}_2^2|$ and $y=|\mathcal{L}_4^3|=|\mathcal{P}_4^3|$ are determined by the collineation. However the formulae for $X_k(\cdot)$ change, since they depend on the numbers $a_1,\ldots,a_8$ (and these numbers depend on the choice of the root). This observation is crucial for the proof of Theorem~\ref{thm:3}: We ``change the root'' of the fixed element structure to obtain suitable expressions $X_k(\cdot)$ which all need to be nonnegative, thus putting severe restrictions on the fixed element structure.
\end{remark}

\begin{cor}\label{cor:y} Let $\theta$ be an exceptional domestic collineation of a finite thick generalised octagon $\Gamma$. In the notation of Proposition~\ref{prop:maincountoct}:
\begin{enumerate}
\item[\emph{1.}] If $a_1=t+1$ then $x=0$, $a_8=0$, and 
$$
(st-s-t-t^2)y=\sum_{k=0}^{8}p_k(s,t)a_k,
$$
Thus, except in the $(s,t)=(6,3)$ case, the value of $y$ is determined (and hence all cardinalities $|\mathcal{L}_i^j|$ are determined), while in the $(s,t)=(6,3)$ case we have $
\sum_{k=0}^8p_k(s,t)a_k=0$.
\item[\emph{2.}] If $a_1=1$, $a_2=s$ and $a_8=sa_7$ then $x=0$ and
$$
y=\frac{1}{2st-s-t}\sum_{k=0}^8p_k(s,t)a_k.
$$
\end{enumerate}
\end{cor}

\begin{proof}
If $a_1=t+1$ then by Lemma~\ref{lem:octP_2} $|\mathcal{L}_2^2|=0$, and so $x=0$. Also, since $a_1>1$ we have $a_8\neq 0$ (for otherwise there are elements at distance $9$ in the fixed element tree). Then $X_2(p_0)=-X_1(p_0)$, and so $X_1(p_0)=X_2(p_0)=0$, and the first claim follows. The second claim is similar, using instead the formulae for $X_3(p_0)$ and $X_4(p_0)$. 
\end{proof}

\subsection*{Proof of Theorem~\ref{thm:3}}

Theorem~\ref{thm:3} is established from Propositions~\ref{prop:octproof1}, \ref{prop:octproof2}, \ref{prop:octproof3}, \ref{prop:octproof4} and~\ref{prop:octproof5} below. Suppose that $\Gamma$ is a thick finite generalised octagon admitting an exceptional domestic collineation $\theta:\Gamma\to\Gamma$. By Lemma~\ref{lem:Cgen} the parameters of $\Gamma$ are of the form $(s,t)$ with $s=2t$ and $2\leq t\leq 7$, or $t=2s$ with $2\leq s\leq 7$. By Lemma~\ref{lem:Foct} the collineation $\theta$ has both fixed points and fixed lines. By Lemma~\ref{lem:Goct} the fixed element structure is a tree $T$ of diameter at least $1$ and at most~$8$ in the incidence graph.

\begin{prop}\label{prop:octproof1} If $s,t\geq 3$ then the diameter of~$T$ is at most~$6$.
\end{prop}

\begin{proof} 
Suppose that $T$ has diameter $7$ or $8$. By Lemma~\ref{lem:Xgen} diameter~$7$ is impossible, for if $p, L\in T$ with $d(p, L)=7$ then either all lines through $p$ are fixed, or all points on $ L$ are fixed, and in either case this produces elements in $T$ at distance~$8$. Thus $T$ has diameter~$8$, and so $T$ has a unique centre, which is either a point or a line.

Suppose, up to duality, that the centre $p_0$ of the tree is a point. Let $f_k=|\Gamma_k(p_0)\cap T|$. Then $f_0=1$, and by assumption $f_k=0$ for $k\geq 5$. Furthermore, by Lemma~\ref{lem:Xgen} we have $f_4=sf_3$, and $x=0$ by Lemma~\ref{lem:octP_2}. Let $p_0'\in \mathcal{P}_0$ be a point of the fixed element structure $T$ at distance~$4$ from~$p_0$ (and so $p_0'$ is an end of the tree~$T$). Let $p_0\bI  L_1\bI p_2\bI L_3\bI p_0'$ be the geodesic from $p_0$ to $p_0'$, and let 
\begin{align*}
\mathcal{U}&=\{ L\in T\mid  L\bI p_2\textrm{ and } L\neq L_1, L_3\}\\
\mathcal{V}&=\{p\in T\mid p\bI L_1\textrm{ and }p\neq p_2,p_0\}\\
\mathcal{W}&=\{ L\in T\mid  L\bI p\textrm{ for some } p\in \mathcal{V},\textrm{ and } L\neq L_1\}.
\end{align*} 
Let $u=|\mathcal{U}|$, $v=|\mathcal{V}|$, and $w=|\mathcal{W}|$. Note that $u\leq t-1$, $v\leq s-1$, and $w\leq vt$. 

Draw the fixed element tree $T$ relative to the new root $p_0'$. Then, in the notation of Proposition~\ref{prop:maincountoct}, we have $a_1=1$, $a_2=s$, $a_3=u+1$, $a_4=su+v+1$, $a_5=w+f_1-1$, $a_6=sw+f_2-v-1$, $a_7=f_3-u-w-1$, and $a_8=s(f_3-u-w-1)$. Thus, by the second part of Corollary~\ref{cor:y} we can compute $y$, and we obtain
\begin{align*}
y&=\frac{st(-p(s,t)-stf_1+t^2f_2+(t-1)(st-2s-2t)f_3+(s+t)(2(t-1)u-tv+2(t-1)w))}{2st-s-t},
\end{align*}
where $p(s,t)=2s+2t-2st-s^2t-t^2-st^2+s^2t^2+3st^3-2s^2t^3-2st^4+s^2t^4$. Write $f_3=tf_2-f_3'$, $f_2=sf_1-f_2'$, $f_1=t+1-f_1'$, $u=t-1-u'$, $w=tv-w'$, and $v=s-1-v'$, and so $f_1',f_2',f_3',u',v',w'\geq 0$. Then
$$
y=-\frac{q_1(s,t)f_1'+q_2(s,t)f_2'+q_3(s,t)f_3'+q_4(s,t)u'+q_5(s,t)v'+q_6(s,t)w'}{2st-s-t},
$$
where the polynomials $q_k(s,t)$ are given by $q_1(s,t)=s^2t^2(t-1)(st-2s-2s+1)$, $q_2(s,t)=st^2(st^2-2t^2-3st+3t+2s)$, $q_3(s,t)=st(t-1)(st-2s-2s)$, $q_4(s,t)=2st(t-1)(s+t)$, $q_5(s,t)=st^2(s+1)(2t-3)$, and $q_6(s,t)=2st(t-1)(s+t)$. These polynomials are all strictly positive for parameter values $(2t,t),(s,2s)$ with $s,t\geq 3$. Hence $y=0$,  and $f_1'=f_2'=f_3'=u'=v'=w'=0$. 

Thus $f_1=t+1$, $f_2=(t+1)s$ and $f_3=(t+1)st$, and so the fixed element tree $T$ is the ball of radius~$4$ centred at~$p_0$. But then $\mathcal{L}_8=\emptyset$. Thus $\theta$ is not exceptional domestic.
\end{proof}

\begin{prop}\label{prop:octproof2} If $s,t\geq 4$ then the diameter of $T$ is at most~$2$.
\end{prop}

\begin{proof}
Suppose that $T$ has diameter between $3$ and $6$. Choose a root $x_0$ of the fixed element tree $T$ so that $f_1,f_2,f_3\neq 0$ and $f_k=0$ for $k>3$, where $f_k=|\Gamma_k(x_0)\cap T|$. We will consider the possibilities of $x_0$ being a point or a line, and thus we may assume that $(s,t)=(2t,t)$.

Suppose first that $x_0=p_0$ is a point. Let $ L_0'\in T$ be a line at distance $3$ from $p_0$, and let $p_0\bI  L_1\bI p_2\bI L_0'$ be the geodesic from $p_0$ to~$ L_0'$. Let 
\begin{align}\label{eq:uvw}
\begin{aligned}
\mathcal{U}&=\{ L\in T\mid  L\bI p_2\textrm{ and } L\neq  L_0', L_1\}\\
\mathcal{V}&=\{p\in T\mid p\bI L_1\textrm{ and }p\neq p_0,p_2\}\\
\mathcal{W}&=\{ L\in T\mid  L\bI p\textrm{ for some $p\in \mathcal{V}$, and } L\neq  L_1\},
\end{aligned}
\end{align}
and let $u=|\mathcal{U}|$, $v=|\mathcal{V}|$, and $w=|\mathcal{W}|$. Furthermore, suppose that we have chosen $ L_0'\in\Gamma_3(p_0)\cap T$ such that $u$ is maximal. Therefore $f_3\leq (u+1)f_2$.

Now take $ L_0'$ to be a new root of the tree. In the notation of (the dual version of) Proposition~\ref{prop:maincountoct} we have $a_1=1$, $a_2=u+1$, $a_3=v+1$, $a_4=f_1-1+w$, $a_5=f_2-1-v$, $a_6=f_3-u-w-1$, and $a_7=a_8=0$, and we compute
\begin{align*}
X_4( L_0')=\frac{-p_1(t)-4t^3f_1-4t^3f_2+8t^3f_3-p_2(t)u+12t^4v-12t^3w-(12t^2-3)x-(4t-3)y}{6t}
\end{align*}
where $p_1(t)=8t^3(4t^4-18t^3+16t^2-6t+1)$ and $p_2(t)=12t^3(4t^2-2t+1)$. Using the inequalities $f_3\leq (u+1)f_2$ and $v\leq s-1$, followed by the inequalities $f_2\leq sf_1$, and then $f_1\leq t+1$, shows that
$
X_4( L_0')\leq -\frac{2}{3}t^2(2t-1)(4t^3-16t^2+4t-3+(4t-3)u),
$
and this is negative for $t\geq 4$, a contradiction.

Now suppose that $x_0= L_0$ is a line. Let $p_0'$ be a point at distance~$3$ from~$ L_0$. Applying the same argument as above (with $ L_0'$ replaced by $p_0'$, and with dual definitions of $u$, $v$ and $w$) we obtain
$$
X_1(p_0')=\frac{-q_1(t)-4t^3f_1-4t^3f_2+2t^3f_3-q_2(t)u+12t^4v-6t^3w-3(t-1)x-(t-3)y}{6t}
$$
where $q_1(t)=2t^3(4t^4-24t^4+40t^2-24t+1)$ and $q_2(t)=6t^3(2t^2-4t+1)$. Using $f_3\leq (u+1)f_2$ and $v\leq t-1$ gives
$$
X_1(p_0')\leq-t^2(1-18t+34t^2-24t^3+4t^4+2f_1-(u-1)f_2+(6t^2-12t+3)u)/3.
$$
If $u=0$ then $X_1(p_0')<0$ for all $t\geq 5$. The case $u=0$ and $t=4$ requires a separate argument (see below). If $u=1$ then $X_1(p_0')<0$ for all $t\geq 4$. So suppose that $u\geq 2$. Using $f_2\leq tf_1$ gives
$$
X_1(p_0')\leq-t^2(1-18t+34t^2-24t^3+4t^4-(tu-t-2)f_1+(6t^2-12t+3)u)/3.
$$
Since $u\geq 2$ the coefficient of $f_1$ is positive, and using $f_1\leq s+1$ gives
$$
X_1(p_0')\leq-t^2(-24t^3+4t^4+36t^2-13t+3+(4t^2-13t+3)u)/3<0,
$$
a contradiction.

It remains to consider the case $(s,t)=(8,4)$ with $u=0$. Since $f_3\leq f_2$ we have
\begin{align*}
X_1(p_0')&=(-4224-256 f_1-256 f_2+128f_3-6528 u+3072 v-384 w-9x-y)/24\\
&\leq -16(33+2f_1+2f_2-f_3-24v)/3\leq-16(33+2f_1+f_2-24v)/3.
\end{align*}
But $v\leq t-1=3$. If $v=0,1$ then $X_1(p_0')<0$, and so $v=3$ (since $v=2$ is impossible). By Lemma~\ref{lem:octP_2} we have $x=0$, because $v=3$ implies that all points on $p_1$ are fixed, 
where $ L_0\bI p_1\bI  L_2\bI p_0'$ is the geodesic from $ L_0$ to $p_0'$. Taking $p_1$ as a new root of the tree, and applying Corollary~\ref{cor:y}, we compute $y=256(477-25f_1+11f_2-f_3-24w)$. Using this value in the formula for $X_1(p_0')$ we obtain
$
X_1(p_0')=-16(305-16f_1+8f_2-f_3-15w).
$
Now, $f_3\leq f_2$ and $w\leq v=3$ (this follows from the fact that $u=0$). Also $f_1\leq s+1$, and thus $X_1(p_0')\leq -16(116+7f_2)$, a contradiction.
\end{proof}

\begin{prop}\label{prop:octproof3} If $s,t\geq 3$ then the diameter of $T$ does not equal~$1$ or~$2$. 
\end{prop}

\begin{proof}
Suppose that the diameter of $T$ equals~$1$ or~$2$, and that $s,t\geq 3$. Choose a root $x_0$ of the fixed element tree $T$ so that $f_1\neq 0$ and $f_k=0$ for $k>1$, where $f_k=|\Gamma_k(x_0)\cap T|$. By considering the the possibilities of $x_0$ being a point or a line, we may assume that $(s,t)=(2t,t)$.

Suppose first that $x_0=p_0$ is a point. Then
$$
X_4(p_0)=\frac{4t^3-12t^4-32t^5+144t^6-32t^7-4t^3(12t^2-6t+1)f_1-3(4t^2-1)x-(4t-3)y}{6t}.
$$
Using $f_1\leq t+1$ gives $X_4(p_0)\leq-\frac{4}{3}(4t^6-12t^5+7t^4-t^3)<0$, a contradiction.

Now suppose that $x_0= L_0$ is a line. Then
$$
X_1( L_0)=\frac{4t^3-12t^4-56t^5+48t^6-8t^7-4t^3(3t^2-6t+1)f_1-3(t-1)x-(t-3)y}{6t}.
$$
Using $f_1\leq s+1$ gives
$
X_1( L_0)\leq-\frac{2t^3}{3}(2t^3-6t^2+5t-1)<0,
$
a contradiction.
\end{proof}

So far we have shown that if $\Gamma$ admits an exceptional domestic collineation then $(s,t)\in\{(4,2),(2,4),(6,3),(3,6)\}$, and if $(s,t)=(3,6),(6,3)$ then the diameter of the fixed element tree is between~$3$ and~$6$. We now eliminate these more difficult cases.  We need two additional preliminary results.

\begin{lemma}\label{lem:div}
Let $\theta$ be an exceptional domestic collineation of a finite thick generalised octagon~$\Gamma$. As in Proposition~\ref{prop:maincountoct}, let $x=|\cL_2^2|$ and $y=|\cL_4^3|$. Then 
\begin{align*}
y&=8t^2|\mathcal{L}_0|+260t^2|\mathcal{P}_0|+184t^6+4t^4+180t^3+92t^2+90tx\mod 360t^2&&\textrm{if $s=2t$}\\
y&=8s^2|\mathcal{P}_0|+260s^2|\mathcal{L}_0|+184s^6+4s^4+180s^3+92s^2+90sx\mod 360s^2&&\textrm{if $t=2s$}.
\end{align*}
If $(s,t)=(2t,t)$ then $x$ is divisible by $2t$, and if $(s,t)=(s,2s)$ then $x$ is divisible by $2s$.
\end{lemma}

\begin{proof} 
Suppose that $s=2t$. Using the formulae for $|\mathcal{L}_i^j|$ from Proposition~\ref{prop:maincountoct} in~(\ref{eq:noct1}) gives
 \begin{align}
 \label{eq:n1}n_1&=\frac{13|\mathcal{L}_0|t^2-5|\mathcal{P}_0|t^2-16t^6-16t^4-8t^2+4y}{45t^2}\\
 \label{eq:n2}n_2&=\frac{4|\mathcal{L}_0|t^2+4|\mathcal{P}_0|t^2-16t^6+20t^4-8t^2-5y}{36t^2}\\
 \nonumber n_3&=\frac{32|\mathcal{L}_0|t^2+20|\mathcal{P}_0|t^2+16t^6-44t^4-60t^3-52t^2+30tx+11y}{120t^2}\\
 \nonumber n_4&=\frac{8|\mathcal{L}_0|t^2-4|\mathcal{P}_0|t^2+16t^6+4t^4+12t^3-4t^2-6tx-y}{24t^2}.
 \end{align}
The formulae for $n_1$ and $n_2$ imply that $y$ is divisible by $t^2$, and the formula for $n_4$ implies that $6x/t$ is an integer. Writing $y=t^2y'$ and $x=tx'/6$, it follows that
\begin{align*}
y'&\equiv8|\mathcal{L}_0|+35|\mathcal{P}_0|+4t^4+4t^2+2\mod 45\\
y'&\equiv8|\mathcal{L}_0|+8|\mathcal{P}_0|+4t^4+4t^2+20\mod 36\\
y'&\equiv8|\mathcal{L}_0|+20|\mathcal{P}_0|+64t^4+4t^2+60t+92+65x'\mod 120\\
y'&\equiv8|\mathcal{L}_0|+20|\mathcal{P}_0|+16t^4+4t^2+12t+20+23x'\mod 24.
\end{align*}
By the Chinese Remainder Theorem these congruences have a solution if and only if the right hand sides are congruent modulo the greatest common divisor of the respective moduli. Thus the second and fourth congruences imply that $x'\equiv 0\mod 12$. Thus $x$ is divisible by $2t$, and the result follows from the Chinese Remainder Theorem. The $t=2s$ case is similar.
\end{proof}

\begin{prop}\label{prop:maincountoct2}
Let $\theta$ be an exceptional domestic collineation of a finite thick generalised octagon~$\Gamma$. Let $p_0\in\cP_0$, and suppose that there is a line $L\in\cL_2^1$ through~$p_0$. Suppose that $|\Gamma_k(p_0)\cap(\mathcal{P}_0\cup \mathcal{L}_0)|=a_k$ for $0\leq k\leq 8$. If $(s,t)\neq (2,4)$ then the following are nonnegative integers:
\begin{align*}
X_1'( L)&=\frac{\sum_{k=0}^8 p_k'(s,t)a_k-(s-1)(t-1)(s+t)x+(st-s-t-s^2)y}{2(s+t)(st-2s-t)}\\
X_2'( L)&=s^2t^2-a_8-(x/2)-X_1'( L)\\
X_3'( L)&=y-3X_1'( L).
\end{align*}
where $p_0'(s,t)=-(s-1)(t-1)(st-2s-2t+1)s^2t^2$, $p_1'(s,t)=(st+s^2-2s-2t+1)s^2t^2$, $p_2'(s,t)=-(s+t-1)s^2t^2$, $p_3'(s,t)=p_4'(s,t)=s^2t^2$, $p_5'(s,t)=-s^3t$, $p_6'(s,t)=(st-s-t)st$, $p_7'(s,t)=-(s^2t-2s-2t)s$, $p_8'(s,t)=s^2t^2-3st^2+2t^2-3s^2t+6st-4t-4s+4s^2$

There is an obvious dual version to these statements.
\end{prop}

\begin{proof} This is analogous to the proof of Proposition~\ref{prop:maincountoct}, and we only sketch the proof. One begins the count at $ L\in \mathcal{L}_2^1$, and computes the point sets $|\mathcal{P}_i^j|$. By counting, and using domesticity, we have
$|\mathcal{P}_0|=1+a_2+a_4+a_6+\alpha_1'$, $|\mathcal{P}_2^1|=s(a_1+a_3+a_5)-(a_2+a_4+a_6)$, $|\mathcal{P}_2^2|=2\alpha_2'$, $|\mathcal{P}_4^1|=(t-a_1+ta_2-c_3+ta_4-a_5)s+s\alpha_1'+\beta_1'$, $|\mathcal{P}_4^2|=(s-1)\alpha_2'+\beta_2'$, $|\mathcal{P}_4^3|=3\alpha_3'+\beta_3'$, $|\mathcal{P}_6^1|=(sa_1+sa_3-a_2-a_4)st+\gamma_1'$, $|\mathcal{P}_6^2|=\gamma_2'$, $|\mathcal{P}_6^3|=2(s-1)\alpha_3'+\gamma_3'$, $|\mathcal{P}_6^4|=4\alpha_4'+2\beta_4'+\gamma_4'$, and 
\begin{align*}
|\mathcal{P}_8|&=s^2t(1+t-a_1+ta_2-a_3)+(s-1)(\alpha_1'+\alpha_2'+\alpha_3')+(s-2)\alpha_4'\\
&\quad+(t-1)s(\alpha_1'+\alpha_2')+(t-2)s\alpha_3'+2(s-1)\alpha_4'+(s-1)(\beta_1'+\beta_2'+\beta_3')+(s-2)\beta_4'
\end{align*}
for integers $\alpha_i',\beta_i',\gamma_i'\geq 0$ with $\sum\alpha_i'=s^2t^2$, $\sum\beta_i'=(t-a_1)s^2t^2+(s-1)t(\alpha_1'+\alpha_2'+\alpha_3')+(s-2)t\alpha_4'$, and $\sum\gamma_i=(sa_1-a_2)s^2t^2+2s(t-1)\alpha_4'$. One now computes the nonnegative integers $\alpha_i',\beta_i',\gamma_i'$ by using the formulae for $|\mathcal{L}_i^j|$ in Proposition~\ref{prop:maincountoct} and the equations~(\ref{eq:octformulae}). The formulae for $X_1'( L),X_2'( L)$ and $X_3'( L)$ are the formulae for $\alpha_3',\alpha_4'$ and $\beta_3'$ respectively. Note that the $(2,4)$ case needs to be excluded, for in this case the denominator of $X_1'( L)$ is zero.
\end{proof}

\begin{prop}\label{prop:octproof4} There are no exceptional domestic collineations for generalised octagons with parameters $(s,t)=(4,2)$ or $(s,t)=(2,4)$.
\end{prop}

\begin{proof} We will prove every statement and its dual, and so we may assume that $(s,t)=(4,2)$. If the fixed element tree has diameter at least~$4$ then Lemma~\ref{lem:octP_2} gives $x=0$ (because there is at least one fixed  point with $2$ of the $3$ lines through it fixed, and thus all lines through this point are fixed).

Suppose that the fixed element tree $T$ has diameter~$7$ or~$8$. As in Proposition~\ref{prop:octproof1}, diameter~$7$ is impossible and so $T$ has diameter~$8$. Let $x_0$ be the unique centre of~$T$, and let $f_k=|\Gamma_k(x_0)\cap T|$. Thus $f_1,f_2,f_3,f_4\geq 2$, and $f_k=0$ for $k>4$. Moreover, by Lemma~\ref{lem:Xgen} we have $f_4=sf_3$ in the case that $x_0$ is a point, and $f_4=tf_3$ in the case that $x_0$ is a line.

Suppose first that $x_0=p_0$ is a point. Then $f_1=t+1=3$ (because fixing $2$ of the $3$ lines through~$p_0$ is impossible), and from Corollary~\ref{cor:y} we compute $y=32(12+f_2-f_3)$. Then observe that $X_4(p_0)=32(12-3f_2+f_3)=-X_6(p_0)$, and so $f_3=3(f_2-4)$. Using these values in the formula (\ref{eq:n1}) gives $n_1=\frac{62}{5}-2f_2$, a contradiction. 

Now suppose that $x_0= L_0$ is a line. Let $ L$ be a fixed line at distance $4$ from $ L_0$. Let $ L_0\bI p_1\bI  L_2\bI p_3\bI  L$ be the geodesic from $ L_0$ to $ L$. Since $2$ of the $3$ lines through $p_3$ are fixed, necessarily all lines through $p_3$ are fixed. Similarly, all lines through $p_1$ are fixed. Let 
\begin{align*}
\mathcal{U}&=\{p\in T\mid p\bI L_2\textrm{ and }p\neq p_1,p_3\}\\
\mathcal{V}&=\{ L'\in T\mid  L'\bI p_1\textrm{ and } L'\neq  L_0\}\\
\mathcal{W}&=\{p\in T\mid p\bI L'\textrm{ for some $ L'\in\mathcal{V}$, and $p\neq p_1$}\}.
\end{align*}
Then $|\mathcal{V}|=1$. Let $u=|\mathcal{U}|$ and $w=|\mathcal{W}|$. Relative to the new root $ L$, in the notation of (the dual of) Proposition~\ref{prop:maincountoct} we have $a_1=1$, $a_2=t$, $a_3=u+1$, $a_4=t(u+1)$, $a_5=f_1+w-1$, $a_6=f_2-t-tw$, $a_7=f_3-u-w-1$, and $a_8=t(f_3-u-w-1)$. By Corollary~\ref{cor:y} we compute
$$
y=16(27-2f_1+4f_2-3f_3+9(u+w))/5.
$$
On the other hand, relative to the root $p_1$ we have $a_1=t+1$, $a_2=f_1+u+w$, $a_3=f_2+tu+tw$, $a_4=f_3-u-w-1$, $a_5=t(f_3-u-w-1)$, $a_6=a_7=a_8=0$, and by Corollary~\ref{cor:y} we compute
$$
y=32(21+f_1-5f_2+3f_3-12(u+w)).
$$
From these two expressions for $y$ we deduce that
$$
u+w=(21+4f_1-18f_2+11f_3)/43.
$$
Note that $f_2$ is even (for each point of $T$ in $\Gamma_1( L_0)$ has either $0$ or $2$ neighbours in $\Gamma_2( L_0)\cap T$) and that $f_1\neq 4$. We also have $2\leq f_1\leq 5$, $2\leq f_2\leq 2f_1$, and $2\leq f_3\leq 4f_2$. Then it is elementary that the only values of $f_1,f_2,f_3$ that make the right hand side of the formula for $u+w$ an integer are $(f_1,f_2,f_3)=(5,6,10), (5,8,25), (5,10,40)$. In these cases $f_1=5=s+1$, and so by Corollary~\ref{cor:y} we compute a third formula for $y$:
$$
y=32(13f_2-3f_3-10)/7.
$$
In the cases $(f_1,f_2,f_3)=(5,6,10),(5,8,25)$ this formula fails to give an integer, leaving only the case $(f_1,f_2,f_3)=(5,10,40)$. But then the fixed element tree equals the ball of radius $4$ around~$ L_0$, and so $\mathcal{P}_8=\emptyset$ and the automorphism is not exceptional domestic. 

Suppose that the fixed element tree has diameter between~$3$ and~$6$. Choose a root $x_0$ of the fixed element tree $T$ so that $f_1,f_2,f_3\neq 0$ and $f_k=0$ for $k>3$, where $f_k=|\Gamma_k(x_0)\cap T|$. Suppose first that $x_0=p_0$ is a point. Then $f_1=1$ or $f_1=3$. Suppose that $f_1=1$. Then $(f_2,f_3)=(1,2),(2,2),(2,4),(4,2),(4,4),(4,6)$. Let $ L\in T$ be a line at distance $3$ from the root~$p_0$. Let $p_0\bI L_1\bI p_2\bI L$ be a the geodesic from $p_0$ to $ L$. Then all lines through $p_2$ are fixed (since $ L_1$ and $ L$ are fixed), and so $x=0$. Relative to the new root $ L$ we have $a_1=1$, $a_2=2$, $a_3=f_2$, $a_4=f_3-2$, and $a_5=a_6=a_7=a_8=0$. From Corollary~\ref{cor:y} we compute $y=32(14+5f_2-f_3)/5$. Only the cases $(f_2,f_3)=(2,4),(4,4)$ make this an integer (with $y=4,6$ respectively), and in these last cases we compute $X_3(p_0)<0$, a contradiction.  

Now suppose that $f_1=3$. Then $x=0$ and from Corollary~\ref{cor:y} we have $y=32(12+f_2-5f_3)$. Then
$
n_1=(30+9f_2-49f_3)/15
$
(see (\ref{eq:n1})). Thus if $f_2=2,4,6,8,10,12$ then $f_3=12,9,6,3,0,12\mod 15$, respectively. But it is clear geometrically that $f_3=0\mod 2$, and so only the $f_2=2,6,10,12$ cases remain. These cases all give either $X_3(p_0)<0$ or $X_4(p_0)<0$, ruling out all cases.

Now suppose that $x_0= L_0$ is a line. We may assume that the diameter of~$T$ is exactly~$6$ (for otherwise we may shift the root $x_0$ and apply the above argument). Thus $f_1,f_2,f_3\geq 2$. Let $p$ be a fixed point at distance~$3$ from the centre~$ L_0$, and let $ L_0\bI p_1\bI L_2\bI p$ be the geodesic from $ L_0$ to $p$. Then all lines through $p_1$ are fixed (since $ L_0$ and $ L_2$ are fixed). Relative to the root $p_1$ we have $a_1=3$, $a_2=f_1-1+u$, $a_3=f_2-2$, $a_4=f_3-u$ and $a_5=a_6=a_7=a_8=0$, where $u=|\mathcal{U}|$, where $\mathcal{U}$ is the set of all $p'\in T$ such that the geodesic from $ L_0$ to $p'$ passes through $p_1$. Then by Corollary~\ref{cor:y} we have $y=32(21+f_1-5f_2+f_3)$ (independent of $u$), and the formula (\ref{eq:n1}) for $n_1$ gives $9f_1+11f_2+9f_3=1\mod 15$. Since $f_2$ is even, and since $f_1\neq 4$, the only values of $(f_1,f_2,f_3)$ which satisfy this congruence and give $y\geq 0$ are $(f_1,f_2,f_3)=(2,2,4), (3,2,3), (3,2,8),(5,2,6), (5,8,17), (5,8,22), (5,8,27), (5,8,32)$. The first  case gives $X_4( L_0)<0$, and the remaining cases give $X_2( L_0)<0$.

Suppose that $T$ has diameter~$2$. Let $x_0$ be the unique centre of~$T$, and let $f_1=|\Gamma_1(x_0)\cap T|$. If $x_0=p_0$ is a point then $f_1=3$ and so $x=0$. By Corollary~\ref{cor:y} we compute $y=384$. But then $X_6(p_0)=-384<0$. If $x_0= L_0$ is a line then $f_1=2,3,5$. If $f_1=2$ then by Lemma~\ref{lem:div} we have $y=1360\mod 1440$. But by the formula for $|\mathcal{L}_6^4|$ we have $y\leq 540$. A similar argument applies for $f_2=3$. If $f_2=5$ then $y=160\mod 1440$, and so $y=160$. But then $X_2( L_0)<0$.

Finally, suppose that $T$ has diameter~$1$, and so there is a unique fixed chamber $\{p_0, L_0\}$. Then by (\ref{eq:n1}) and (\ref{eq:n2}) we have $n_1=(y-320)/45$ and $n_1=-(5y+704)/144$, and so $y\equiv 320\mod 720$. Since $|\mathcal{L}_6^4|=8(536-3x-y)/3$ we have $y=320$. Then $X_4(p_0)=-3x/2$, and so $x=0$. Let $ L\in \mathcal{L}_2^1$ be a line through $p_0$. Then $X_2'( L)=-64<0$, a contradiction. 
\end{proof}

\begin{prop}\label{prop:octproof5} There are no exceptional domestic collineations for generalised octagons with parameters $(s,t)=(6,3)$ or $(s,t)=(3,6)$.
\end{prop}

\begin{proof}
By Propositions~\ref{prop:octproof1} and~\ref{prop:octproof3} we may assume that the diameter of $T$ is between $3$ and~$6$. Since the $(s,t)=(6,3),(3,6)$ cases are relatively uninteresting in the sense that it is widely believed that no octagon exists with these parameters, we shorten the exposition here by appealing to basic computer searches at a few points in the argument.

Choose a root $x_0$ of the fixed element tree $T$ so that $f_1,f_2,f_3\neq 0$ and $f_k=0$ for $k>3$, where $f_k=|\Gamma_k(x_0)\cap T|$. By considering the the possibilities of $x_0$ being a point or a line, we may assume that $(s,t)=(6,3)$. 

Suppose that $x_0=p_0$ is a point. Suppose first that $f_1\neq 4$ and so $1\leq f_1\leq 2$, and there is a line $ L\in \mathcal{L}_2^1$ through~$p_0$. Then
\begin{align*}
X_3(p_0)&=(288-936f_1+360f_2-288f_3+2x+3y)/3\\
X_1'( L)&=(-360+1332f_1-288f_2+36f_3-10x-3y)/6.
\end{align*}
The formula for $|\mathcal{L}_6^4|$ gives $x\leq 1064$ and $y\leq 4788$, and Lemma~\ref{lem:div} gives
\begin{align}\label{eq:yyt}
y\equiv 72(f_1+f_3)+2340(1+f_2)+270x+828\mod 3240,
\end{align}
which limits $y$ to 1 or 2 values for each given $f_1,f_2,f_3,x$. It is now routine to check by computer that for $x\leq 1064$ (with $x$ divisible by~$6$, see Lemma~\ref{lem:div}), $1\leq f_1\leq 2$, $1\leq f_2\leq sf_1$ and $1\leq f_3\leq tf_2$ that there are no solutions satisfying $X_3(p_0)\geq0$, $X_1'( L)\geq0$, and~(\ref{eq:yyt}).

Thus we have $f_1=4$, and so $x=0$. Then by Corollay~\ref{cor:y} we compute $f_3=(5f_2-48)/4$. Therefore $f_2$ is divisible by $4$, and $f_2\geq 10$. Thus $f_2=12,16,20,24$ are the only possibilities, with $f_3=3,8,13,18$ respectively. For each possibility we compute $y=2592\mod 3240$ (using Lemma~\ref{lem:div}), and so $y=2592$. The cases $(f_2,f_3)=(20,13),(24,18)$ give $X_4(p_0)<0$. Thus the only surviving cases are $(f_1,f_2,f_3)=(4,12,3),(4,16,8)$ with $x=0$ and $y=2592$. We now eliminate these final cases. Let $ L_0'$ be a fixed line in $\Gamma_3(p_0)$, and let $u,v,w$ be as in~(\ref{eq:uvw}). Thus $0\leq v\leq 5$. In the case $(f_1,f_2,f_3)=(4,12,3)$ we compute $X_4( L_0)=-6(156+93u-9v+w)<0$, and in the case $(f_1,f_2,f_3)=(4,16,8)$ we compute $X_4( L_0)=-6(150+93u-9v+w)<0$.

Now suppose that $x_0= L_0$ is a line. We may assume that the diameter of $T$ is precisely~$6$ (for otherwise the above case applies by shifting the root of the tree). Thus $ L_0$ is the unique centre of $T$, and we have $f_1,f_2,f_3\geq 2$. Suppose first that $f_1\neq 7$. Suppose further that no point in $\Gamma_1( L_0)\cap T$ is ideal, and that no line in $\Gamma_2( L_0)\cap T$ is full (we will call a point \textit{ideal} if all lines through it are fixed, and we will call a line \textit{full} if all points on it are fixed). Let $p_0'$ be a fixed point at distance $3$ from $ L_0$, and let $ L_0\bI p_1\bI  L_2\bI p_0'$ be the geodesic from $ L_0$ to $p_0'$. Let
\begin{align}\label{eq:uvw2}
\begin{aligned}
\mathcal{U}&=\{p \in T\mid p\bI  L_2\textrm{ and }p\neq p_0',p_1\}\\
\mathcal{V}&=\{ L\in T\mid  L\bI p_1\textrm{ and } L\neq  L_0, L_2\}\\
\mathcal{W}&=\{p\in T\mid p\bI  L\textrm{ for some $ L\in \mathcal{V}$, and }p\neq p_1\}.
\end{aligned}
\end{align}
Let $u=|\mathcal{U}|$, $v=|\mathcal{V}|$, and $w=|\mathcal{W}|$. Furthermore, suppose that we have chosen $p_0'$ amongst the points in $\Gamma_3( L_0)\cap T$ such that $u$ is maximal. Therefore $f_3\leq (u+1)f_2$.
Then
$$
X_1(p_0')=(315-18f_1-18f_2+9f_3-189u+162v-27w-x)/3.
$$
Since no point in $\Gamma_1( L_0)\cap T$ is ideal we have $v=0$, and hence $w=0$, and also $f_2\leq f_1$. Since no line in $\Gamma_2( L_0)\cap T$ is full we have $f_3\leq 4 f_2$. It follows that
$
X_1(p_0')\leq(315-189u-x)/3.
$
Therefore $u=0$ or $u=1$. If $u=0$ then $f_3\leq f_2$, and if $u=1$ then $f_3\leq 2f_2$. Note also that $x\leq 315$. Furthermore, from the formula for $|\mathcal{L}_6^4|$ we have $y\leq 4775$.

Let $p\in \mathcal{P}_2^1$ be a point through $ L_0$ (such a point exists by assumption), and let $ L\in \mathcal{L}_2^1$ be a line through~$p_0'$ (such a line exists because $p_0'$ is an end point of the tree). We compute
\begin{align*}
X_1( L_0)&=(-1260+936f_1-1332f_2+288f_3-5x+3y)/12\\
X_3( L_0)&=(1260-936f_1+1332f_2-288f_3+5x+3y)/6\\
X_3'(p)&=(180-180f_1+144f_2-18f_3+5x+2y)/2\\
X_1'( L)&=(756+36f_1-72f_2-288u-10x-3y)/6
\end{align*}
(recall that $v=w=0$ and $u=0,1$ in the formula for $X_1'( L)$). We have 
\begin{align}\label{eq:ycong}
y\equiv 72(f_2+1)+2340(f_1+f_3)+270x+828\mod 3240.
\end{align}
It is now routine to check on a computer that for $0\leq x\leq 315$ (with $x$ divisible by~$6$), $2\leq f_1\leq 5$, $2\leq f_2\leq f_1$, $2\leq f_3\leq 2f_2$, $y\leq 4775$ there are no solutions for $(f_1,f_2,f_3,x,y)$ satisfying (\ref{eq:ycong}) with $X_1( L_0),X_3( L_0),X_3'(p),X_1'( L)\geq 0$ (in this search, use $u=0$ in the range $2\leq f_3\leq f_2$, and $u=1$ in the range $f_2<f_3\leq 2f_2$).

Thus either $f_1=7$ (and so the central line is full), or at least one point in $\Gamma_1( L_0)\cap T$ is ideal, or at least one line in $\Gamma_2( L_0)\cap T$ is full. Suppose that $f_1=7$, and so $x=0$. From Corollary~\ref{cor:y} to compute $y=12(-147+37f_2-8f_3)$. We also have $y\equiv 1080+72f_2+2340f_3\mod 3240$, and thus $237+239f_2+203f_3\equiv 0\mod 270$. It is elementary that the only values of $f_2,f_3$ with $1\leq f_2\leq 3f_1$ and $1\leq f_3\leq 6f_2$ that satisfy this congruence, and give $y\geq 0$, are $(f_2,f_3)=(10,11),(17,40),(21,18)$. The second and third cases give $X_4( L_0)<0$, and so the only surviving case is $(f_1,f_2,f_3)=(7,10,11)$ with $x=0$ and $y=1620$. We now eliminate this case. Let $p_0$ be a point in $\Gamma_3(p_0)\cap T$, and let $u,v,w$ be as in the previous arguments. Since $f_3>f_2$ we may choose $p_0$ so that $u\geq 1$. Since $v\leq 2$ we have $X_4(p_0)=18(4-25u+6v-w)\leq -18(9+w)<0$.

Now suppose that there is an ideal point $p_1\in\Gamma_1( L_0)\cap T$. Then $x=0$ and $f_2\geq 3$. 
We may assume that $ L_0$ is not full (for otherwise the above argument applies). Let $\mathcal{Z}$ be the set of points in $\Gamma_3( L_0)\cap T$ for which the geodesic from $p$ to $ L_0$ passes through~$p_1$, and let $z=|\mathcal{Z}|$. Relative to the root $p_1$ we have $a_1=4$, $a_2=f_1+z-1$, $a_3=f_2-3$ and $a_4=f_3-z$. Then by Corollary~\ref{cor:y} we deduce that $z=(82-10f_1+8f_2-f_3)/9$, and hence $f_1+f_2+f_3\equiv 1\mod 9$. Furthermore we have $z>0$, 
for if not then $f_3=82-10f_1+8f_2$ and we calculate
$
X_2( L_0)=(1620-648f_1-324f_2-y)/4,
$
a contradiction since $f_1,f_2\geq 2$.

Thus, let $p_0'\in\mathcal{Z}$, and let $ L\in \mathcal{L}_2^1$ be a line through $p_0'$. Then we compute
$$
X_1'( L)=(318+12f_1-24f_2+6f_3-102u+6w-y)/2,
$$
where $u,v$ and $w$ are as in~(\ref{eq:uvw2}). Then $v=2$ by hypothesis (since $p_1$ is ideal). Thus, since $u\geq 0$ and $w\leq 6v=12$, we have
$$
X_1'( L)\leq(330+12f_1-24f_2+6f_3-y)/2.
$$
We compute
\begin{align*}
X_1( L_0)&=(-420+312f_1-444f_2+96f_3+y)/4\\
X_3( L_0)&=(420-312f_1+444f_2-96f_3+y)/2\\
X_4( L_0)&=(-420+2256f_1-1092f_2+96f_3-y)/2.
\end{align*}
A search shows that the only solutions $(f_1,f_2,f_3,y)$ with $2\leq f_1\leq 5$, $3\leq f_2\leq 3f_1$, $2\leq f_3\leq 6f_2$ with $f_1+f_2+f_3\equiv 1\mod 9$, and (\ref{eq:ycong}) (with $x=0$) which satisfy $X_1( L_0),X_3( L_0),X_4( L_0)\geq 0$ and $330+12f_1-24f_2+6f_3-y\geq 0$ are
$
(f_1,f_2,f_3,y)=(2,5,21,0)$, $(3,7,27,324)$, $(4,7,26,324)$ and $(5,12,47,324)$. The case $(f_1,f_2,f_3)=(2,5,21)$ is impossible because $f_2=3f_1-1$ is impossible. Consider the case $(f_1,f_2,f_3)=(3,7,27)$. Let the three points in $\Gamma_1( L_0)\cap T$ be $p_1,p_2$ and $p_3$. Then $p_1$ is ideal and (after renaming if necessary) $p_2$ is also ideal, and $p_3$ has exactly $7-2\times 3=1$ sibling $ L\in\Gamma_2( L_0)\cap T$. We have $z=9$, and so there are exactly $9$ points $p\in \Gamma_3( L_0)\cap T$ for which the geodesic from $p$ to $ L_0$ passes through $p_1$. Similarly, since $p_2$ is also ideal, there are exactly $9$ points $p\in\Gamma_3( L_0)\cap T$ for which the geodesic from $p$ to $ L_0$ passes through $p_2$. Thus there are exactly $27-2\times 9=9$ points $p\in\Gamma_3( L_0)\cap T$ for which the geodesic from $p$ to $ L_0$ passes through $p_3$. But this is impossible, because $p_3$ has exactly $1$ sibling $ L\in \Gamma_2( L_0)\cap T$, and then $ L$ has at most $3$ siblings in $\Gamma_3( L_0)\cap T$. The remaining 2 cases are impossible for analogous reasons. 

Finally, suppose that there is a full line $ L_0'\in \Gamma_2( L_0)\cap T$. Then $x=0$. Let $ L_0\bI p_1\bI L_0'$ be the geodesic from $ L_0$ to $ L_0'$. Let
\begin{align*}
\mathcal{V}&=\{ L\in T\mid  L\bI p_1\textrm{ and } L\neq  L_0, L_0'\}\\
\mathcal{W}&=\{p\in T\mid p\bI  L\textrm{ for some $ L\in \mathcal{V}$, and }p\neq p_1\}.
\end{align*}
Let $v=|\mathcal{V}|$ and $w=|\mathcal{W}|$. Note that $v=0$ or $v=2$, and if $v=0$ then $w=0$. Relative to the root $ L_0'$ we have $a_1=7$, $a_2=v+1$, $a_3=f_1-1+w$, $a_4=f_2-1-v$, and $a_5=f_3-6-w$. Then by Corollary~\ref{cor:y} we have 
$
y=12(-109-8f_1+f_2+f_3+36v-9w).
$
From this formula it follows that $v\neq 0$. For if $v=0$ then $w=0$, and so $f_2+f_3\geq 109$. But if $v=0$ then $f_2\leq f_1$ and so $f_3\leq 6f_2\leq 6f_1$, and so $f_2+f_3\leq 7f_1\leq 7\times 6=42$. Thus $v=2$, and so there is an ideal point in $\Gamma_1(p_0)\cap T$, and thus this case is eliminated by the previous argument.
\end{proof}

\section{Exceptional domestic collineations of Moufang hexagons}\label{sect:8}

In this section we prove Theorem~\ref{thm:5}. 

\begin{lemma}\label{lem:pldom}
Let $\theta$ be a domestic collineation of a generalised hexagon. If $\cP_4^2=\cP_4^3=\emptyset$ then $\theta$ is either point-domestic or line-domestic.
\end{lemma}

\begin{proof}
Suppose that $\cP_4^3=\cP_4^2=\emptyset$. It follows that $\cL_4^3$ and $\cL_4^2$ are also empty, and so up to duality we may assume that there is a point $p\in\cP_6$. Suppose that there is also a line $L\in\cL_6$. Then by domesticity $d(p,L)\neq 1$, and so $d(p,L)=3$ or $d(p,L)=5$. Suppose that $d(p,L)=3$, and let $p\bI M\bI q\bI L$ be the geodesic joining $p$ to $L$. Then $M\in\cL_4^1$ (by domesticity and the assumption that $\cP_4^2$ and $\cP_4^3$ are empty). Then $q\in\cP_2^1\cup\cP_6$, and so $L\in\cL_0\cup\cL_4^1$, a contradiction. Similarly, if $d(p,L)=5$ and if $p\bI M\bI q\bI N\bI r\bI L$ is the geodesic joining $p$ to $L$ then $N\in \cL_0\cup \cL_4^1$ and so $r\in\cP_0\cup\cP_2^1\cup\cP_6$, and thus $L\in\cL_0\cup\cL_2^1\cup\cL_4^1$, a contradiction.
\end{proof}

\begin{proof}[Proof of Theorem~\ref{thm:5}]
Let $\theta$ be an exceptional domestic collineation of a Moufang hexagon $\Gamma$ with both parameters infinite. By \cite[Theorem~2]{ron:80} we can choose the point set in such a way that, with the terminology of \cite[\S1.9]{mal:98}, all points are regular points. This means that given two opposite points $p,q$, the set of points collinear to $p$ and at distance 4 from $q$ is determined by any pair of its points. We will call such a set of points a \emph{hyperbolic line}, and we denote it by~$p^q$. 

The key observation for the proof is as follows: Suppose that the point $p$ is mapped onto an opposite point. The set of points collinear to $p$ and not mapped onto an opposite point is $p^{p^\theta}\cup ((p^\theta)^p)^{\theta^{-1}}$. Hence this set consists of the union of at least one and at most two hyperbolic lines. Consequently, if three points collinear with $p$ and contained in a hyperbolic line $H$ are not mapped onto opposite points, then no point of the hyperbolic line $H$ is mapped onto an opposite point.

By Lemma~\ref{lem:pldom} we may assume that either $\cP_4^3\neq\emptyset$ or that $\cP_4^2\neq\emptyset$. If $\cP_4^3\neq\emptyset$, choose $q\in\cP_4^3$ and choose a point $p_2\in\cP_6$ collinear with~$q$ (with the line $qp_2$ in $\cL_4^3$). Then choose $p_1\in\cP_6$ collinear with~$p_2$ and at distance~$4$ from~$q$, and choose $p\in\cP_6$ collinear with~$p_1$ and at distance~$6$ from~$q$. If $\cP_4^3=\emptyset$ then instead we choose $q\in\cP_4^2$, and then choose $p_2\in\cP_2^2$ collinear with $q$ (with the line $qp_2\in\cL_2^2$). Then choose $p_1\in\cP_6$ collinear with $p_2$ and at distance~$4$ from~$q$, and finally choose $p\in\cP_6$ collinear with $p_2$ and at distance~$6$ from~$q$. Note that in both cases we have $p,p_1\in\cP_6$ and~$q\in\cP_4$.

Since at most two lines through $q$ belong to $\mathcal{L}_4$ we can find a line $L\in \mathcal{L}_6$ through~$q$. Let $q_1$ be the unique point of $L$ not opposite $p$. By domesticity all points on $L$ belong to $\mathcal{P}_4$ and so in particular $q_1\in\mathcal{P}_4$. At most two lines through $q_1$ do not belong to $\mathcal{L}_6$, hence there are an infinite number of lines through $q_1$ in $\mathcal{L}_6\setminus\{L\}$. Let $\cR$ be the set of points on these lines belonging to~$q_1^{p_1}$.

Let $q_2$ be the unique point collinear to both $p$ and $q_1$. Since $p_1\in p^q\cap \mathcal{P}_6$ there are at most $2$ points in $p^q\setminus\{q_2\}$ that are not in~$\mathcal{P}_6$. Let $\mathcal{S}$ be the set of points in $p^q$ belonging to $\mathcal{P}_6$ (and so $\mathcal{S}$ is an infinite set). Since each of these points is at distance 4 from both $q_2$ and $q$, and these two points belong to $q_1^{p_1}$, every point in $\cR$ is at distance $4$ from every point of $\cS$. Given $y\in\cS$, let $\cA_{y}$ denote the set of lines at distance $3$ from $y$ and incident with a member of $\cR$.  Through each element of $\cR$ there are at most two lines that do not belong to $\mathcal{L}_6$. We may choose $y\in\cS$ such that $\cA_y$ contains at least~$3$ elements of $\cL_6$ (because for two different points $y_1,y_2$ of $\cS$ and one  point $x\in\cR$, the lines through $x$ at distance $3$ from $y_1,y_2$, respectively, differ). Thus there are three points $x_1,x_2,x_3\in\cR$ such that for $i=1,2,3$ the line $L_i$ through $x_i$ at distance 3 from $y$ belongs to $\mathcal{L}_6$. Let $z_i$ be the point on $L_i$ collinear to $y$, $i=1,2,3$. Then $z_i\in \mathcal{P}_4$, as $L_i\in \mathcal{L}_6$. Since $z_i\in y^{q_1}$, $i=1,2,3$, we deduce $y^{q_1}\subseteq \mathcal{P}_4$ (by the key observation above). But $p\in y^{q_1}\cap \mathcal{P}_6$, a contradiction
\end{proof}

\section{Anisotropic automorphisms}\label{sect:9}

An automorphism of a generalised polygon (and more generally of a twin building) is \textit{anisotropic} if it maps every chamber to an opposite chamber. There are plenty of examples of anisotropic automorphisms of infinite generalised polygons. For example, the duality of the real projective plane $\mathsf{PG}(2,\mathbb{R})$ given by $(a,b,c)\leftrightarrow [a,b,c]$ is anisotropic (here $(a,b,c)$ is a point of the projective plane in homogeneous coordinates, and $[a,b,c]$ is the line of the projective plane corresponding to the plane $ax+by+cz=0$ in $\mathbb{R}^3$). 

For another example, consider the split Cayley hexagon $\mathsf{H}(\mathbb{F})$. By the standard embedding the points of $\mathsf{H}(\mathbb{F})$ can be identified with the points of the parabolic quadric $\mathsf{Q}(6,\mathbb{F})$ in $\mathsf{PG}(6,\mathbb{F})$ with equation $X_0X_4+X_1X_5+X_2X_6=X_3^2$, and the lines of $\mathsf{H}(\mathbb{F})$ can be identified with the lines of $\mathsf{Q}(6,\mathbb{F})$ whose Grassmann coordinates satisfy $6$ explicit linear equations (see \cite[\S2.4.13]{mal:98}). The map $(X_0,X_1,X_2,X_3,X_4,X_5,X_6)\mapsto(-X_4,-X_5,-X_6,X_3,-X_0,-X_1,-X_2)$ preserves the  quadric and the equations amongst the Grassmann coordinates, and thus induces an involutory collineation $\theta$ of the hexagon~$\mathsf{H}(\mathbb{F})$. We claim that if $a^2+b^2=-1$ has no solution in $\mathbb{F}$ (for example, if $\mathbb{F}=\mathbb{R}$) then $\theta$ is anisotropic. A point $p$ is opposite $q$ (in both the quadric and the hexagon) if and only if $\langle p,q\rangle\neq0$, where $\langle\cdot,\cdot\rangle$ is the inner product associated to the quadratic form associated to the quadric.
Thus if $p$ and $p^{\theta}$ are not opposite, we have $-2X_0^2-2X_1^2-2X_2^2=2X_3^2$, a contradiction. Thus $\theta$ maps every point to an opposite point, and it follows that it maps every chamber to an opposite chamber.

The finite case is in stark contrast to this situation. In \cite[Remark~4.5]{TTM:13} it is shown that there exist quadrangles of order $(2^n-1,2^n+1)$ which admit anisotropic automorphisms, and so far these are the \textbf{only} known examples for finite polygons. In this section we show that no finite thick generalised hexagon or octagon admits an anisotropic automorphism (thus proving Theorem~\ref{thm:6}). In fact we prove quite a bit more. But first an elementary lemma.

\begin{lemma}\label{lem:fibonacci} Let $a,b\geq 1$ be integers. 
\begin{itemize}
\item[\emph{1.}] Suppose that $a$ divides $b^2+1$ and that $b$ divides $a^2+1$. If $a\leq b$ then $(a,b)=(F_{2n-1},F_{2n+1})$ for some $n\geq 1$, where $(F_k)_{k\geq 0}$ is the sequence of Fibonacci numbers.
\item[\emph{2.}] Suppose that $2a$ divides $b^2+1$ and that $b$ divides $2a^2+1$. If $b<\sqrt{2}a$ then $(a,b)=(P_{2n+1},P_{2n}+P_{2n-1})$ for some $n\geq 1$, and if $b>\sqrt{2}a$ then $(a,b)=(P_{2n-1},P_{2n}+P_{2n-1})$ for some $n\geq 1$, where $(P_k)_{k\geq 0}$ is the sequence of Pell numbers, with $P_0=0$, $P_1=1$ and $P_{k+2}=2P_{k+1}+P_k$. 
\end{itemize}
\end{lemma}

\begin{proof}
1. To prove the first statement, let $a_0=a$ and $b_0=b$, and for $k\geq 1$ recursively define $a_k$ and $b_k$ by $a_k=(a_{k-1}^2+1)/b_{k-1}$ and $b_k=a_{k-1}$. Then $a_k$ divides $b_k^2+1$, and so $a_k$ and $b_k$ are coprime. It follows that $b_k$ divides $a_k^2+1$, because
$$
a_k^2+1=\frac{a_{k-1}^2(a_{k-1}^2+2)+(b_{k-1}^2+1)}{b_{k-1}^2}.
$$
Furthermore, by induction $a_k\leq b_k$ (with strict inequality if $a_k>1$). Therefore there is an index $N$ for which $a_N=1$, and then $b_N=1$ or $b_N=2$ (because $b_N$ divides $a_N^2+1=2$). If $b_N=2$ then we can apply the algorithm one more time, and so we may assume that $(a_{N},b_{N})=(1,1)$. We can now work backwards to recover $(a,b)$, and it follows (with the help of Cassini's identity) that $a=F_{2n-1}$ and $b=F_{2n+1}$ for some~$n$.

2. We first make the following observation. Suppose that $2a$ divides $b^2+1$ and 
that $b$ divides $2a^2+1$. We construct a new pair $(a',b')$ according to the rules: (1) If $b<\sqrt{2}a$ let $(a',b')=((b^2+1)/(2a),b)$, and (2) if $b>\sqrt{2}a$ let $(a',b')=(a,(2a^2+1)/b)$. It is simple to check that in both cases the new pair $(a',b')$ satisfies $2a'$ divides $b'^2+1$ and $b'$ divides $2a'^2+1$. We claim that if $(a,b)$ satisfies $b<\sqrt{2}a$ then $(a',b')$ satisfies $b'>\sqrt{2}a'$, and that if $(a,b)$ satisfies $b>\sqrt{2}a$ then $(a',b')$ satisfies $b'<\sqrt{2}a'$ (and so inductively continuing the process we alternate between cases (1) and (2)). The claim is easily verified for small values of $a$ and $b$, and so we may assume that $a> 1$ and $b> 4$. For example, suppose that $b<\sqrt{2}a$ and that also $b'^2<\sqrt{2}a'$. The inequality $b'^2<\sqrt{2}a'$ gives $b^2-\sqrt{2}ab+1>0$. Since $a>1$, the smaller of the two roots of this quadratic expression is smaller than $1$, and so we see that $b$ must be larger than the larger of the two roots. Thus $b>(a/\sqrt{2})(1+\sqrt{1-2a^2})$. Combining this with the inequality $b<\sqrt{2}a$ gives
$$
\bigg|\sqrt{2}-\frac{b}{a}\bigg|<\frac{1}{\sqrt{2}}\left(1-\sqrt{1-\frac{2}{a^2}}\right)<\frac{\sqrt{2}}{a^2}
$$
(where we use $\sqrt{1-x^2}>1-x^2$). However since $b>4$ we have $2a^2-b^2>4$ (because $2a^2-b^2=0$ is clearly impossible, and if $2a^2-b^2=1$ then since $b$ divides $2a^2+1$ we have that $b$ divides $b^2+2$, and so $b=1$ or $b=2$, and if $2a^2-b^2=2$ then $b=1$ or $b=3$, and similarly for $2a^2-b^2=3$). Thus 
$$
\bigg|\sqrt{2}-\frac{b}{a}\bigg|=\frac{|2a^2-b^2|}{a^2|\sqrt{2}+b/a|}>\frac{4}{a^2(\sqrt{2}+\sqrt{2})}=\frac{\sqrt{2}}{a^2},
$$
a contradiction. The argument is similar if we suppose that both $b>\sqrt{2}a$ and $b'>\sqrt{2}a'$, completing the proof of the claim.

We now inductively construct a sequence $(a_{k},b_{k})$ by setting $a_0=a$, $b_0=b$, and 
$$
(a_{k+1},b_{k+1})=\begin{cases}\left(\frac{b_k^2+1}{2a_k},b_k\right)&\textrm{if $b_k<\sqrt{2}a_k$}\\
\left(a_k,\frac{2a_k^2+1}{b_k}\right)&\textrm{if $b_k>\sqrt{2}a_k$.}
\end{cases}
$$
From the above observations we see that we alternate between the two cases, and moreover the ``size'' of the pair decreases at each step, in the sense that $\sqrt{2}a_{k+1}+b_{k+1}\leq\sqrt{2}a_k+b_k$ (with strict inequality if $a\geq 2$). Thus the process finally terminates at $(a_N,b_N)=(1,1)$. The lemma now easily follows by working backwards to recover $(a,b)$. 
\end{proof}

The following theorem implies Theorem~\ref{thm:6}, and is the main theorem of this section.

\begin{thm}\label{thm:rest} No duality of a thick generalised $2n$-gon is anisotropic. Moreover, if $\Gamma$ is a finite thick generalised $2n$-gon with parameters~$(s,t)$ then:
\begin{itemize}
\item[\emph{1.}] If $n=2$ then anisotropic collineations can only exist if $s$ and $t$ are coprime. 
\item[\emph{2.}] If $n=3$ then anisotropic collineations do not exist. Moreover, if a collineation maps all lines (or all points) to distance $4$ and $6$ only then $s$ and $t$ are coprime.\item[\emph{3.}] If $n=4$ then anisotropic collineations do not exist. Moreover, no collineation maps all lines (or all points) to distance $6$ and $8$ only, and if a collineation maps all lines (or all points) to distance $4,6$ and $8$ only then $s$ and $t$ are coprime.
\end{itemize}
\end{thm}

\begin{proof} It follows from \cite[Theorem~1.3]{DPM:13} that dualities of $2n$-gons are never anisotropic.

1. This is proved in \cite[Corollary~4.3]{TTM:09} (see also \cite{ben:70}).

2. Suppose that $\theta$ is a collineation of a finite thick hexagon such that all lines are mapped to distance $4$ or $6$. Thus $|\cL_0|=|\cL_2|=0$, and it follows from the formula for $n_3-n_2$ in~(\ref{eq:nhex1}) that $s$ and $t$ are coprime (see also \cite[Corollary~5.2]{TTM:09}). Suppose further that $|\cL_4|=0$ (and so $\theta$ is anisotropic). Since $\sqrt{st}$ is an integer it follows that $s$ and $t$ are perfect squares. Then, again from the formula for $n_3-n_2$, we see that $\sqrt{s}$ divides $t+1$, and $\sqrt{t}$ divides $s+1$. We may suppose that $s\leq t$. Then by Lemma~\ref{lem:fibonacci} we see that $s=F_{2n-1}^2$ and $t=F_{2n+1}^2$ for some~$n$. Using $F_{2n+1}-F_{2n-1}=F_{2n}$ and Cassini's identity $F_{2n-1}F_{2n+1}=F_{2n}^2+1$ we see that the denominator of $n_1$ is
$$
s^2+st+t^2=((\sqrt{s}-\sqrt{t})^2+3\sqrt{st})((\sqrt{s}-\sqrt{t})^2+\sqrt{st})=(4F_{2n}^2+3)(2F_{2n}^2+1),
$$
and similarly the numerator of $n_1$ is 
$
s^2t^2+st+1=(F_{2n}^4+3F_{2n}^2+3)(F_{2n}^4+F_{2n}^2+1).
$
Thus $n_1$ is a rational function in $F_{2n}$, and after polynomial division we see that
$$
\frac{21(6F_{2n}^2+11)}{(4F_{2n}^2+3)(2F_{2n}^2+1)}
$$
must be an integer. But this only occurs when $n=1$, and in this case $\sqrt{s}=F_1=1$, contradicting thickness.

3. Suppose that $\theta$ is a collineation of an octagon mapping all lines to distance $4,6$ or $8$, and so $|\cL_0|=|\cL_2|=0$. If follows from (\ref{eq:noct1}) that $s$ and $t$ are coprime (see also \cite[Lemma~6.2]{TTM:09}). Suppose now that $\theta$ is anisotropic, and so $|\cL_0|=|\cL_2|=|\cL_4|=|\cL_6|=0$. Then the formula for $n_2$ implies that $s$ divides $t+1$, and that $t$ divides $s+1$, a contradiction. 

Suppose now that only $|\cL_0|=|\cL_2|=|\cL_4|=0$. Since $\sqrt{2st}\in\mathbb{Z}$ and $s$ and $t$ are coprime we have (up to duality) $s=2a^2$ and $t=b^2$ for some integers $a,b>1$. The formula for $n_2$ in (\ref{eq:noct1}) implies that $4a^2$ divides $b^2+1-|\cL_6|$, and the formula for $n_3-n_4$ implies that $2a$ divides $1-b^4-|\cL_6|$. Thus $2a$ divides $b^2(b^2+1)$, and so $2a$ divides $b^2+1$ (since $s$ and $t$ are coprime). Similarly $b$ divides $2a^2+1$. Thus by Lemma~\ref{lem:fibonacci} we have $(a,b)=(P_{2n+1},P_{2n}+P_{2n-1})$ (in the case that $b<\sqrt{2}a$) or $(a,b)=(P_{2n-1},P_{2n}+P_{2n-1})$ (in the case that $b>\sqrt{2}a$) for some $n\geq 1$, where $(P_k)_{k\geq 0}$ is the sequence of Pell numbers. Now, since $s=2a^2$ and $t=b^2$ are coprime, the divisibility condition from Section~\ref{sect:2} implies that the number $$N=\frac{(2a^2b^2+1)(4a^4b^4+1)}{(2a^2+b^2)(4a^4+b^4)}$$ is a positive integer. It is now possible to see directly that if $(a,b)=(P_{2n+1},P_{2n}+P_{2n-1})$ or $(a,b)=(P_{2n-1},P_{2n}+P_{2n-1})$ then this divisibility condition is violated. For example, consider the case $(a,b)=(P_{2n+1},P_{2n}+P_{2n-1})$. Using the defining formula and the Cassini identity $P_{n+1}P_{n-1}-P_n^2=(-1)^n$ for Pell numbers we see that $a=P_{2n}+\sqrt{2P_{2n}^2+1}$ and $b=\sqrt{2P_{2n}^2+1}$. Thus we can write $N$ as $N=(p_1(z)+p_2(z)\sqrt{2z^2+1})/q(z)$ where $p_1,p_2,q$ are polynomials, and $z=P_{2n}$. After performing polynomial division for $p_1(z)/q(z)$ and $p_2(z)/q(z)$ one sees (with some calculation) that $N$ is not an integer, a contradiction. 
\end{proof}

We provide the following table to summarise the result of Theorem~\ref{thm:rest}, where \textit{codistance} in a generalised $n$-gon is defined by $\mathrm{codist}(x,y)=n-d(x,y)$. 
$$
\begin{array}{|c||c|c|c|}
\hline
\textrm{All lines mapped to codistance:}&0\textrm{ only}&0\textrm{ and }2&0,2\textrm{ and }4\\
\hline\hline
\textrm{quadrangles}&\gcd(s,t)=1&\textrm{possible}&\textrm{possible}\\
\hline
\textrm{hexagons}&\textrm{impossible}&\gcd(s,t)=1&\textrm{possible}\\
\hline
\textrm{octagons}&\textrm{impossible}&\textrm{impossible}&\gcd(s,t)=1\\
\hline
\end{array}
$$
This can be reworded as follows: No automorphism of a finite generalised $2m$-gon with parameters $(s,t)$ maps every line (or every point) to at least distance~$6$, and if some automorphism maps every line (or every point) to distance at least~$4$, then $s$ and $t$ are relatively prime.

As mentioned above, examples are known of anisotropic collineations in some generalised quadrangles with parameters $(s,t)=(2^n-1,2^n+1)$. No finite thick hexagons or octagons with coprime parameters are known, and we conjecture that none exist. As a consequence, we conjecture that it is impossible for a collineation of a finite thick hexagon or octagon to map all lines (or all points) to distance $4$ or more.

\begin{appendix}

\section{Eigenvalue techniques}\label{app:A}

In this appendix we describe the eigenvalue techniques that are used at various points of this paper, extending the exposition in \cite{ben:70,tem:10} (see also \cite{KS:73}). It is convenient to set up the theory in the general context of distance regular graphs (in our case the line or point graph of a generalised $n$-gon).

Let $G$ be a finite connected graph with vertex set~$V$, and let $d(x,y)$ denote the graph distance between vertices $x$ and~$y$. Let $N$ be the diameter of $G$, and for natural numbers $0\leq k\leq N$ let $S_k(x)=\{y\in V\mid d(x,y)=k\}$ be the sphere of radius~$k$ centred at~$x$. We assume that $G$ is \textit{distance regular}, meaning that for each $0\leq k, \ell,m\leq N$ there is an integer $a_{k, \ell}^m\geq0$ such that
$$
a_{k, \ell}^m=|S_k(x)\cap S_{ \ell}(y)|\qquad\textrm{whenever $d(x,y)=m$.}
$$
That is, the cardinality of the intersection $S_k(x)\cap S_{\ell}(y)$ depends only on $k, \ell$, and $m=d(x,y)$.

Let $A_k$ be the $|V|\times |V|$ matrix whose $(x,y)$-entry is $1$ if $d(x,y)=k$ and $0$ otherwise. Then $A_kA_{ \ell}=\sum_m a_{k, \ell}^mA_m$ for all $0\leq k, \ell\leq N$, and thus the span over $\mathbb{C}$ of $\{A_k\mid 0\leq k\leq N\}$ is an associative algebra~$\mathcal{A}$ with identity~$A_0$ (called the \textit{Bose-Mesner algebra}). Since $a_{k, \ell}^m=a_{ \ell,k}^m$ the algebra $\mathcal{A}$ is commutative, and from the formula
\begin{align}\label{eq:recurrence}
A_1A_k=a_{1,k}^{k-1}A_{k-1}+a_{1,k}^kA_k+a_{1,k}^{k+1}A_{k+1}\qquad\textrm{for $1\leq k\leq N-1$}
\end{align}
we see that $\mathcal{A}$ is generated by $A_1$ (because by induction each $A_k$ is a polynomial in $A_1$).

Thus every irreducible representation of $\mathcal{A}$ is~$1$-dimensional. These representations are the (nontrivial) algebra homomorphisms $\chi:\mathcal{A}\to \mathbb{C}$, and each homomorphism is determined by its value on~$A_1$. Equivalently, the values of the $1$-dimensional representations at $A_1$ are the eigenvalues of the matrix~$A_1$. Since $A_1$ is symmetric, each $\chi(A_1)$ is real, and by Perron-Frobenius there is a unique largest  eigenvalue, and this eigenvalue has multiplicity~$1$. 

Let $\theta:G\to G$ be an automorphism of~$G$, and let $\Theta$ be the $|V|\times |V|$ permutation matrix of~$\theta$. Thus $(\Theta)_{x,y}$ equals $1$ if $y=\theta(x)$ and $0$ otherwise. The eigenvalues of $\Theta$ are $n$-th roots of unity, where $n=\mathrm{ord}(\theta)$. A calculation shows that $\Theta$ commutes with $A_1$, and thus $\Theta$ commutes with each element of the algebra~$\mathcal{A}$. Therefore the unitary matrix $\Theta$ stabilises each eigenspace of the symmetric matrix $A_1$, and thus there is a unitary matrix~$P$ simultaneously diagonalising $\Theta$ and each $A_k$. Write $\Theta=PTP^*$ and $A_k=PD_kP^*$. 

For each $0\leq k\leq N$ let $d_k=|\{x\in V\mid d(x,\theta(x))=k\}|$. Then
$$
d_k=\tr(\Theta A_k)=\tr(PTD_kP^*)=\tr(TD_k)=\sum_{j=0}^Mn_j\chi_j(A_k)
$$
for some numbers $n_j$ (not depending on~$k$), where $\chi_0,\ldots,\chi_M$ is the complete list of nontrivial algebra homomorphisms $\chi_j:\mathcal{A}\to\mathbb{C}$ (necessarily $M\leq N$).

The key point is that if each $\chi_j(A_k)$ (with $0\leq j,k\leq N$) is an integer, then the numbers $n_j$ are necessarily integers as well. (For if $\zeta$ and $\zeta'$ are primitive $d$th roots of unity, and $\lambda$ is an eigenvalue of $A_1$, then if $\zeta\lambda$ is an eigenvalue of $\Theta A_1$ if and only if $\zeta'\lambda$ is an eigenvalue of $\Theta A_1$. Moreover, since the coefficients of the characteristic polynomial of $\Theta A_1$ are integers, and the minimal polynomials of $\zeta\lambda$ and $\zeta'\lambda$ coincide, the eigenvalues $\zeta\lambda$ and $\zeta'\lambda$ have the same multiplicity; see \cite[Lemma~3.1]{TTM:09}).

Arrange the algebra homomorphisms so that  $\chi_0(A_1)$ is the largest eigenvalue of~$A_1$. Since this eigenvalue has multiplicity~$1$ we have $n_0=1$, and so we have formulae:
\begin{align}\label{eq:neqns}
d_k-\chi_0(A_k)=\sum_{j=1}^{M}n_j\chi_j(A_k)\qquad\textrm{for $0\leq k\leq N$}.
\end{align}
This (overdetermined) system can be solved to give formulae for the $n_j$ in terms of the $d_k$. Since the $n_j$ are integers, the formulae that one obtains typically put very severe constraints on the numbers $d_k$. This is what we refer to as `the eigenvalue technique'.

Let us now specialise to the line graph of a finite generalised $2n$-gon $\Gamma=(\mathcal{P},\mathcal{L},\mathbf{I})$. Thus $G$ is the graph with vertex set $\mathcal{L}$, with $ L$ connected to $ L'$ by an edge if and only if there is a point $p\in\mathcal{P}$ with $ L\bI p\bI  L'$. Thus $d_k=|\cL_{2k}|$. In this paper we only require the hexagon and octagon cases, but we include the quadrangle case for completeness.  

In the quadrangle case $G$ has diameter~$2$, and the matrices $A_k$ satisfy
\begin{align*}
A_1A_1&=t(s+1)A_0+(t-1)A_1+(s+1)A_2\\
A_1A_2&=stA_1+(t-1)(s+1)A_2.
\end{align*} 
The algebra homomorphisms $\chi:\mathcal{A}\to\mathbb{C}$ are easily computed from these formulae. There are 3 nontrivial homomorphisms $\chi_0,\chi_1,\chi_2$ given by
$$
\begin{array}{|c||c|c|c|}
\hline
&\chi_0&\chi_1&\chi_2\\
\hline\hline
A_1&(s+1)t&-(s+1)&t-1\\
\hline
A_2&st^2&s&-t\\
\hline
\end{array}
$$
and solving the equations~(\ref{eq:neqns}) gives
\begin{align*}
n_1&=\frac{st+1+(t-1)|\cL_0|-|\cL_2|}{s+t}&n_2&=-\frac{st+s+t+1-(s+1)|\cL_0|-|\cL_2|}{s+t}.
\end{align*}
Since each $\chi_j(A_k)$ is an integer, the above theory forces $n_1$ and $n_2$ to be integers as well.

In the hexagon case $G$ has diameter~$3$, and the matrices $A_k$ satisfy
\begin{align*}
A_1A_1&=t(s+1)A_0+(t-1)A_1+A_2\\
A_1A_2&=stA_1+(t-1)A_2+(s+1)A_3\\
A_1A_3&=stA_2+(t-1)(s+1)A_3.
\end{align*}
There are $4$ nontrivial algebra homomorphisms $\chi_j:\mathcal{A}\to\mathbb{C}$ given by (with $r=\sqrt{st}\in\mathbb{Z}$)
$$
\begin{array}{|c||c|c|c|c|}
\hline
&\chi_0&\chi_1&\chi_2&\chi_3\\
\hline\hline
A_1&(s+1)t&-(s+1)&t+r-1&t-r-1\\
\hline
A_2&(s+1)st^2&s(s+1)&(t-1)r-t&-(t-1)r-t\\
\hline
A_3&s^2t^3&-s^2&-tr&tr\\
\hline
\end{array}
$$
and solving the equations~(\ref{eq:neqns}) we see that $n_1,n_2,n_3$ satisfy
\begin{align}
\label{eq:nhex1}\begin{aligned}
n_1&=\frac{(t^2-t+1)|\cL_0|-(t-1)|\cL_2|+|\cL_4|-(s^2t^2+st+1)}{s^2+st+t^2}\\
n_3+n_2&=\frac{(s^2-1)(t^2-1)+(s+1)(s+t-1)|\cL_0|+(t-1)|\cL_2|-|\cL_4|}{s^2+st+t^2}\\
n_3-n_2&=\frac{p(s,t)+(s^2t-s^2-s-t)|\cL_0|-(s^2+s+t)|\cL_2|-(s+t)|\cL_4|}{\sqrt{st}(s^2+st+t^2)},
\end{aligned}
\end{align}
where $p(s,t)=(s+1)(t+1)(s+t)(st+1)-(s+1)(t+1)st$. Again, since each $\chi_j(A_k)$ is an integer (because $\sqrt{st}$ is an integer) the theory implies that $n_1,n_2$ and $n_3$ are integers as well.

In the octagon case $G$ has diameter~$4$, and the matrices $A_k$ satisfy 
\begin{align*}
A_1A_1&=t(s+1)A_0+(t-1)A_1+A_2\\
A_1A_2&=stA_1+(t-1)A_2+A_3\\
A_1A_3&=stA_2+(t-1)A_3+(s+1)A_4\\
A_1A_4&=stA_3+(t-1)(s+1)A_4.
\end{align*}
There are $5$ nontrivial algebra homomorphisms $\chi_j:\mathcal{A}\to\mathbb{C}$, given by (with $r=\sqrt{2st}\in\mathbb{Z}$)
$$
\begin{array}{|c||c|c|c|c|c|}
\hline
&\chi_0&\chi_1&\chi_2&\chi_3&\chi_4\\
\hline\hline
A_1&(s+1)t&-(s+1)&t-1&t+r-1&t-r-1\\
\hline
A_2&(s+1)st^2&(s+1)s&-(s+1)t&t(s-1)+(t-1)r&t(s-1)-(t-1)r\\
\hline
A_3&(s+1)s^2t^3&-(s+1)s^2&-(t-1)st&st(t-1)-tr&st(t-1)+tr\\
\hline
A_4&s^3t^4&s^3&st^2&-st^2&-st^2\\
\hline
\end{array}
$$
and solving the equations~(\ref{eq:neqns}) gives
\begin{align}\label{eq:noct1}
\begin{aligned}
n_1&=\frac{q_1(s,t)+(t^2+1)(t-1)|\cL_0|-(t^2-t+1)|\cL_2|+(t-1)|\cL_4|-|\cL_6|}{(s+t)(s^2+t^2)}\\
n_2&=\frac{q_2(s,t)+(s+1)(st-1)|\cL_0|+(st-s-1)|\cL_2|-(s+1)|\cL_4|-|\cL_6|}{2st(s+t)}\\
n_3+n_4&=\frac{-q_3(s,t)+q_5(s,t)|\cL_0|+q_6(s,t)|\cL_2|+(s+s^2+t-st)|\cL_4|+(s+t)|\cL_6|}{2st(s^2+t^2)}\\
n_3-n_4&=\frac{q_4(s,t)-(s^2-1)(t-1)|\cL_0|+(s^2+t-1)|\cL_2|+(t-1)|\cL_4|-|\cL_6|}{\sqrt{2st}(s^2+t^2)}
\end{aligned}
\end{align}
where the polynomials $q_j(s,t)$ are $q_1(s,t)=(st+1)(s^2t^2+1)$, $q_2(s,t)=(s+1)(t+1)(s^2t^2+1)$, $q_3(s,t)=(s+1)(t+1)(st+1)(s^2t+st^2+s+t-2st)$, $q_4(s,t)=(s^2-1)(t^2-1)(st+1)$, $q_5(s,t)=(s+1)(s+t+s^2t+st^2-2st)$ and $q_6(s,t)=(st^2-s^2t-st+s+t+s^2)$. Again, since each $\chi_j(A_k)$ is an integer (since $\sqrt{2st}$ is an integer), so too are $n_1,n_2,n_3$ and~$n_4$.

\end{appendix}


\bibliographystyle{plain}
\bibliography{domestic_automorphisms_AnnComb}

\medskip

\noindent\begin{minipage}{0.5\textwidth}
 James Parkinson\newline
School of Mathematics and Statistics\newline
University of Sydney\newline
Carslaw Building, F07\newline
NSW, 2006, Australia\newline
\texttt{jamesp@maths.usyd.edu.au}
\end{minipage}
\begin{minipage}{0.5\textwidth}
\noindent Beukje Temmermans\newline
Department of Mathematics\newline
Ghent University\newline
Krijgslaan 281, S22\newline
9000 Gent, Belgium\newline
\texttt{beukje@gmail.com}
\end{minipage}
\bigskip

\noindent Hendrik Van Maldeghem\newline
Department of Mathematics\newline
Ghent University\newline
Krijgslaan 281, S22\newline
9000 Gent, Belgium\newline
\texttt{hvm@cage.UGent.be}

\end{document}